\documentclass[11pt,a4paper]{article}
\usepackage[T1]{fontenc}
\usepackage[utf8]{inputenc}
\usepackage[english]{babel}
\usepackage{mathptmx}
\usepackage[a4paper,margin=1in]{geometry}
\usepackage{amsmath,amssymb,amsthm,mathtools,bm}
\usepackage{microtype}
\usepackage{enumitem}
\usepackage{array}
\usepackage[numbers,sort&compress]{natbib}
\usepackage{hyperref}
\hypersetup{hidelinks}
\allowdisplaybreaks
\numberwithin{equation}{section}

\newtheorem{theorem}{Theorem}[section]
\newtheorem{proposition}[theorem]{Proposition}
\newtheorem{lemma}[theorem]{Lemma}
\newtheorem{corollary}[theorem]{Corollary}
\theoremstyle{definition}
\newtheorem{definition}[theorem]{Definition}
\newtheorem{assumption}[theorem]{Assumption}
\theoremstyle{remark}
\newtheorem{remark}[theorem]{Remark}

\newcommand{\R}{\mathbb R}
\newcommand{\Ptwo}{\mathcal P_2}
\newcommand{\PP}{\mathcal P}

\newcommand{\cS}{\mathbb S}
\newcommand{\cH}{\mathbb H}
\newcommand{\cJ}{\mathbb J}
\newcommand{\Q}{\mathsf Q}
\newcommand{\cC}{\mathfrak S}
\newcommand{\Law}{\mathcal L}

\newcommand{\Var}{\operatorname{Var}}
\newcommand{\proj}{\operatorname{proj}}

\newcommand{\dd}{\,d}
\newcommand{\1}{\mathbf 1}

\title{Occupation Times at a Moving Interface and Sharp Yosida Rates for McKean--Vlasov BSDEs with Jumps and Fuzzy Selections}

\author{%
Kayembe Tshiswaka Tcheick\textsuperscript{a,}\thanks{Corresponding author. Email: tcheick.kayembe@unikin.ac.cd},\ \
Mabela Matendo Rostin\textsuperscript{a},\\[0.5ex]
Kabengele Mpunga Yannick\textsuperscript{a}\ \
and Bosonga Bofeki Jean-Pierre\textsuperscript{b}\\[2ex]
{\small \textsuperscript{a}University of Kinshasa, Faculty of Science and Technology,}\\
{\small Department of Mathematics, Statistics and Computer Science,}\\
{\small Kinshasa, Democratic Republic of the Congo}\\[0.5ex]
{\small \textsuperscript{b}University of Kinshasa, Faculty of Economics and Management,}\\
{\small Department of Mathematical Economics, Kinshasa, Democratic Republic of the Congo}}
\date{}

\begin{document}

\maketitle

\begin{abstract}
We study a forward--backward McKean--Vlasov system with jumps whose backward equation contains a source obtained by evaluating the minimal-norm selection of a maximal monotone operator along the forward state. This selection changes branch across an interface that moves with the state law, the control, and an imprecisely known parameter described by fuzzy cuts. Assuming normal non-degeneracy and finite-statistics dependence of the interface on the law, we prove that the first two moments of the time spent within distance $r$ of the interface are bounded by $Cr$ and $Cr^2$. The second-moment estimate makes it possible to integrate the source error before squaring it, which yields an $O(\lambda)$ Yosida error in the natural solution norm; a first-moment argument gives only $O(\lambda^{1/2})$. An explicit example shows that the order $\lambda$ is sharp. For compact classes of admissible selections, the corresponding solutions form the cuts of a fuzzy solution set. These cuts are compact, stable with respect to the fuzzy level and the data, and approximated at rate $O(\lambda)$ in Hausdorff distance. All estimates are uniform over a non-dominated family of probability measures representing volatility uncertainty. The quantitative results are obtained for separable quadratic potentials with finitely many affine branches.
\end{abstract}



\section{Introduction}\label{sec:intro}

\subsection{The problem and the main idea}

Yosida approximation is a classical tool for regularizing maximal monotone operators in stochastic equations and inclusions \cite{Brezis1973,PardouxRascanu1998}. Here the operator does not act on the backward variable. Its minimal-norm selection is evaluated along the forward state and then accumulated as a source in the backward equation. For the finite-branch potentials considered below, that selection is discontinuous across the switching set, so the regularization error is concentrated in a thin neighborhood of the interface. The size of this error is therefore tied to a probabilistic quantity: the time spent by the forward state near the interface. Our aim is to quantify this mechanism in a mean-field setting with jumps and to identify the resulting rate.

The source accumulated in the backward equation is
\[
K_t=\int_0^t (A^{q_s})^0(X_{s-},\mu_{s-},\alpha_s)\dd s,
\qquad
A^q(x,\mu,\alpha)=\partial\phi^q\bigl(x+h^q(\mu,\alpha)\bigr)+c^q(\mu,\alpha),
\]
where $X$ is the forward state, $\mu$ its law, $\alpha$ a control, $(A^q)^0$ the minimal-norm selection of $A^q$, and $\phi^q$ a separable quadratic potential with finitely many affine branches (Assumption~\ref{ass:finite-branches}). The density of $K$ changes branch when a coordinate of $X+h^q(\mu,\alpha)$ crosses a threshold. The interface is therefore not fixed: it moves with the law, with the control, and with an imprecisely known parameter $q$.

Our main prototype is an asymmetric loss-aversion (ALA) penalty (Assumption~\ref{ass:ALA}): deviations above and below a reference level are penalized with different slopes $\kappa_+^q$ and $\kappa_-^q$, while the reference level $-h^q(\mu,\alpha)$ depends on population statistics and on the control. This is meant to capture situations in which part of the parameter information is epistemic---for example, based on expert assessment rather than on a fully specified probabilistic model. We represent that information by a fuzzy parameter. At fuzzy level $\gamma\in(0,1]$, the admissible values of $q$ form the cut $\Q_\gamma(\mu,\alpha)$, and $q$ is a predictable selection of that cut. Since different selections may lead to different solutions, the resulting object is naturally a \emph{fuzzy set of solutions}.

\paragraph{Why the second occupation moment matters}
Let $\mathsf O_r$ denote the time spent by $X+h^q(\mu,\alpha)$ within distance $r$ of the thresholds, and let $\chi^\lambda$ and $\chi$ be the densities of the regularized and limiting sources. By the finite-branch geometry, the Yosida error is confined, up to a term of order $\lambda$, to a layer whose width is proportional to $\lambda$ (Lemma~\ref{lem:samebranch}). If one uses only the first-moment estimate $\mathbb E[\mathsf O_r]\le Cr$ and squares the pointwise error before integrating, one obtains only
$\mathbb E\bigl[\int_0^T|\chi_t^\lambda-\chi_t|^2\dd t\bigr]\le C\lambda$, hence an $O(\lambda^{1/2})$ bound in the backward norm. Our argument instead uses the second-moment estimate
\[
\mathbb E\bigl[\mathsf O_r^2\bigr]\le Cr^2,
\]
and combines it with a stability estimate for the backward equation that depends on $\sup_{t\le T}|K_t^1-K_t^2|$ rather than on the $L^2$ norm of the source densities (Lemma~\ref{lem:energy}). This allows the source error to be integrated in time \emph{before} it is squared. We then obtain $\mathbb E\bigl[\sup_{t\le T}|K_t^\lambda-K_t|^2\bigr]\le C\lambda^2$ and an $O(\lambda)$ error in the natural solution norm. An explicit model shows that this order is sharp.

\subsection{The system}

The forward state $X$ solves a McKean--Vlasov equation with jumps and Lipschitz coefficients, which does not depend on the backward solution. For a given selection $q$, the backward equation reads
\[
Y_t=g(X_T,\mu_T)+\int_t^Tf_s\dd s+K_T-K_t-\int_t^TZ_s\dd B_s-\int_t^T\int_EU_s(e)\widetilde N(\dd s,\dd e),
\]
where $f_s$ is the Lipschitz generator evaluated at $(s,X_{s-},Y_s,Z_s,U_s,\mu_{s-},\Law(Y_s),\alpha_s)$ and $K$ is the source above. The system is triangular. Its full form and the standing assumptions are given in Sections~\ref{sec:system}--\ref{sec:cadre}. If $\mathfrak A_{\gamma,M}$ is a compact class of admissible selections at level $\gamma$ (Definition~\ref{def:selsystem}), the \emph{solution cut} at level $\gamma$ is the set $\cC_{\gamma,M}$ of quadruples $(Y,Z,U,K)$ obtained as $q$ ranges over $\mathfrak A_{\gamma,M}$.

Model uncertainty on the volatility is represented by a non-dominated family $\PP$ of probability measures, as in \cite{SonerTouziZhang2011,SonerTouziZhang2012} and their extensions with jumps \cite{KaziTaniPossamaiZhou2015,KaziTaniPossamaiZhou2015b}. All estimates are proved under each $P\in\PP$, with constants that do not depend on $P$; no aggregated solution is constructed.

\subsection{Main results}

The quantitative results hold for the separable quadratic finite-branch potentials of Assumption~\ref{ass:finite-branches}. Unless stated otherwise, $C=C(M;\mathfrak C_{\rm str})$ is uniform in the model $P$, the fuzzy level $\gamma$, the data, and the selection (Remark~\ref{rem:constants}). We use the conventions
\begin{gather*}
\sup_{P,\gamma,q}:=\sup_{P\in\PP}\sup_{(\xi,\alpha)\in\mathfrak D_P}\sup_{\gamma\in(0,1]}\sup_{q\in\mathfrak A_{\gamma,M}^P[\xi,\alpha]},\qquad
\sup_{P,\gamma}:=\sup_{P\in\PP}\sup_{(\xi,\alpha)\in\mathfrak D_P}\sup_{\gamma\in(0,1]},
\end{gather*}
where $\mathfrak D_P$ is the class of data $(\xi,\alpha)$; the dependence on the data is omitted in the compact displays below.

\begin{enumerate}[label=(\Alph*),leftmargin=*]
\item \emph{Occupation of a moving interface} (Theorem~\ref{thm:occupation}). Let $\mathsf O_r^{P,q}=\mathsf O_r^{P,q}[\xi,\alpha]$ be the time spent at distance at most $r$ from the interface. If the interface depends on the law through finitely many statistics, the control and the selections are absolutely continuous in time, and the diffusion is non-degenerate in the normal direction, then
\[
\sup_{P,\gamma,q}\mathbb E^P[\mathsf O_r^{P,q}]\le Cr,
\qquad
\sup_{P,\gamma,q}\mathbb E^P[(\mathsf O_r^{P,q})^2]\le Cr^2,
\qquad 0<r\le1 .
\]
Only second moments of the jumps are required. Remark~\ref{rem:time-scope} shows that the time regularity of the control cannot simply be dropped.

\item \emph{Sharp Yosida rate} (Theorems~\ref{thm:yosidarate} and~\ref{thm:sharp}). Let $\mathcal Y^{P,q}=(Y,Z,U,K)$ be the selected solution, $\mathcal Y^{P,q,\lambda}$ its Yosida regularization, and $\|\cdot\|_{\mathfrak E_P^2}$ the natural norm of the quadruple. Then
\[
\sup_{P,\gamma,q}\|\mathcal Y^{P,q,\lambda}-\mathcal Y^{P,q}\|_{\mathfrak E_P^2}\le C\lambda,\qquad 0<\lambda\le1,
\]
and there is a model in the class for which $\inf_{P\in\PP}\|\mathcal Y^{P,\lambda}-\mathcal Y^{P}\|_{\mathfrak E_P^2}\ge c\lambda$ with $c>0$. Qualitative convergence for a fixed selection holds without the quantitative assumptions (Theorem~\ref{thm:yosidaqual}).

\item \emph{Approximation of the cuts} (Theorem~\ref{thm:hausY}, Corollary~\ref{cor:Y-joint}). The cuts $\cC_{\gamma,M,\lambda}^P$ of the regularized system satisfy
\[
\sup_{P,\gamma}d_H\bigl(\cC_{\gamma,M,\lambda}^P,\cC_{\gamma,M}^P\bigr)\le C\lambda,\qquad 0<\lambda\le1,
\]
where $d_H$ is the Hausdorff distance in $\mathfrak E_P^2$.

\item \emph{The fuzzy solution set} (Theorem~\ref{thm:stability}, Corollary~\ref{cor:lipschitzT}, Theorem~\ref{thm:cutcompact}, Propositions~\ref{prop:leftcuts}, \ref{prop:hausdorff-data}, and~\ref{prop:fuzzy-standard}). The solution map is Lipschitz with respect to the selection. For fixed $P$, the cuts $\cC_{\gamma,M}^P$ are compact, nested, and left-continuous, and they are exactly the cuts of a normal upper semicontinuous fuzzy set of solutions with compact support. For two sets of data,
\[
d_H\bigl(\cC_{\gamma,M}^P[\xi,\alpha],\cC_{\gamma',M}^P[\xi',\alpha']\bigr)
\le C\,\delta_P^{1/4}+C|\gamma-\gamma'|^\nu,
\]
where $\delta_P=\mathbb E^P[|\xi-\xi'|^2]+\mathbb E^P[\int_0^T|\alpha_t-\alpha'_t|^2\dd t]$.

\item \emph{Selection classes built from the fuzzy data} (Theorem~\ref{thm:feedback-class}, Corollary~\ref{cor:feedback}, Proposition~\ref{prop:center-radius}). The compactness and transport properties (S1)--(S5) used in (C)--(D) are not merely postulated: they are derived from the geometry of the fuzzy data for Lipschitz feedback selections of convex cuts when the fuzzy parameter is scalar, and for center--radius cuts in any dimension. For general convex cuts in dimension $d_q\ge2$, all conclusions remain valid except the quantitative Hausdorff bounds that rely on (S4).
\end{enumerate}

Well-posedness for a fixed selection (Theorem~\ref{thm:backward}, Corollary~\ref{cor:wellposed-selected}) rests on classical fixed-point and stability arguments; it serves as a starting point and is not claimed as a main contribution.

\subsection{Relation to the literature}

McKean--Vlasov equations describe dynamics whose coefficients depend on the law of the state \cite{McKean1966,Sznitman1991,CarmonaDelarue2018}, while BSDEs with jumps and law dependence are by now well established \cite{PardouxPeng1990,Royer2006,Li2018}. The closest points of contact with the present work are the following.

\paragraph{Comparison with the closest work}
The closest related preprint is \cite{KayembeSourceJDE}, which treats the same architecture without a fuzzy parameter, with an interface shifted by the control alone. For the ALA prototype and linear occupation, \cite{KayembeSourceJDE} establishes a first-moment occupation bound and derives from it, under a moment assumption of order $2+\delta$ on the source, a squared error of order $\lambda^{\varrho}$ with $\varrho=\delta/(3+\delta)<1$; the corresponding error in norm is therefore of order $\lambda^{\varrho/2}$. To compare on a common basis, specialize the present framework to this prototype: the selection is fixed, the interface no longer depends on the law, and the absolute continuity of the control and the normal non-degeneracy are already imposed in \cite{KayembeSourceJDE}. The main additional regularity requirement in this specialization is the $C^1$ representation with bounded derivatives of the shift in Assumption~\ref{ass:finitestat}, whereas \cite{KayembeSourceJDE} assumes a Lipschitz shift. In this specialization, the present paper obtains an $O(\lambda)$ rate in the natural solution norm without requiring jump moments above order two, and shows that this rate is sharp. This gain rests on three complementary ingredients: the finite-branch geometry, which localizes the Yosida error in a layer of width proportional to $\lambda$; the second-moment estimate for the occupation of this layer; and the stability Lemma~\ref{lem:energy}, which makes it possible to work with the integrated source. In the general model studied here, Assumption~\ref{ass:finitestat} also covers the motion of the interface due to the selection and to finitely many statistics of the law. The preprint \cite{KayembeRobustALA} introduces the ALA source in a robust McKean--Vlasov variational setting, with well-posedness, stability, Yosida approximation, and propagation of chaos, but without occupation estimates. Neither \cite{KayembeSourceJDE} nor \cite{KayembeRobustALA} treats fuzzy uncertainty.

\paragraph{Discontinuous coefficients}
Occupation and local-time estimates near a discontinuity are a central tool for SDEs with discontinuous drift, their numerical approximation, and irregular functionals of SDEs \cite{LeobacherSzolgyenyi2017,LeobacherSzolgyenyi2018,MullerGronbachYaroslavtseva2020,Avikainen2009,FournierPrintems2010,MenoukeuOuknineTangpi2019}. For jump diffusions with discontinuous drift, existence, uniqueness, and a strong Euler--Maruyama rate are proved in \cite{PrzybylowiczSzolgyenyi2021}, and matching lower error bounds in \cite{PrzybylowiczSchwarzSzolgyenyi2024}. In these works the discontinuity set is fixed by the coefficients. Here the required estimate is different: we need a second occupation moment, uniform in the selection, for an interface that depends simultaneously on the selection, the law, and the control.

\paragraph{Multivalued and mean-field backward equations}
BSDEs with subdifferential operators go back to \cite{PardouxRascanu1998}. On the McKean--Vlasov side, multivalued equations and variational inequalities include well-posed multivalued BSDEs \cite{QiaoGong2023}, stability for multivalued forward equations \cite{QiaoGong2025}, and coupled forward--backward systems \cite{NingWuZheng2024,DuanWu2026}. Mean-field BSDEs with jumps also enter dual representations of global risk measures \cite{ChenDumitrescuMincaSulem2023}, and non-asymptotic rates for McKean--Vlasov BSDEs arising from mean-field games are studied in \cite{PossamaiTangpi2025}. In \cite{PardouxRascanu1998,QiaoGong2023,QiaoGong2025,NingWuZheng2024,DuanWu2026} the multivalued operator acts on the variables of the system as a constraint or an inclusion; here it produces a source evaluated along the forward state.
To our knowledge, the works cited above do not establish, in their respective settings, a quantitative Yosida rate in the full solution norm of the type derived here.

\paragraph{Fuzzy stochastic equations}
Fuzzy stochastic differential equations have been developed in Brownian, set-valued, and semimartingale settings \cite{MalinowskiMichta2011,MalinowskiMichta2012,MalinowskiMichtaSobolewska2013,Michta2015,Malinowski2016}; in that literature the solution itself is typically fuzzy-set-valued. Here, by contrast, the elements of the fuzzy set are ordinary solution quadruples $(Y,Z,U,K)$, obtained by transporting compact families of predictable selections through the solution map.

\subsection{Scope}

The quantitative results (A), (B), (C) and the rates in (D) require normal non-degeneracy (Assumption~\ref{ass:elliptic}), the finite-statistics and time-regularity structure of Assumption~\ref{ass:finitestat} (see Remark~\ref{rem:time-scope} for a counterexample otherwise), and separable potentials with finitely many affine branches and a bounded shift $c^q$ (Assumption~\ref{ass:finite-branches}). Well-posedness and qualitative Yosida convergence do not. The solution cuts, and hence the fuzzy solution set, are defined relative to a fixed compact selection system (Definition~\ref{def:selsystem}, Remark~\ref{rem:chosen-class}). The sharpness of the rate is established by an explicit example within the class considered: it shows that the order $\lambda$ cannot be improved uniformly, not that it is attained for every model. Fully coupled systems, aggregation over $\PP$ in the sense of second-order BSDEs, and quantitative Hausdorff bounds for general convex cuts in dimension $d_q\ge2$ are not covered.

\subsection{Organization}

Section~\ref{sec:cadre} states the system and sets up the framework. Section~\ref{sec:bienpose} recalls well-posedness and establishes the stability lemma in the integrated-source norm. Section~\ref{sec:occupation-stabilite} contains the occupation estimates. Section~\ref{sec:feedback} constructs the selection classes. Stability and the fuzzy solution set are developed in Sections~\ref{sec:stability} and~\ref{sec:coupes}. Section~\ref{sec:yosida} is devoted to the Yosida rates and their optimality, and Section~\ref{sec:example} verifies the assumptions on an explicit example before the concluding discussion in Section~\ref{sec:conclusion}.

\section{Framework and assumptions}\label{sec:cadre}

\subsection{The selected system}\label{sec:system}

We begin with the full system and then introduce the objects and assumptions used throughout the paper.

Let $T>0$ and let $\PP$ be a non-dominated family of probability measures describing volatility uncertainty (Section~\ref{sec:base}). Under $P\in\PP$, $B^P$ is a continuous martingale with $\dd\langle B^P\rangle_t=a_t^P\dd t$, $\underline a\le a_t^P\le\overline a$, and $\widetilde N^P$ is a compensated Poisson random measure with compensator $\Pi(\dd e)\dd t$. The initial value is a square-integrable $\mathcal F_0$-measurable random variable,
\[
X_0^P=\xi\in L^2(\Omega,\mathcal F_0,P;\R^d),\qquad\text{that is,}\qquad \mathbb E^P\!\left[|\xi|^2\right]<\infty,
\]
so that the initial law $\mu_0^P=\Law^P(\xi)$ belongs to $\Ptwo(\R^d)$. For such an initial datum $\xi$ and a predictable control $\alpha$, the forward state $X^P$ and $\mu_t^P=\Law^P(X_t^P)$ satisfy
\begin{align}
X_t^P={}&\xi+\int_0^t b(s,X_{s-}^P,\mu_{s-}^P,\alpha_s)\dd s
 +\int_0^t\sigma(s,X_{s-}^P,\mu_{s-}^P,\alpha_s)\dd B_s^P \notag\\
&+\int_0^t\int_E\beta(s,X_{s-}^P,\mu_{s-}^P,\alpha_s,e)\widetilde N^P(\dd s,\dd e).
\label{eq:forward}
\end{align}
The epistemic parameter is a predictable process $q$ taking values in a compact set $\Q\subset\R^{d_q}$ and, at fuzzy level $\gamma\in(0,1]$,
\begin{equation}
q_t\in\Q_{\gamma}(\mu_{t-}^P,\alpha_t),\qquad \dd t\otimes\dd P\text{-a.e.}
\label{eq:selection}
\end{equation}
For $q\in\Q$, set
\begin{equation}
A^q(x,\mu,\alpha):=\partial\phi^q\bigl(x+h^q(\mu,\alpha)\bigr)+c^q(\mu,\alpha),
\label{eq:operator-intro}
\end{equation}
where $\phi^q$ is convex, $h^q$ shifts the interface, and $c^q$ is a translation. The ALA prototype (Assumption~\ref{ass:ALA}) combines a quadratic part with distinct slopes $\kappa_\pm^q$ on the two sides of the interface. If $(A^q)^0$ denotes the minimal-norm selection,
\begin{equation}
K_t^{P,q}=\int_0^t (A^{q_s})^0(X_{s-}^P,\mu_{s-}^P,\alpha_s)\dd s,
\label{eq:K}
\end{equation}
and, with $\eta_t^{P,q}=\Law^P(Y_t^{P,q})$,
\begin{align}
Y_t^{P,q}={}&g(X_T^P,\mu_T^P)
+\int_t^T f(s,X_{s-}^P,Y_s^{P,q},Z_s^{P,q},U_s^{P,q},\mu_{s-}^P,\eta_s^{P,q},\alpha_s)\dd s
+K_T^{P,q}-K_t^{P,q}\notag\\
&-\int_t^T Z_s^{P,q}\dd B_s^P
-\int_t^T\int_E U_s^{P,q}(e)\widetilde N^P(\dd s,\dd e).
\label{eq:backward}
\end{align}
The system is triangular. Distinct admissible selections may produce distinct solutions; the images of selection classes therefore form the cuts of a fuzzy solution set. The quantitative difficulty stems from the discontinuity of $(A^q)^0$ when $X^P+h^q(\mu^P,\alpha)$ crosses the interface.

\subsection{Stochastic framework and non-dominated family}\label{sec:base}

We work in the volatility-uncertainty framework of Soner, Touzi, and Zhang \cite{SonerTouziZhang2011,SonerTouziZhang2012}, with jumps as in \cite{KaziTaniPossamaiZhou2015}. Let $T>0$ and let $(\Omega,\mathcal F)$ be a measurable space endowed with a filtration $\mathbb F=(\mathcal F_t)_{0\le t\le T}$, an $\R^n$-valued càdlàg process $B$, and an integer-valued random measure $N$ on $[0,T]\times E$, where $(E,\mathcal E)$ is a measurable space. The filtration $\mathbb F$ is the raw, non-completed filtration. Let $\PP$ be a nonempty family of probability measures on $(\Omega,\mathcal F)$, and let $0<\underline a\le\overline a$ be two symmetric $n\times n$ matrices, with matrix inequalities understood in the Loewner sense. The family $\PP$ is said to be \emph{dominated} if there exists a probability measure $Q$ such that $P\ll Q$ for every $P\in\PP$, and \emph{non-dominated} otherwise; families of laws of martingales with different quadratic-variation densities provide standard examples of mutually singular models \cite{SonerTouziZhang2011}. For $P\in\PP$, let $\mathbb F^P$ denote the usual augmentation of $\mathbb F$ under $P$; ``predictable'' means predictable with respect to $\mathbb F^P$, and random fields $U(e)$ are measurable with respect to the product of the predictable $\sigma$-field and $\mathcal E$.

\paragraph{Conventions} Matrix norms are Hilbert--Schmidt norms. The letter $C$ denotes a deterministic constant that may change from line to line; its dependence is specified in Remark \ref{rem:constants}. For fixed $\lambda$, $C_\lambda$ denotes a constant that may depend on $\lambda$. Integrands are evaluated at left limits to ensure predictability; since a càdlàg path has at most countably many jumps, $X_{s-}=X_s$ for $\dd s\otimes\dd P$-almost every $(s,\omega)$.

For $P\in\PP$ (see Assumption \ref{ass:base} below for $a^P$), a $d\times n$ matrix $z$, and $u\in L^2(\Pi)$, set
\[
|z|_{a_t^P}^2=\operatorname{Tr}(za_t^Pz^\top),
\qquad
\|u\|_\Pi^2=\int_E|u(e)|^2\Pi(\dd e).
\]
We write $L^2_\Pi(E;\R^d)$, or $L^2(\Pi)$, for the space of measurable maps $u:E\to\R^d$ such that $\|u\|_\Pi<\infty$. For $s\in\R$, $s^+=\max(s,0)$ and $s^-=\max(-s,0)$; $\operatorname{Leb}$ denotes Lebesgue measure on $[0,T]$. For a function $V$ from $[0,T]$ into a Euclidean space, $\Var_{[0,T]}(V)=\sup\sum_{i=1}^{p}|V_{t_i}-V_{t_{i-1}}|$ denotes its total variation in the Euclidean norm, where the supremum is taken over partitions $0=t_0<t_1<\cdots<t_p=T$; if $V$ is absolutely continuous, $\Var_{[0,T]}(V)=\int_0^T|\dot V_t|\dd t$.
For $P\in\PP$, $\cS_P^2$ denotes the space of adapted càdlàg processes $Y$ such that $\mathbb E^P[\sup_{t\le T}|Y_t|^2]<\infty$, $\cH_P^2$ the space of predictable processes $Z$ such that $\mathbb E^P[\int_0^T|Z_t|_{a_t^P}^2\dd t]<\infty$, and $\cJ_P^2$ the space of predictable fields $U$ such that $\mathbb E^P[\int_0^T\|U_t\|_\Pi^2\dd t]<\infty$. Equalities between processes are understood up to indistinguishability, and relations between coefficients hold $\dd t\otimes\dd P$-a.e.
Since $\underline a\le a_t^P\le\overline a$, one has $\lambda_{\min}(\underline a)|z|^2\le|z|_{a_t^P}^2\le\lambda_{\max}(\overline a)|z|^2$ for every $d\times n$ matrix $z$, where $|z|$ is the Hilbert--Schmidt norm. Thus, as a set, $\cH_P^2$ coincides with the space of $\R^{d\times n}$-valued predictable processes that are square integrable with respect to $\dd t\otimes\dd P$, with equivalent norms and equivalence constants independent of $P$. After identifying elements at zero distance, $\cS_P^2$ is a Banach space, while $\cH_P^2$ and $\cJ_P^2$ are Hilbert spaces.

\paragraph{Additional notation} For a measurable process $V$ with values in a Euclidean space, $\|V\|_{\infty,P}$ denotes its norm in $L^\infty(\dd t\otimes\dd P)$, namely the essential supremum of $|V_t(\omega)|$ with respect to $\dd t\otimes\dd P$. For nonempty subsets $\mathcal A,\mathcal B$ of a metric space $(\mathfrak E,\mathrm d)$, the excess of $\mathcal A$ over $\mathcal B$ and the Hausdorff distance are
\[
e(\mathcal A,\mathcal B)=\sup_{x\in\mathcal A}\inf_{y\in\mathcal B}\mathrm d(x,y),
\qquad
d_H^{\mathfrak E}(\mathcal A,\mathcal B)=\max\{e(\mathcal A,\mathcal B),e(\mathcal B,\mathcal A)\}.
\]
The map $d_H^{\mathfrak E}$ is a distance on the collection of nonempty compact subsets of $\mathfrak E$ \cite{DiamondKloeden1994}; we write $d_H$ when the ambient space is clear.

\begin{assumption}\label{ass:base}
For every $P\in\PP$:
\begin{enumerate}[label=(\roman*)]
\item under $P$, $B=B^P$ is a continuous martingale with bracket $\dd\langle B\rangle_t=a_t^P\dd t$, where $a^P$ is predictable and $\underline a\le a_t^P\le\overline a$, $\dd t\otimes\dd P$-a.e.;
\item under $P$, $N$ is a Poisson random measure with compensator $\Pi(\dd e)\dd t$, where the $\sigma$-finite measure $\Pi$ does not depend on $P$; we write $\widetilde N^P=N-\Pi(\dd e)\dd t$;
\item (predictable representation) every square-integrable $(\mathbb F^P,P)$-martingale $\mathsf M$ admits the representation
\[
\mathsf M_t=\mathsf M_0+\int_0^t Z_s\dd B_s^P+
\int_0^t\int_E U_s(e)\widetilde N^P(\dd s,\dd e),
\qquad t\in[0,T],
\]
with $(Z,U)\in\cH_P^2\times\cJ_P^2$, uniquely in these spaces.
\end{enumerate}
\end{assumption}

The canonical example is the canonical space, on which $\PP$ is a family of strong-formulation measures in the sense of \cite{SonerTouziZhang2011}, under which the volatility of $B$ is uncertain in $[\underline a,\overline a]$ and the jump measure is Poisson with a fixed compensator. For the continuous part, predictable representation is established in \cite{SonerTouziZhang2011,SonerTouziZhang2012}; for the jump part, our fixed-compensator assumption is a special case of the framework in \cite{KaziTaniPossamaiZhou2015}, where the compensator may depend on $P$. The full property (iii), for the pair $(B^P,\widetilde N^P)$, nevertheless remains an independent assumption on the admissible models: it does not follow from non-domination alone and rules out any additional orthogonal martingale component. The classical setting of an independent Brownian motion and Poisson random measure with predictable representation corresponds to $\PP=\{P\}$ and $a^P$ equal to the identity. Since the equations only involve increments of $B$, we normalize $B_0=0$. A generic control value in $\mathbb D$ is denoted by $\alpha$, as is the control process $(\alpha_t)$; the letter $a$ is reserved for the covariance $a_t^P$ and its bounds $\underline a,\overline a$.

All constructions are carried out under a fixed probability measure $P\in\PP$. Since the constants are uniform over $\PP$ (Remark \ref{rem:constants}), statements of the form
\[
\sup_{P\in\PP}\mathbb E^P\!\left[\Phi^P\right]\le C
\]
are robust estimates with respect to volatility uncertainty. We do not construct an aggregated version of the solutions: the coefficients depend on the law $\Law^P(X^P)$, which varies with $P$, and aggregation in the sense of second-order BSDEs \cite{SonerTouziZhang2012,KaziTaniPossamaiZhou2015} lies beyond the scope of this paper.

Under each $P$, the Burkholder--Davis--Gundy (BDG) and Doob inequalities for stochastic integrals with respect to $B^P$ and $\widetilde N^P$ hold with constants independent of $P$: the continuous part is measured in the norm $|\cdot|_{a^P}$, and the jump part has the same compensator $\Pi(\dd e)\dd t$ for every model. Explicitly, for $(Z,U)\in\cH_P^2\times\cJ_P^2$ and $\mathsf M=\int_0^\cdot Z_s\dd B_s^P+\int_0^\cdot\int_EU_s(e)\widetilde N^P(\dd s,\dd e)$,
\begin{gather*}
\mathbb E^P\bigl[[\mathsf M]_T\bigr]=\mathbb E^P\!\left[\int_0^T\bigl(|Z_t|_{a_t^P}^2+\|U_t\|_\Pi^2\bigr)\dd t\right],\\
\mathbb E^P[\sup_{t\le T}|\mathsf M_t|^2]\le C_{\rm BDG}\,\mathbb E^P\bigl[[\mathsf M]_T\bigr],
\qquad
\mathbb E^P[\sup_{t\le T}|\mathsf M_t|]\le C_{\rm D}\,\mathbb E^P\bigl[[\mathsf M]_T^{1/2}\bigr],
\end{gather*}
with $C_{\rm BDG}=4$ (Doob) and a universal constant $C_{\rm D}$ (BDG, $p=1$) \cite[chap.~IV]{Protter2005}. Throughout the paper, Itô's formula for semimartingales with jumps \cite[chap.~II]{Protter2005} is used without further comment, and the isometry and compensation formula for stochastic integrals with respect to $\widetilde N^P$ are those of \cite[chap.~II]{IkedaWatanabe1989}. A \emph{strong solution} under $P$ is an $\mathbb F^P$-adapted solution constructed on the given space from $(\xi,B^P,\widetilde N^P)$.

Let $\Ptwo(\R^d)$ denote the set of Borel probability measures with finite second moment. For $\mu,\mu'\in\Ptwo(\R^d)$, the order-two Wasserstein distance is
\[
W_2(\mu,\mu')=
\left(\inf_{\pi\in\operatorname{Cpl}(\mu,\mu')}
\int_{\R^d\times\R^d}|x-y|^2\pi(\dd x,\dd y)\right)^{1/2},
\]
where $\operatorname{Cpl}(\mu,\mu')$ is the set of probability measures with marginals $\mu,\mu'$. The Dirac measure at $0$ is denoted by $\mathfrak d_0$. We set
\[
M_2(\mu)=\left(\int_{\R^d}|x|^2\mu(\dd x)\right)^{1/2}.
\]
For square-integrable random variables $V,V'$ under $P$, the natural coupling yields $W_2^2(\Law^P(V),\Law^P(V'))\le\mathbb E^P[|V-V'|^2]$. The space $(\Ptwo(\R^d),W_2)$ is a complete separable metric space \cite[chap.~6]{Villani2009}. Consequently, the Borel $\sigma$-field of a finite product of separable metric spaces, such as $[0,T]\times\R^d\times\Ptwo(\R^d)\times\mathbb D$, is the product of the Borel $\sigma$-fields; joint Borel measurability is always understood in this sense.

\begin{lemma}\label{lem:lawflows}
Fix $P\in\PP$ and $d'\ge1$. Let $V,V'\in L^2(\dd t\otimes\dd P;\R^{d'})$ be two jointly measurable processes. Then there exist Borel maps $\upsilon^V,\upsilon^{V'}:[0,T]\to(\Ptwo(\R^{d'}),W_2)$ such that $\upsilon_t^V=\Law^P(V_t)$ and $\upsilon_t^{V'}=\Law^P(V'_t)$ for almost every $t\in[0,T]$. Their equivalence classes depend only on the equivalence classes of $V$ and $V'$ in $L^2(\dd t\otimes\dd P)$, and
\[
\int_0^TW_2^2(\upsilon_t^V,\upsilon_t^{V'})\dd t\le\mathbb E^P\!\left[\int_0^T|V_t-V'_t|^2\dd t\right].
\]
\end{lemma}

\begin{proof}
By Tonelli's theorem, $t\mapsto\mathbb E^P[|V_t|^2]$ is measurable and finite for almost every $t$; on the remaining null set, set $\upsilon_t^V=\mathfrak d_0$, and elsewhere set $\upsilon_t^V=\Law^P(V_t)$. For $\psi\in C_b(\R^{d'})$, Fubini's theorem shows that $t\mapsto\int\psi\dd \upsilon_t^V=\mathbb E^P[\psi(V_t)]$ is measurable; a countable family of such functions determining weak convergence yields measurability of $\upsilon^V$ for the weak topology. Since $t\mapsto M_2(\upsilon_t^V)^2$ is also measurable and the $W_2$ topology is the topology of weak convergence together with convergence of second moments, $\upsilon^V$ is Borel as a map into $(\Ptwo(\R^{d'}),W_2)$. Finally, for almost every $t$, $\Law^P(V_t,V'_t)$ is a coupling of $(\upsilon_t^V,\upsilon_t^{V'})$, hence $W_2^2(\upsilon_t^V,\upsilon_t^{V'})\le\mathbb E^P[|V_t-V'_t|^2]$; integration in time concludes the proof, and the case $V=V'$ in $L^2(\dd t\otimes\dd P)$ shows that the flows coincide for almost every $t$.
\end{proof}

\paragraph{Coefficients and admissible controls}
The coefficients have the following dimensions:
\begin{align*}
&b:[0,T]\times\R^d\times\Ptwo(\R^d)\times\mathbb D\to\R^d,
\qquad
\sigma:[0,T]\times\R^d\times\Ptwo(\R^d)\times\mathbb D\to\R^{d\times n},\\
&\beta:[0,T]\times\R^d\times\Ptwo(\R^d)\times\mathbb D\times E\to\R^d,\\
&f:[0,T]\times\R^d\times\R^d\times\R^{d\times n}\times L^2_\Pi(E;\R^d)\times\Ptwo(\R^d)\times\Ptwo(\R^d)\times\mathbb D\to\R^d,\\
&g:\R^d\times\Ptwo(\R^d)\to\R^d,
\end{align*}
where $\mathbb D\subset\R^{d_\alpha}$ is a Borel set containing $0$, and $f(t,x,y,z,u,\mu,\eta,\alpha)$ is evaluated at $y=Y_t$, $z=Z_t$, $u=U_t$, $\mu=\mu_{t-}$, $\eta=\Law^P(Y_t)$. To ensure that the same control is defined under all models, let $\mathcal A_2$ be the set of $\mathbb D$-valued processes $\alpha$ that are predictable with respect to the raw filtration $\mathbb F$ (and hence $\mathbb F^P$-predictable under every $P\in\PP$) and satisfy $\sup_{P\in\PP}\mathbb E^P\!\left[\int_0^T|\alpha_t|^2\dd t\right]<\infty$. The initial datum $\xi$ is $\mathcal F_0$-measurable for the raw filtration.

All statements below are formulated under a fixed probability measure $P\in\PP$. Their proofs only use $\xi\in L^2(\Omega,\mathcal F_0^P,P;\R^d)$ and the fact that $\alpha$ is $\mathbb F^P$-predictable, $\mathbb D$-valued, and satisfies $\mathbb E^P\!\left[\int_0^T|\alpha_t|^2\dd t\right]<\infty$. We denote by $\mathcal A_2^P\supseteq\mathcal A_2$ this larger class of controls. All results stated for $\alpha\in\mathcal A_2$ therefore remain valid verbatim for model-dependent controls $\alpha=\alpha^P\in\mathcal A_2^P$, for instance feedback controls constructed from $X^P$ (Section~\ref{sec:example}); uniformity of the constants over $\PP$ then requires the bounds on the data to be uniform in $P$.

\begin{assumption}\label{ass:F}
The coefficients $b,\sigma$ are jointly Borel in $(t,x,\mu,\alpha)$, and $\beta$ is jointly Borel in $(t,x,\mu,\alpha,e)$.

There exist $L_F,C_F>0$, independent of $P$, such that, for all $(t,x,\mu,\alpha)$ and $(t,x',\mu',\alpha')$,
\begin{align*}
&|b(t,x,\mu,\alpha)-b(t,x',\mu',\alpha')|^2
+|\sigma(t,x,\mu,\alpha)-\sigma(t,x',\mu',\alpha')|_{a_t^P}^2\\
&\qquad
+\int_E|\beta(t,x,\mu,\alpha,e)-\beta(t,x',\mu',\alpha',e)|^2\Pi(\dd e)\\
&\le
L_F^2\bigl(|x-x'|^2+W_2^2(\mu,\mu')+|\alpha-\alpha'|^2\bigr),
\end{align*}
and
\begin{align*}
&|b(t,0,\mathfrak d_0,0)|^2
+|\sigma(t,0,\mathfrak d_0,0)|_{a_t^P}^2
+\int_E|\beta(t,0,\mathfrak d_0,0,e)|^2\Pi(\dd e)
\le C_F^2.
\end{align*}
The initial state and the control satisfy
\[
\sup_{P\in\PP}\mathbb E^P\!\left[|\xi|^2\right]
+\sup_{P\in\PP}\mathbb E^P\!\left[\int_0^T|\alpha_t|^2\dd t\right]<\infty.
\]
\end{assumption}

\begin{assumption}\label{ass:B}
The generator $f$ is jointly Borel in its arguments, and $g$ is Borel. There exist $L_B,C_B>0$, independent of $P$, such that
\begin{align*}
&|f(t,x,y,z,u,\mu,\eta,\alpha)-f(t,x',y',z',u',\mu',\eta',\alpha')|\\
&\quad\le L_B\Bigl(
|x-x'|+|y-y'|+|z-z'|_{a_t^P}+\|u-u'\|_\Pi
+W_2(\mu,\mu')+W_2(\eta,\eta')+|\alpha-\alpha'|
\Bigr),
\end{align*}
and $|f(t,0,0,0,0,\mathfrak d_0,\mathfrak d_0,0)|\le C_B.$ The terminal function satisfies
\[
|g(x,\mu)-g(x',\mu')|
\le L_B\bigl(|x-x'|+W_2(\mu,\mu')\bigr),
\qquad |g(0,\mathfrak d_0)|\le C_B.
\]
\end{assumption}

\begin{remark}\label{rem:random-norms}
The norm $|\cdot|_{a_t^P}$ is random. The inequalities in Assumptions~\ref{ass:F} and \ref{ass:B} involving this norm, as well as the inequality in Assumption~\ref{ass:elliptic} below, are required for every $P\in\PP$, for $\dd t\otimes\dd P$-almost every $(t,\omega)$, and for all values of the remaining arguments. By the norm equivalence in Section~\ref{sec:base}, they hold whenever the corresponding deterministic inequalities hold for the Hilbert--Schmidt norm: in Assumption~\ref{ass:F}, with $L_F^2,C_F^2$ multiplied by $\max(1,\lambda_{\max}(\overline a))$; in Assumption~\ref{ass:B}, with $L_B$ multiplied by $\max(1,\lambda_{\min}(\underline a)^{-1/2})$; and in Assumption~\ref{ass:elliptic} if $\mathbf e_k^\top\sigma\sigma^\top \mathbf e_k\ge\underline\sigma^2/\lambda_{\min}(\underline a)$, since $\sigma a_t^P\sigma^\top\ge\lambda_{\min}(\underline a)\sigma\sigma^\top$.
\end{remark}

\subsection{Cuts of the fuzzy parameter}

Let $\Q\subset\R^{d_q}$ be compact and nonempty. For $(\mu,\alpha)\in\Ptwo(\R^d)\times\mathbb D$, the fuzzy parameter $\widetilde{\Q}(\mu,\alpha)$, viewed as a fuzzy subset of $\R^{d_q}$ in the sense of \cite{DiamondKloeden1994}, is described by its cuts
\[
\Q_{\gamma}(\mu,\alpha)=[\widetilde{\Q}(\mu,\alpha)]_\gamma,
\qquad \gamma\in(0,1].
\]

\begin{assumption}\label{ass:fuzzycuts}
\begin{enumerate}[label=(\roman*)]
\item For every $(\gamma,\mu,\alpha)$, $\Q_{\gamma}(\mu,\alpha)$ is nonempty and compact, and the cuts are nested:
\[
0<\gamma_1\le\gamma_2\le1
\quad\Longrightarrow\quad
\Q_{\gamma_2}(\mu,\alpha)\subseteq\Q_{\gamma_1}(\mu,\alpha).
\]
\item For every $\gamma\in(0,1]$, the multifunction $(\mu,\alpha)\mapsto\Q_\gamma(\mu,\alpha)$ is Borel measurable: for every open set $O\subset\R^{d_q}$,
\[
\{(\mu,\alpha)\in\Ptwo(\R^d)\times\mathbb D:\ \Q_\gamma(\mu,\alpha)\cap O\ne\varnothing\}
\]
is a Borel set.
\end{enumerate}
\end{assumption}

\begin{remark}
The selection $q$ represents an admissible epistemic realization. It should not be confused with the minimal-norm selection $(A^q)^0$ of the maximal monotone operator.
\end{remark}

\subsection{Maximal monotone family}
\begin{assumption}\label{ass:operator}
For every $q\in\Q$, $\phi^q:\R^d\to\R$ is convex and finite on $\R^d$ (hence continuous). In particular, $D(\partial\phi^q)=\R^d$. The maps $(q,\mu,\alpha)\mapsto h^q(\mu,\alpha)$ and $(q,\mu,\alpha)\mapsto c^q(\mu,\alpha)$ are Borel measurable, and there exist constants $L_h,L_c,C_{hc},C_\phi$ such that, for all $q,q'\in\Q$, $\mu,\mu'\in\Ptwo(\R^d)$, and $\alpha,\alpha'\in\mathbb D$,
\begin{align*}
|h^q(\mu,\alpha)-h^{q'}(\mu',\alpha')|
&\le L_h\bigl(|q-q'|+W_2(\mu,\mu')+|\alpha-\alpha'|\bigr),\\
|c^q(\mu,\alpha)-c^{q'}(\mu',\alpha')|
&\le L_c\bigl(|q-q'|+W_2(\mu,\mu')+|\alpha-\alpha'|\bigr),
\end{align*}
and
\[
|h^q(\mu,\alpha)|+|c^q(\mu,\alpha)|
\le C_{hc}(1+|\alpha|+M_2(\mu)).
\]
Finally, for the minimal selection, $|(\partial\phi^q)^0(z)|\le C_\phi(1+|z|)$ for all $q\in\Q$ and $z\in\R^d$.
\end{assumption}

For $x\in\R^d$, set
\[
A^q(x,\mu,\alpha)=\partial\phi^q(x+h^q(\mu,\alpha))+c^q(\mu,\alpha).
\]
For fixed $(q,\mu,\alpha)$, the map $x\mapsto A^q(x,\mu,\alpha)$ is maximal monotone, since it is obtained by translating a maximal monotone subdifferential. Its values are nonempty closed convex sets; we denote by $(A^q)^0(x,\mu,\alpha)$ the unique element of minimal Euclidean norm in $A^q(x,\mu,\alpha)$ (the \emph{minimal-norm selection}). For $\lambda>0$, the \emph{resolvent} and the \emph{Yosida approximation} in the state variable are
\begin{equation}\label{eq:resolvent-def}
J_\lambda^{q,\mu,\alpha}
=
(I+\lambda A^q(\cdot,\mu,\alpha))^{-1},
\qquad
A_\lambda^q(x,\mu,\alpha)
=
\frac{x-J_\lambda^{q,\mu,\alpha}(x)}{\lambda};
\end{equation}
we also write $J_\lambda^q(x,\mu,\alpha)=J_\lambda^{q,\mu,\alpha}(x)$. For a convex function $\phi$, $J_{\lambda\partial\phi}=(I+\lambda\partial\phi)^{-1}$ and $(\partial\phi)_\lambda=\lambda^{-1}(I-J_{\lambda\partial\phi})$. Since $A^q(\cdot,\mu,\alpha)$ is maximal monotone on $\R^d$, $J_\lambda^{q,\mu,\alpha}$ is defined on all of $\R^d$, single-valued and nonexpansive, and $A_\lambda^q(\cdot,\mu,\alpha)$ is maximal monotone and $\lambda^{-1}$-Lipschitz \cite[chap.~II]{Brezis1973}.

\begin{lemma}\label{lem:minimalgrowth}
Suppose Assumption \ref{ass:operator} holds. Then there exists a constant $C_A<\infty$ such that, for all $q\in\Q$, $x\in\R^d$, $\mu\in\Ptwo(\R^d)$, and $\alpha\in\mathbb D$,
\[
|(A^q)^0(x,\mu,\alpha)|
\le C_A(1+|x|+|\alpha|+M_2(\mu)).
\]
\end{lemma}

\begin{proof}
Set
\[
z=x+h^q(\mu,\alpha),\qquad c=c^q(\mu,\alpha),\qquad \mathcal G=\partial\phi^q(z).
\]
Then $A^q(x,\mu,\alpha)=\mathcal G+c$ and, by translation of the closed convex set $\mathcal G$,
\[
(A^q)^0(x,\mu,\alpha)=\operatorname*{argmin}_{v\in \mathcal G+c}|v|
=\proj_{\mathcal G}(-c)+c.
\]
Set $w^0=(\partial\phi^q)^0(z)=\proj_{\mathcal G}(0)$. The minimizing property of the projection gives
\[
|\proj_{\mathcal G}(-c)+c|
\le |w^0+c|
\le |w^0|+|c|.
\]
By Assumption \ref{ass:operator},
\[
|w^0|\le C_\phi(1+|z|),
\qquad
|z|\le |x|+|h^q(\mu,\alpha)|,
\]
and
\[
|h^q(\mu,\alpha)|+|c|
\le C\bigl(1+|\alpha|+M_2(\mu)\bigr).
\]
Consequently,
\begin{align*}
|(A^q)^0(x,\mu,\alpha)|
&\le C_\phi\bigl(1+|x|+|h^q(\mu,\alpha)|\bigr)+|c|\\
&\le C_A\bigl(1+|x|+|\alpha|+M_2(\mu)\bigr),
\end{align*}
with $C_A$ independent of $(q,\mu,\alpha,x)$.
\end{proof}

\subsection{Measurability of the operator family}

\begin{assumption}\label{ass:measop}
The map $(q,z)\mapsto\phi^q(z)$ is a convex normal integrand in the sense of \cite[chap.~14]{RockafellarWets1998}. For every $\lambda>0$, the maps
\[
(q,\mu,\alpha,x)\longmapsto J_\lambda^q(x,\mu,\alpha),\qquad
(q,\mu,\alpha,x)\longmapsto A_\lambda^q(x,\mu,\alpha)
\]
are Borel measurable, and the minimal selection $(q,\mu,\alpha,x)\mapsto(A^q)^0(x,\mu,\alpha)$ is Borel measurable.
\end{assumption}

\begin{remark}\label{rem:measop}
(i) The last requirement in Assumption~\ref{ass:measop} follows from the preceding ones. By Lemma~\ref{lem:yosida-minimal} below, $(A^q)^0(x,\mu,\alpha)=\lim_{n\to\infty}A_{1/n}^q(x,\mu,\alpha)$ pointwise, and a pointwise limit of Borel maps is Borel.
(ii) Under Assumption~\ref{ass:ALA} or Assumption~\ref{ass:finite-branches}, Assumption~\ref{ass:measop} is automatically satisfied. Indeed, by \eqref{eq:resolvent-translation} and Lemma~\ref{lem:resolvent-slopes}, $(q,\mu,\alpha,x)\mapsto J_\lambda^q(x,\mu,\alpha)$ is continuous because $h$ and $c$ are Lipschitz and the slopes are Lipschitz in $q$; therefore $A_\lambda^q$ is continuous as well.
(iii) If $q$, $\alpha$, and $X_-$ are predictable and $t\mapsto\mu_t$ is a deterministic Borel flow, then $(t,\omega)\mapsto(q_t,\mu_{t-},\alpha_t,X_{t-})$ is measurable from the predictable $\sigma$-field into the Borel $\sigma$-field of $\Q\times\Ptwo(\R^d)\times\mathbb D\times\R^d$, which is the product $\sigma$-field (Section~\ref{sec:base}). Composition with the Borel maps in Assumption~\ref{ass:measop} therefore yields predictable processes. This is the argument used for all source densities below.
\end{remark}

\subsection{Epistemic selections and selection systems}

\begin{definition}\label{def:selection}
Let $P\in\PP$, let $(\xi,\alpha)$ be data with $\alpha\in\mathcal A_2$, let $\mu^P$ be the flow of laws of the forward state, and let $\gamma\in(0,1]$. An \emph{admissible epistemic selection at level $\gamma$} is a predictable $\Q$-valued process $q$ such that
\[
q_t\in\Q_\gamma(\mu_{t-}^P,\alpha_t),\qquad \dd t\otimes\dd P\text{-a.e.}
\]
\end{definition}

\begin{remark}[Existence of admissible selections]\label{rem:existence-selection}
By Assumption \ref{ass:fuzzycuts} and the Kuratowski--Ryll-Nardzewski selection theorem \cite{KuratowskiRyllNardzewski1965}, for every $\gamma\in(0,1]$ there exists a Borel map $s_\gamma:\Ptwo(\R^d)\times\mathbb D\to\Q$ such that $s_\gamma(\mu,\alpha)\in\Q_\gamma(\mu,\alpha)$ for every $(\mu,\alpha)$. Since the flow $\mu^P$ is continuous (Lemma \ref{lem:flowcontinuity}) and $\alpha$ is predictable, the process $q_t:=s_\gamma(\mu_{t-}^P,\alpha_t)$ is predictable (Remark \ref{rem:measop}(iii)) and is an admissible epistemic selection at level $\gamma$. Hence the set of admissible selections is nonempty; under Assumption \ref{ass:paramclass}, it contains the class $\mathfrak A_{\gamma,M}^P$, which is nonempty by Proposition \ref{prop:paramclass}.
\end{remark}

\begin{definition}\label{def:solution}
Let $P\in\PP$, let $(\xi,\alpha)$ satisfy $\alpha\in\mathcal A_2$, let $\gamma\in(0,1]$, and let $q$ be an admissible selection at level $\gamma$. A \emph{selected solution} is a quintuple
\[
(X,Y,Z,U,K)\in\cS_P^2\times\cS_P^2\times\cH_P^2\times\cJ_P^2\times\cS_P^2
\]
such that $X,Y$ are adapted and càdlàg, $Z,U$ are predictable, $K$ is adapted, continuous, and absolutely continuous in time, and, with $\mu_t=\Law^P(X_t)$ and $\eta_t=\Law^P(Y_t)$, the identities \eqref{eq:forward}, \eqref{eq:K}, and \eqref{eq:backward} hold for every $t\in[0,T]$, up to $P$-indistinguishability. For $\lambda>0$, a \emph{regularized solution} is defined in the same way by replacing $(A^{q_s})^0$ with $A_\lambda^{q_s}$ in \eqref{eq:K}.
\end{definition}

 \begin{definition}\label{def:regular}
Let $P\in\PP$. For an $\mathbb F^P$-adapted process $V$ with continuous paths in a Euclidean space, define, componentwise,
\[
\dot V_t:=\limsup_{n\to\infty}n\bigl(V_t-V_{(t-1/n)\vee0}\bigr)\ \text{ whenever this upper limit is finite},
\qquad
\dot V_t:=0\ \text{ otherwise}.
\]
Let $(\xi,\alpha)$ be data and $\gamma\in(0,1]$. An admissible epistemic selection $q$ at level $\gamma$ is called \emph{regular} if its paths are absolutely continuous, $P$-a.s.; it is regular \emph{with energy bound} $M$ if, in addition, $\mathbb E^P\!\left[\int_0^T|\dot q_t|^2\dd t\right]\le M^2$.
\end{definition}

The process $\dot V$ is predictable: each $n(V_t-V_{(t-1/n)\vee0})$ is adapted and continuous, hence predictable, and an upper limit of predictable processes is predictable. If $V$ has absolutely continuous paths, the Lebesgue differentiation theorem shows that, $P$-a.s., $\dot V_t$ coincides with the path derivative for almost every $t$, so that $V_t=V_0+\int_0^t\dot V_s\dd s$ for every $t\in[0,T]$. All time derivatives of controls and selections are taken in this predictable version.

For the quantitative results and reconstruction from cuts, we use selection systems (Definition \ref{def:selsystem} below). For each $P$, fix a nonempty set $\mathfrak D_P$ of data $(\xi,\alpha)$ admissible for \ref{ass:F}, with common quadratic bounds independent of $P$. The properties below are imposed only on realized pairs
\[
(\mu,\alpha)=(\mu^P[\xi,\alpha],\alpha),
\qquad (\xi,\alpha)\in\mathfrak D_P,
\]
where $\mu^P[\xi,\alpha]$ is the flow of laws of the unique forward solution. Any comparison between two data sets uses two elements of $\mathfrak D_P$ under the same model $P$. In the quantitative sub-regime, $\mathfrak D_P$ is restricted to absolutely continuous controls satisfying the time-regularity bounds specified below.

\begin{definition}\label{def:selsystem}
Let $M>0$, $L_Q>0$, and $\nu\in(0,1]$. A \emph{selection system} with energy bound $M$ assigns to every $P\in\PP$, every datum $(\xi,\alpha)\in\mathfrak D_P$ with forward flow $\mu=\mu^P[\xi,\alpha]$, and every $\gamma\in(0,1]$ a set of processes $\mathfrak A_{\gamma,M}^P[\xi,\alpha]$ such that:
\begin{enumerate}[label=(S\arabic*)]
\item $\mathfrak A_{\gamma,M}^P[\xi,\alpha]$ is nonempty and compact in $(L^\infty(\dd t\otimes\dd P),\|\cdot\|_{\infty,P})$, and each of its elements is a regular admissible selection at level $\gamma$ with energy bound $M$ (Definitions \ref{def:selection} and \ref{def:regular});
\item $\mathfrak A_{\gamma_2,M}^P[\xi,\alpha]\subseteq\mathfrak A_{\gamma_1,M}^P[\xi,\alpha]$ for $0<\gamma_1\le\gamma_2\le1$;
\item $\mathfrak A_{\gamma,M}^P[\xi,\alpha]=\bigcap_{0<\iota<\gamma}\mathfrak A_{\iota,M}^P[\xi,\alpha]$;
\item for all $(\xi,\alpha),(\xi',\alpha')\in\mathfrak D_P$ with flows $\mu,\mu'$, all $\gamma,\gamma'\in(0,1]$, and every $q\in\mathfrak A_{\gamma,M}^P[\xi,\alpha]$, there exists $q'\in\mathfrak A_{\gamma',M}^P[\xi',\alpha']$ such that
\[
|q_t-q'_t|\le L_Q\bigl(|\gamma-\gamma'|^\nu+W_2(\mu_{t-},\mu'_{t-})+|\alpha_t-\alpha'_t|\bigr),\qquad\dd t\otimes\dd P\text{-a.e.};
\]
\item there exists a compact set $\overline{\mathfrak A}_M^P[\xi,\alpha]\subset L^\infty(\dd t\otimes\dd P)$ of predictable $\Q$-valued processes with absolutely continuous paths and energy bound $M$, containing $\mathfrak A_{\gamma,M}^P[\xi,\alpha]$ for every $\gamma\in(0,1]$.
\end{enumerate}
When the data are fixed, we write $\mathfrak A_{\gamma,M}^P$.
\end{definition}

All results in Sections \ref{sec:stability}--\ref{sec:yosida} are stated for an arbitrary selection system. Two instances are available. Section \ref{sec:feedback} constructs \emph{intrinsic} selection systems from the geometry of the fuzzy data: Lipschitz feedback selections of convex cuts, for which (S1)--(S3) and (S5) are proved in every dimension and (S4) when $d_q=1$, and center--radius classes, for which (S1)--(S5) are proved in every dimension. The parametric families in the following assumption provide a second, more general instance (Proposition \ref{prop:paramclass}).

\begin{assumption}\label{ass:paramclass}
There exists a nonempty compact metric space $(\mathfrak P,d_{\mathfrak P})$ and, for every $\gamma\in(0,1]$ and every realized pair $(\mu,\alpha)$ arising from data in $\mathfrak D_P$, a family $\{q^{\gamma,\mathfrak p}[\mu,\alpha]:\mathfrak p\in\mathfrak P\}$ of predictable processes with absolutely continuous paths such that:
\begin{enumerate}[label=(\roman*)]
\item $q_t^{\gamma,\mathfrak p}[\mu,\alpha]\in\Q_\gamma(\mu_{t-},\alpha_t)$, $\dd t\otimes\dd P$-a.e.;
\item $\mathfrak p\mapsto q^{\gamma,\mathfrak p}[\mu,\alpha]$ is continuous from $\mathfrak P$ into $(L^\infty(\dd t\otimes\dd P),\|\cdot\|_{\infty,P})$;
\item there exists a constant $M>0$ common to all indices: for every $P\in\PP$, every datum $(\xi,\alpha)\in\mathfrak D_P$, every $\gamma\in(0,1]$, and every $\mathfrak p\in\mathfrak P$,
\[
\mathbb E^P\!\left[\int_0^T
|\dot q_t^{\gamma,\mathfrak p}[\mu^P[\xi,\alpha],\alpha]|^2
\dd t\right]\le M^2;
\]
\item there exists a constant $C_Q<\infty$ such that, for all $\gamma,\gamma'\in(0,1]$ and $\mathfrak p\in\mathfrak P$,
\[
\|q^{\gamma,\mathfrak p}[\mu,\alpha]-q^{\gamma',\mathfrak p}[\mu,\alpha]\|_{\infty,P}\le C_Q|\gamma-\gamma'|^\nu;
\]
\item if $\gamma_1\le\gamma_2$, then for every $\mathfrak p$ there exists $\mathfrak p'$ such that $q^{\gamma_2,\mathfrak p}[\mu,\alpha]=q^{\gamma_1,\mathfrak p'}[\mu,\alpha]$;
\item for all $(\gamma,\mu,\alpha)$, $(\gamma',\mu',\alpha')$, and every $\mathfrak p$,
\begin{align}
|q_t^{\gamma,\mathfrak p}[\mu,\alpha]-q_t^{\gamma',\mathfrak p}[\mu',\alpha']|
&\le L_Q\bigl(|\gamma-\gamma'|^\nu+W_2(\mu_{t-},\mu'_{t-})+|\alpha_t-\alpha'_t|\bigr),
\label{eq:param-cross-coupling}
\end{align}
$\dd t\otimes\dd P$-a.e.
\end{enumerate}
For the flow $(\mu^P,\alpha)$ under consideration, set
\begin{equation}\label{eq:paramclass}
\mathfrak A_{\gamma,M}^P=\{q^{\gamma,\mathfrak p}[\mu^P,\alpha]:\mathfrak p\in\mathfrak P\}.
\end{equation}
\end{assumption}

\begin{remark}[Scope of the selection class]\label{rem:chosen-class}
The family $\mathfrak F=(\mathfrak P,(q^{\gamma,\mathfrak p}[\mu,\alpha])_{\gamma,\mathfrak p})$ is part of the model specification. Writing $\operatorname{Sel}_\gamma^P(\mu,\alpha)$ for the set of admissible epistemic selections at level $\gamma$ (Definition \ref{def:selection}), one has $\mathfrak A_{\gamma,M}^P\subseteq\operatorname{Sel}_\gamma^P(\mu,\alpha)$, and the inclusion is generally strict. The notation $\mathfrak A_{\gamma,M}^P$ implicitly refers to this fixed family. Two parameterizations with the same image of selections define the same class; different images may lead to different solution cuts for the same fuzzy data. For this parametric family, compactness and nesting are assumed; Section \ref{sec:feedback} proves them for intrinsic classes.
\end{remark}

\begin{proposition}\label{prop:paramclass}
Fix $P\in\PP$ and a datum $(\xi,\alpha)\in\mathfrak D_P$, with forward flow of laws $\mu^P$, and suppose Assumption \ref{ass:paramclass} holds. Then, for every $\gamma\in(0,1]$, the class $\mathfrak A_{\gamma,M}^P$ is nonempty and compact in $(L^\infty(\dd t\otimes\dd P),\|\cdot\|_{\infty,P})$; for $0<\gamma_1\le\gamma_2\le1$, $\mathfrak A_{\gamma_2,M}^P\subseteq\mathfrak A_{\gamma_1,M}^P$; and
\begin{equation}\label{eq:left-paramclass}
\mathfrak A_{\gamma,M}^P=\bigcap_{0<\iota<\gamma}\mathfrak A_{\iota,M}^P.
\end{equation}

In particular, the classes \eqref{eq:paramclass} form a selection system in the sense of Definition \ref{def:selsystem}, with the constants $M$, $L_Q$, and $\nu$ from Assumption \ref{ass:paramclass}.
\end{proposition}

\begin{proof}
Fix $P$, $\gamma$, $\mu^P$, and $\alpha$, and define
\[
\Psi_\gamma:\mathfrak P\longrightarrow L^\infty(\dd t\otimes\dd P),
\qquad
\Psi_\gamma(\mathfrak p)=q^{\gamma,\mathfrak p}[\mu^P,\alpha].
\]
By \ref{ass:paramclass}(ii), $\Psi_\gamma$ is continuous. Since $\mathfrak P$ is compact, $\mathfrak A_{\gamma,M}^P=\Psi_\gamma(\mathfrak P)$ is nonempty and compact in $L^\infty(\dd t\otimes\dd P)$.

Let $0<\gamma_1\le\gamma_2\le1$ and $q\in\mathfrak A_{\gamma_2,M}^P$. There exists $\mathfrak p\in\mathfrak P$ such that $q=q^{\gamma_2,\mathfrak p}[\mu^P,\alpha]$. By \ref{ass:paramclass}(v), there exists $\mathfrak p'\in\mathfrak P$ such that
\[
q^{\gamma_2,\mathfrak p}[\mu^P,\alpha]
=q^{\gamma_1,\mathfrak p'}[\mu^P,\alpha]
\in\mathfrak A_{\gamma_1,M}^P,
\]
hence $\mathfrak A_{\gamma_2,M}^P\subseteq\mathfrak A_{\gamma_1,M}^P$. This inclusion yields
\[
\mathfrak A_{\gamma,M}^P\subseteq\bigcap_{0<\iota<\gamma}\mathfrak A_{\iota,M}^P.
\]

For the reverse inclusion, choose a sequence $\iota_n\uparrow\gamma$ and
\[
q\in\bigcap_{0<\iota<\gamma}\mathfrak A_{\iota,M}^P.
\]
For each $n$, there exists $\mathfrak p_n\in\mathfrak P$ such that $q=q^{\iota_n,\mathfrak p_n}[\mu^P,\alpha]$. Compactness of $\mathfrak P$ yields a subsequence, still denoted $(\mathfrak p_n)$, and $\mathfrak p\in\mathfrak P$ such that $d_{\mathfrak P}(\mathfrak p_n,\mathfrak p)\longrightarrow0$. Then
\begin{align*}
\|q-q^{\gamma,\mathfrak p}[\mu^P,\alpha]\|_{\infty,P}
&\le
\|q^{\iota_n,\mathfrak p_n}[\mu^P,\alpha]-q^{\gamma,\mathfrak p_n}[\mu^P,\alpha]\|_{\infty,P}\\
&\quad+
\|q^{\gamma,\mathfrak p_n}[\mu^P,\alpha]-q^{\gamma,\mathfrak p}[\mu^P,\alpha]\|_{\infty,P}\\
&\le C_Q|\gamma-\iota_n|^\nu
+\|\Psi_\gamma(\mathfrak p_n)-\Psi_\gamma(\mathfrak p)\|_{\infty,P}
\longrightarrow0.
\end{align*}
Thus $q=q^{\gamma,\mathfrak p}[\mu^P,\alpha]\in\mathfrak A_{\gamma,M}^P$, proving \eqref{eq:left-paramclass}.

We now verify that these classes form a selection system. (S1)--(S3) are the assertions already proved, together with items (i) and (iii) of Assumption \ref{ass:paramclass}. For (S4), given $q=q^{\gamma,\mathfrak p}[\mu,\alpha]$, the process $q'=q^{\gamma',\mathfrak p}[\mu',\alpha']$ satisfies the required bound by \eqref{eq:param-cross-coupling}. For (S5), item (iv) shows that, for every $\mathfrak p$, $q^{\gamma,\mathfrak p}[\mu^P,\alpha]$ converges in $L^\infty(\dd t\otimes\dd P)$, as $\gamma\downarrow0$, to a limit $q^{0,\mathfrak p}$, with $\|q^{\gamma,\mathfrak p}-q^{0,\mathfrak p}\|_{\infty,P}\le C_Q\gamma^\nu$ uniformly in $\mathfrak p$. Hence the map $\overline\Psi(\gamma,\mathfrak p)=q^{\gamma,\mathfrak p}[\mu^P,\alpha]$, extended by $q^{0,\mathfrak p}$ at $\gamma=0$, satisfies
\[
\|\overline\Psi(\gamma,\mathfrak p)-\overline\Psi(\bar\gamma,\bar{\mathfrak p})\|_{\infty,P}\le C_Q|\gamma-\bar\gamma|^\nu+\|\overline\Psi(\bar\gamma,\mathfrak p)-\overline\Psi(\bar\gamma,\bar{\mathfrak p})\|_{\infty,P},
\]
and is continuous on the compact set $[0,1]\times\mathfrak P$ (for $\bar\gamma=0$, the last term is a uniform limit of continuous functions of $\mathfrak p$). Its image $\overline{\mathfrak A}_M^P$ is therefore compact and contains every $\mathfrak A_{\gamma,M}^P$. Its elements are predictable and $\Q$-valued. Fix $\mathfrak p$ and $\gamma_n\downarrow0$, and write $q^n=q^{\gamma_n,\mathfrak p}$. Since the $q^n$ are continuous, the bound $|q_t^n-q_t^{n'}|\le\|q^n-q^{n'}\|_{\infty,P}$, valid $\dd t\otimes\dd P$-a.e., extends, $P$-a.s., to every $t\in[0,T]$. Thus the sequence is uniformly Cauchy on $[0,T]$, $P$-a.s., and its limit provides a continuous adapted, hence predictable, representative of $q^{0,\mathfrak p}$, with $q_t^n\to q_t^{0,\mathfrak p}$ in $L^2(P)$ for every $t\in[0,T]$. The derivatives $\dot q^n$ are bounded by $M$ in $L^2_{\rm pred}(\dd t\otimes\dd P;\R^{d_q})$; a subsequence converges weakly to some $v$. For fixed $t$, the map $w\mapsto\int_0^tw_s\dd s$ is linear and bounded from $L^2(\dd s\otimes\dd P)$ into $L^2(P)$, hence weakly continuous, and the identity $q_t^n-q_0^n=\int_0^t\dot q_s^n\dd s$ passes to the limit: $q_t^{0,\mathfrak p}-q_0^{0,\mathfrak p}=\int_0^tv_s\dd s$ in $L^2(P)$. This equality holds $P$-a.s. simultaneously for all rational $t$, and then for every $t\in[0,T]$ by continuity of both sides. Therefore $q^{0,\mathfrak p}$ has absolutely continuous paths and, by weak lower semicontinuity of the norm, energy bound $M$.
\end{proof}

\begin{remark}\label{rem:noncompact}
The parameterization is structural. The class of all predictable admissible processes satisfying only an average bound on $\dot q$ is not compact in the norm $\|\cdot\|_{\infty,P}$ on a nonatomic space. Indeed, let $\bar q$ be a regular admissible selection with $\mathbb E^P\!\left[\int_0^T|\dot{\bar q}_t|^2\dd t\right]<\infty$, let $\mathbf e_1$ be a unit vector in $\R^{d_q}$, and let $\varepsilon_0>0$ be such that $\bar q_t+s\mathbf e_1\in\Q_\gamma(\mu_{t-},\alpha_t)$ for every $s\in[0,\varepsilon_0/2]$, $\dd t\otimes\dd P$-a.e.; let $A_n\in\mathcal F_{T/2}$ be disjoint events with $P(A_n)=2^{-n}$. The perturbations
\[
q_t^n=\bar q_t+\varepsilon_0T^{-1}(t-T/2)^+\1_{A_n}\mathbf e_1
\]
are continuous and adapted, hence predictable, and admissible because $0\le\varepsilon_0T^{-1}(t-T/2)^+\le\varepsilon_0/2$. They satisfy
\[
\sup_n\mathbb E^P\!\left[\int_0^T|\dot q_t^n|^2\dd t\right]\le2\,\mathbb E^P\!\left[\int_0^T|\dot{\bar q}_t|^2\dd t\right]+\frac{\varepsilon_0^2}{T},
\]
whereas, for $n\ne n'$, $|q_t^n-q_t^{n'}|=\varepsilon_0T^{-1}(t-T/2)^+$ on $A_n\cup A_{n'}$, so that $\|q^n-q^{n'}\|_{\infty,P}=\varepsilon_0/2$. Section \ref{sec:feedback} shows that Lipschitz feedback selections provide a canonical compact substitute.
\end{remark}

\begin{remark}[Structural constants]\label{rem:constants}
The data class $\mathfrak D_P$ satisfies the common bounds
\[
\mathfrak M_\xi:=\sup_{P\in\PP}\sup_{(\xi,\alpha)\in\mathfrak D_P}\mathbb E^P[|\xi|^2]<\infty,
\qquad
\mathfrak M_\alpha:=\sup_{P\in\PP}\sup_{(\xi,\alpha)\in\mathfrak D_P}\mathbb E^P\!\left[\int_0^T|\alpha_t|^2\dd t\right]<\infty .
\]
We denote by $\mathfrak C_{\rm str}$ the collection consisting of $T$, the dimensions $d,n,d_q,d_\alpha$, the matrices $\underline a,\overline a$, the constants $L_F,C_F,L_B,C_B,L_h,L_c,C_\phi,C_{hc}$, the constants $L_Q,\nu$ of the selection system (Definition~\ref{def:selsystem}), the constant $C_Q$ from Assumption~\ref{ass:paramclass} when that assumption is used, and the bounds $\mathfrak M_\xi,\mathfrak M_\alpha$. In the quantitative sub-regime it also contains $\ell$, $\max_j\|\psi_j\|_{C_b^2}$, the bound on the first derivatives of $\widehat h$, $C_{\dot\alpha}$, $\underline\sigma$, $\alpha_A$, $\underline\kappa,\overline\kappa$, $L_\kappa$, $C_c$, and $N_\Theta,\Theta_*,\overline v,L_v$ from Assumption~\ref{ass:finite-branches}. A constant denoted by $C(M;\mathfrak C_{\rm str})$, or simply $C$ in quantitative results, depends only on $M$ and $\mathfrak C_{\rm str}$. In particular, it does not depend on $P$, the datum in $\mathfrak D_P$, $\gamma$, $\mathfrak p$, the selection, or $\lambda$. In uniform statements, we use the conventions $\sup_{P,\gamma,q}$ and $\sup_{P,\gamma}$ from the Introduction; when a statement fixes $P$ or the data, the corresponding supremum is omitted. We fix a reference occupation width $r_0=1$. Intermediate local-time estimates remain valid for an arbitrary fixed spatial radius, notably when the interfaces have several nonzero thresholds.
\end{remark}

\subsection{Quantitative regime: motion of the interface}

Lipschitz continuity of $h^q$ in $W_2$ alone does not control how the interface moves in time. For the occupation estimates below, we therefore impose $L^2$ bounds on the relevant time derivatives; these bounds control the total variation of the interface shift up to second moments.

\begin{assumption}\label{ass:finitestat}
There exist an integer $\ell\ge1$, functions $\psi_j\in C_b^2(\R^d)$, and a map $\widehat h:\Q\times\mathbb D\times\R^\ell\longrightarrow\R^d$ that is the restriction of a $C^1$ map defined on an open neighborhood of $\Q\times\mathbb D\times\R^\ell$, with uniformly bounded first derivatives, such that
\[
h^q(\mu,\alpha)
=
\widehat h\!\left(q,\alpha,m_1(\mu),\ldots,m_\ell(\mu)\right),
\qquad
m_j(\mu)=\int_{\R^d}\psi_j(x)\mu(\dd x).
\]
Set $B_\psi=\bigl(\sum_{j=1}^\ell\|D\psi_j\|_\infty^2\bigr)^{1/2}$.

The control and the regular selections admit versions with absolutely continuous paths: for every $P\in\PP$ and every regular selection,
\begin{equation}\label{eq:AC-all-t}
\alpha_t=\alpha_0+\int_0^t\dot\alpha_s\dd s,\qquad
q_t=q_0+\int_0^t\dot q_s\dd s,
\qquad P\text{-a.s., for every }t\in[0,T],
\end{equation}
and the derivatives, defined for $\dd t$-almost every $t$, satisfy
\[
\sup_{P,\gamma,q}
\mathbb E^P\!\left[\int_0^T|\dot q_t|^2\dd t\right]\le M^2,
\qquad
\sup_{P\in\PP}\sup_{(\xi,\alpha)\in\mathfrak D_P}\mathbb E^P\!\left[\int_0^T|\dot\alpha_t|^2\dd t\right]
\le C_{\dot\alpha}<\infty.
\]
The bound $M$ is the energy bound of the selection system (Definition \ref{def:selsystem}); $C_{\dot\alpha}$ is a structural constant.
\end{assumption}

Absolute continuity further implies
\begin{equation}\label{eq:alpha0-bound}
\mathbb E^P\!\left[|\alpha_0|^2\right]\le \frac{2}{T}\mathbb E^P\!\left[\int_0^T|\alpha_t|^2\dd t\right]+2T \mathbb E^P\!\left[\int_0^T|\dot\alpha_t|^2\dd t\right]<\infty.
\end{equation}

\begin{remark}[Temporal scope of the rates]\label{rem:time-scope}
The occupation results and the rates derived from them use Assumption \ref{ass:finitestat}. Lipschitz feedback alone is not sufficient. For example, if $X_t=B_t^P$ ($d=n=1$), $\alpha_t=X_t$, and $h^q(\mu,\alpha)=-\alpha$, then $S_t=X_t+h^q(\mu_t,\alpha_t)=X_t-X_t=0$ for every $t\in[0,T]$, so $\int_0^T\1_{\{|S_t|\le r\}}\dd t=T$ for every $r>0$, despite the non-degeneracy of $X$. An extension to semimartingale controls would therefore require assumptions on the martingale part of $S$.
\end{remark}

\begin{remark}
The occupation estimates use only second moments of the jumps: the quadratic assumptions in Assumption \ref{ass:F}, together with $\psi_j\in C_b^2$, control the derivative of the law statistics in $L^1(\dd t)$, and the convex jump remainder in Tanaka's formula is nonnegative.
\end{remark}

\subsection{Finite-branch potentials; the ALA prototype}\label{sec:ALA}

\begin{assumption}[Finite-branch potentials]\label{ass:finite-branches}
For $q\in\Q$,
\[
\phi^q(s)=\frac{\alpha_A}{2}|s|^2+\sum_{k=1}^d\mathfrak g_k^q(s_k),\qquad\alpha_A>0\ \text{independent of }q,
\]
where, for each $k$, there exists a finite nonempty set of thresholds $\Theta_k=\{\theta_{k,1}<\cdots<\theta_{k,N_k}\}$, independent of $q$, such that $\mathfrak g_k^q$ is convex, continuous, and affine on each connected component $I_{k,0},\ldots,I_{k,N_k}$ of $\R\setminus\Theta_k$, ordered increasingly. The slopes $v_{k,j}^q$ of $\mathfrak g_k^q$ on $I_{k,j}$ satisfy
\[
|v_{k,j}^q|\le\overline v,
\qquad
|v_{k,j}^q-v_{k,j}^{q'}|\le L_v|q-q'|,
\]
and $(q,s)\mapsto \mathfrak g_k^q(s)$ is Borel measurable. Finally, $\sup_{q,\mu,\alpha}|c^q(\mu,\alpha)|\le C_c$.
\end{assumption}

Our main example is the asymmetric loss-aversion (ALA) prototype.

\begin{assumption}\label{ass:ALA}
For $q\in\Q$,
\[
\phi^q(s)=\sum_{k=1}^d\{\kappa_+^q(s_k)^++\kappa_-^q(s_k)^-\}
+\frac{\alpha_A}{2}|s|^2,
\]
with
\[
0<\underline\kappa\le\kappa_+^q,\kappa_-^q\le\overline\kappa,
\qquad \alpha_A>0.
\]
The quadratic coefficient $\alpha_A$ does not depend on $q$; there exist $L_\kappa,C_c<\infty$ such that
\[
|\kappa_\pm^q-\kappa_\pm^{q'}|\le L_\kappa|q-q'|\quad(q,q'\in\Q),
\qquad
\sup_{q\in\Q,\ \mu\in\Ptwo(\R^d),\ \alpha\in\mathbb D}|c^q(\mu,\alpha)|\le C_c .
\]
Both the interface and the nonsmooth part remain $q$-dependent.
\end{assumption}

Assumption \ref{ass:ALA} is the special case $\Theta_k=\{0\}$, with slopes $v_{k,0}^q=-\kappa_-^q$ and $v_{k,1}^q=\kappa_+^q$, $\overline v=\overline\kappa$, and $L_v=L_\kappa$. By convexity, $\partial \mathfrak g_k^q(s)$ is the singleton $\{v_{k,j}^q\}$ if $s\in I_{k,j}$, and the interval $[v_{k,j-1}^q,v_{k,j}^q]$ if $s=\theta_{k,j}$; hence all its elements are bounded by $\overline v$. In particular, $|(\partial\phi^q)^0(z)|\le\alpha_A|z|+\sqrt d\,\overline v$, and Assumption \ref{ass:operator} is satisfied. Assumption \ref{ass:measop} is also satisfied: the resolvent is continuous in $(q,x)$ by Lemma \ref{lem:resolvent-slopes}, and the minimal selection is Borel because it is given branch by branch by continuous formulas. Set
\[
N_\Theta=\sum_{k=1}^dN_k,
\qquad
\Theta_*=\max_{k,j}|\theta_{k,j}|,
\qquad
\rho^q(x,\mu,\alpha)=\min_{1\le k\le d}\ \min_{1\le j\le N_k}|x_k+h_k^q(\mu,\alpha)-\theta_{k,j}|.
\]

\subsection{Normal non-degeneracy and the quantitative regime}\label{sec:Hquant}

\begin{assumption}\label{ass:elliptic}
For each $k$,
\[
\mathbf e_k^\top\sigma(t,x,\mu,\alpha)a_t^P\sigma(t,x,\mu,\alpha)^\top \mathbf e_k
\ge\underline\sigma^2>0
\]
for every $P\in\PP$, $\dd t\otimes\dd P$-a.e., and all $(x,\mu,\alpha)\in\R^d\times\Ptwo(\R^d)\times\mathbb D$ (see Remark~\ref{rem:random-norms}).
\end{assumption}

We denote by $(H_{\rm quant})$ the collection of Assumptions \ref{ass:base}, \ref{ass:F}, \ref{ass:B}, \ref{ass:fuzzycuts}, \ref{ass:operator}, \ref{ass:measop}, \ref{ass:finitestat}, \ref{ass:finite-branches}, and \ref{ass:elliptic}, on the probabilistic basis of Section~\ref{sec:cadre}, together with a selection system $(\mathfrak A_{\gamma,M}^P)$ with energy bound $M$ (Definition~\ref{def:selsystem}), fixed once and for all. The parametric classes of Assumption~\ref{ass:paramclass} (Proposition~\ref{prop:paramclass}) and the intrinsic classes of Section~\ref{sec:feedback} are instances of this setting. The occupation results below use only the assumptions explicitly listed in their statements.
\section{Well-posedness of the selected systems}\label{sec:bienpose}

The triangular system is solved through the composition
\[
(\xi,\alpha)\longmapsto(X^P,\mu^P)\longmapsto K^{P,q}\longmapsto(Y^{P,q},Z^{P,q},U^{P,q}),
\]
The forward state is solved first, then the selected source, and finally the BSDE. None of these steps uses occupation estimates, compactness of the cuts, or the ALA geometry. Because the well-posedness arguments are standard, we record them in a condensed form while keeping the constants needed later. The estimate that matters for the quantitative part of the paper is the integrated-source stability bound in Lemma~\ref{lem:energy}.

\subsection{Forward McKean--Vlasov equation with jumps}

\begin{theorem}\label{thm:forward}
Fix $P\in\PP$ and suppose Assumptions \ref{ass:base} and \ref{ass:F} hold. Let $\xi\in L^2(\Omega,\mathcal F_0,P;\R^d)$ and $\alpha\in\mathcal A_2$. Then equation \eqref{eq:forward} has a unique strong solution $X^P\in\cS_P^2$, and there exists a constant $C$, depending only on $T$, $\overline a$, $L_F$, and $C_F$, such that
\begin{equation}\label{eq:forward-apriori}
\mathbb E^P\!\left[\sup_{t\le T}|X_t^P|^2\right]
\le
C\left(
1+\mathbb E^P\!\left[|\xi|^2\right]+\mathbb E^P\!\left[\int_0^T|\alpha_t|^2\dd t\right]
\right).
\end{equation}
\end{theorem}

\begin{proof}
The result is classical for McKean--Vlasov SDEs with coefficients Lipschitz in $(x,\mu)$ driven by a continuous martingale and a compensated Poisson random measure; see \cite{Sznitman1991,CarmonaDelarue2018} for the continuous case and \cite{Li2018} for the case with jumps. We recall the argument because its constants are used later.

\emph{Step 1: frozen law.} Let $\mathcal M_2$ be the set of flows $\upsilon\in C([0,T];\Ptwo(\R^d))$ with $\upsilon_0=\Law^P(\xi)$, endowed with the complete metric $d_{\omega}(\upsilon,\upsilon')^2=\sup_{t\le T}e^{-\omega t}W_2^2(\upsilon_t,\upsilon'_t)$, $\omega>0$. For $\upsilon\in\mathcal M_2$, the equation obtained by replacing $\mu_{s-}$ with $\upsilon_s$ in \eqref{eq:forward} has predictable coefficients that are Lipschitz in $x$, uniformly in $(t,\omega)$, and square integrable at $x=0$ with respect to $\dd t\otimes\dd P$ (Assumption~\ref{ass:F}); hence it has a unique solution $X^\upsilon\in\cS_P^2$, see for instance \cite[Chap.~IV]{IkedaWatanabe1989}. Let $\upsilon,\upsilon'\in\mathcal M_2$. Both solutions start from $\xi$ and are driven by the same $(B^P,\widetilde N^P)$. The inequality $(x_1+x_2+x_3)^2\le3(x_1^2+x_2^2+x_3^2)$, Cauchy--Schwarz for the drift, Doob's inequality with the isometries for the two stochastic integrals, and Assumption~\ref{ass:F} give
\[
\mathbb E^P\!\left[\sup_{r\le t}|X_r^\upsilon-X_r^{\upsilon'}|^2\right]\le3\max(T,4)L_F^2\int_0^t\Bigl(\mathbb E^P\!\left[\sup_{u\le s}|X_u^\upsilon-X_u^{\upsilon'}|^2\right]+W_2^2(\upsilon_s,\upsilon'_s)\Bigr)\dd s .
\]
Gronwall's lemma then yields, with $C_T=3\max(T,4)L_F^2\exp\bigl(3\max(T,4)L_F^2T\bigr)$,
\[
\mathbb E^P\!\left[\sup_{r\le t}|X_r^\upsilon-X_r^{\upsilon'}|^2\right]
\le C_T\int_0^t W_2^2(\upsilon_s,\upsilon'_s)\dd s,
\qquad t\in[0,T].
\]
The same inequalities show that $t\mapsto X_t^\upsilon$ is continuous in $L^2(P)$, so that $\Gamma(\upsilon)_t:=\Law^P(X_t^\upsilon)$ defines a map $\Gamma:\mathcal M_2\to\mathcal M_2$.

\emph{Step 2: fixed point and uniqueness.} The coupling $(X_t^\upsilon,X_t^{\upsilon'})$ and the preceding estimate give
\[
e^{-\omega t}W_2^2(\Gamma(\upsilon)_t,\Gamma(\upsilon')_t)
\le C_Td_\omega(\upsilon,\upsilon')^2e^{-\omega t}\int_0^te^{\omega s}\dd s
\le\frac{C_T}{\omega}d_\omega(\upsilon,\upsilon')^2 .
\]
For $\omega=2C_T$, $\Gamma$ is a contraction; its unique fixed point $\mu$ yields the solution $X^P=X^\mu$, with $\mu_t=\Law^P(X_t^P)$. If $X$ and $\bar X$ are two solutions, the natural coupling $W_2^2(\Law^P(X_s),\Law^P(\bar X_s))\le\mathbb E^P[|X_s-\bar X_s|^2]$ and the same estimates give $\mathbb E^P[\sup_{r\le t}|X_r-\bar X_r|^2]\le C\int_0^t\mathbb E^P[\sup_{u\le s}|X_u-\bar X_u|^2]\dd s$, and uniqueness follows from Gronwall's lemma.

\emph{Step 3: quadratic estimate.} Taking $(x',\mu',\alpha')=(0,\mathfrak d_0,0)$ in Assumption \ref{ass:F} and using $W_2(\mu,\mathfrak d_0)=M_2(\mu)$ gives the linear-growth bound
\[
|b(t,x,\mu,\alpha)|^2+|\sigma(t,x,\mu,\alpha)|_{a_t^P}^2+\int_E|\beta(t,x,\mu,\alpha,e)|^2\Pi(\dd e)\le2C_F^2+2L_F^2\bigl(|x|^2+M_2(\mu)^2+|\alpha|^2\bigr).
\]
This bound, $M_2(\mu_s)^2=\mathbb E^P[|X_s|^2]$, the inequalities of Step~1, and Gronwall's lemma yield \eqref{eq:forward-apriori}. All constants depend only on $T$, $L_F$, $C_F$, and $\overline a$, hence are uniform in $P$.
\end{proof}

\begin{lemma}\label{lem:flowcontinuity}
Fix $P\in\PP$ and suppose \ref{ass:base} and \ref{ass:F} hold. Let $\xi\in L^2(\Omega,\mathcal F_0,P;\R^d)$ and $\alpha\in\mathcal A_2$, and let $X^P$ be the solution of \eqref{eq:forward} given by Theorem \ref{thm:forward}, with $\mu_t^P=\Law^P(X_t^P)$. Then, for every $t\in(0,T]$, $P(X_t^P\ne X_{t-}^P)=0$. Moreover, the flow $t\mapsto\mu_t^P$ is continuous from $[0,T]$ into $(\Ptwo(\R^d),W_2)$, and $\mu_{t-}^P=\mu_t^P=\Law^P(X_{t-}^P)$ for every $t\in(0,T]$.
\end{lemma}

\begin{proof}
For $0\le s<t\le T$, Cauchy--Schwarz for the drift and the isometries for the two stochastic integrals give
\[
\mathbb E^P\bigl[|X_t^P-X_s^P|^2\bigr]\le3\,\mathbb E^P\!\left[\int_s^t\Bigl(T|b_u|^2+|\sigma_u|_{a_u^P}^2+\int_E|\beta_u(e)|^2\Pi(\dd e)\Bigr)\dd u\right],
\]
where $b_u$, $\sigma_u$, and $\beta_u(e)$ denote the coefficients evaluated along $X^P$. By the linear-growth bound in the proof of Theorem~\ref{thm:forward} and \eqref{eq:forward-apriori}, the integrand is integrable on $[0,T]\times\Omega$; hence the right-hand side tends to zero as $t-s\to0$, and the natural coupling yields $W_2(\mu_s^P,\mu_t^P)\to0$. Fix $t\in(0,T]$ and let $s_n\uparrow t$ with $s_n<t$. Then $X_{s_n}^P\to X_{t-}^P$, $P$-a.s., and $X_{s_n}^P\to X_t^P$ in $L^2(P)$; hence $X_t^P=X_{t-}^P$, $P$-a.s. Consequently $\Law^P(X_{t-}^P)=\mu_t^P$, and continuity of the flow gives $\mu_{t-}^P=\mu_t^P$.
\end{proof}

In particular, $\mu_{s-}^P$ may be replaced by $\mu_s^P$ in \eqref{eq:forward}, \eqref{eq:K}, and \eqref{eq:backward} without changing these equations; this also shows that the fixed point constructed in the proof of Theorem~\ref{thm:forward}, which solves the equation with $\upsilon_s=\mu_s$, solves \eqref{eq:forward}.

\subsection{Integrability of the selected source}

\begin{proposition}\label{prop:K}
Fix $P\in\PP$ and suppose Assumptions \ref{ass:base}, \ref{ass:F}, \ref{ass:operator}, and \ref{ass:measop} hold. Let $\xi\in L^2(\Omega,\mathcal F_0,P;\R^d)$, $\alpha\in\mathcal A_2$, let $X^P$ be the solution of \eqref{eq:forward} given by Theorem \ref{thm:forward}, set $\mu_t^P=\Law^P(X_t^P)$, let $\gamma\in(0,1]$, and let $q$ be an admissible epistemic selection at level $\gamma$ (Definition \ref{def:selection}). Then the process
\[
\chi_t^{P,q}
:=
(A^{q_t})^0(X_{t-}^P,\mu_{t-}^P,\alpha_t),\qquad t\in[0,T],
\]
is predictable and belongs to $L^2(\dd t\otimes\dd P;\R^d)$, and the process $K_t^{P,q}=\int_0^t\chi_s^{P,q}\dd s$ is adapted, continuous, of finite variation, and belongs to $\cS_P^2$. Moreover,
\begin{align}
&\mathbb E^P\!\left[\int_0^T|\chi_t^{P,q}|^2\dd t\right]
+\mathbb E^P\!\left[\sup_{t\le T}|K_t^{P,q}|^2\right] \notag\\
&\qquad\le
C\left(
1+\mathbb E^P\!\left[|\xi|^2\right]+\mathbb E^P\!\left[\int_0^T|\alpha_t|^2\dd t\right]
\right),
\label{eq:K-apriori}
\end{align}
where $C$ is independent of $P$, $\gamma$, and $q$.
\end{proposition}

\begin{proof}
By Assumption \ref{ass:measop}, the map $(q,\mu,\alpha,x)\longmapsto(A^q)^0(x,\mu,\alpha)$ is Borel measurable. The processes $q$, $\alpha$, and $X_{-}^P$ are predictable, and the flow $t\mapsto\mu_t^P$ is deterministic and Borel. Hence $\chi^{P,q}$ is predictable (Remark~\ref{rem:measop}(iii)). Lemma \ref{lem:minimalgrowth} gives
\[
|\chi_t^{P,q}|^2
\le
C\left(
1+|X_{t-}^P|^2+|\alpha_t|^2+M_2(\mu_{t-}^P)^2
\right),
\]
and $M_2(\mu_t^P)^2=\mathbb E^P[|X_t^P|^2]$; integrating in time, taking expectations, and using \eqref{eq:forward-apriori} gives the first part of \eqref{eq:K-apriori}. The second part follows from
\[
\sup_{t\le T}|K_t^{P,q}|^2
\le
\left(\int_0^T|\chi_s^{P,q}|\dd s\right)^2
\le
T\int_0^T|\chi_s^{P,q}|^2\dd s.
\]
Finally, $K^{P,q}$ is continuous and $\Var_{[0,T]}(K^{P,q})=\int_0^T|\chi_t^{P,q}|\dd t<\infty$, $P$-a.s.
\end{proof}

\subsection{McKean--Vlasov BSDE with a selected monotone source}

In the ALA case (Assumption \ref{ass:ALA}), set, for $k=1,\ldots,d$,
\[
S_{t-}^{P,q,k}=X_{t-}^{P,k}+h_k^{q_t}(\mu_{t-}^P,\alpha_t),
\qquad c_{t,k}=c_k^{q_t}(\mu_{t-}^P,\alpha_t).
\]
The minimal selection entering $K_t^{P,q}=\int_0^t\chi_s^{P,q}\dd s$ is
\[
\chi_{t,k}^{P,q}=
\begin{cases}
\alpha_AS_{t-}^{P,q,k}+c_{t,k}+\kappa_+^{q_t},&S_{t-}^{P,q,k}>0,\\
\alpha_AS_{t-}^{P,q,k}+c_{t,k}-\kappa_-^{q_t},&S_{t-}^{P,q,k}<0,\\
c_{t,k}+\proj_{[-\kappa_-^{q_t},\kappa_+^{q_t}]}(-c_{t,k}),
&S_{t-}^{P,q,k}=0.
\end{cases}
\]
Indeed, away from the interface the subdifferential is a singleton; on the interface, minimizing $|c_{t,k}+v|$ over $v\in[-\kappa_-^{q_t},\kappa_+^{q_t}]$ gives the displayed projection. Separation of the coordinates justifies this minimization for the Euclidean norm in $\R^d$. Predictability and integrability of this source follow from Proposition \ref{prop:K}, including on the interface.

The ALA geometry only enters later, in source stability and the Yosida rates.

\begin{theorem}\label{thm:backward}
Fix $P\in\PP$ and suppose Assumptions \ref{ass:base}, \ref{ass:F}, \ref{ass:B}, \ref{ass:operator}, and \ref{ass:measop} hold. Let $\xi\in L^2(\Omega,\mathcal F_0,P;\R^d)$, $\alpha\in\mathcal A_2$, $\gamma\in(0,1]$, and let $q$ be an admissible epistemic selection at level $\gamma$, with forward solution $X^P$, flow $\mu^P$, and source $K^{P,q}$ as in Proposition \ref{prop:K}. Then \eqref{eq:backward} has a unique solution
\[
(Y^{P,q},Z^{P,q},U^{P,q})
\in
\cS_P^2\times\cH_P^2\times\cJ_P^2,
\]
and there exists a constant $C$, independent of $P$, $\gamma$, and $q$, such that
\begin{align}
&\mathbb E^P\!\left[
\sup_{t\le T}|Y_t^{P,q}|^2
+\int_0^T|Z_t^{P,q}|_{a_t^P}^2\dd t
+\int_0^T\|U_t^{P,q}\|_\Pi^2\dd t
+\sup_{t\le T}|K_t^{P,q}|^2
\right]\notag\\
&\le
C\left(
1+\mathbb E^P\!\left[|\xi|^2\right]+\mathbb E^P\!\left[\int_0^T|\alpha_t|^2\dd t\right]
\right).
\label{eq:backward-apriori}
\end{align}
\end{theorem}

\begin{proof}
We omit the indices $(P,q)$ and set $\zeta=g(X_T,\mu_T)$, $\chi_t=(A^{q_t})^0(X_{t-},\mu_{t-},\alpha_t)$, $K_t=\int_0^t\chi_s\dd s$, and $L=L_B$. By Theorem \ref{thm:forward}, Assumption \ref{ass:B}, and Proposition \ref{prop:K}, $\zeta\in L^2(P)$ and $\chi\in L^2(\dd t\otimes\dd P)$. Since $K_T-K_t=\int_t^T\chi_s\dd s$ and $\chi$ does not depend on $(Y,Z,U)$, the density $\chi$ enters \eqref{eq:backward} as an exogenous square-integrable term of the generator. The argument is then the standard one for mean-field BSDEs with jumps \cite{Royer2006,Li2018}.

\emph{Step 1: Picard map.} Let $\mathcal H_P^2=L^2_{\rm pred}(\dd t\otimes\dd P;\R^d)\times\cH_P^2\times\cJ_P^2$. For $(y,z,u)\in\mathcal H_P^2$, the frozen generator
\[
\mathsf f_s=f\bigl(s,X_{s-},y_s,z_s,u_s,\mu_{s-},\Law^P(y_s),\alpha_s\bigr)+\chi_s
\]
is well defined by Lemma \ref{lem:lawflows} and belongs to $L^2(\dd t\otimes\dd P)$ by the linear growth of $f$, since $M_2(\Law^P(y_s))^2=\mathbb E^P[|y_s|^2]$. Predictable representation (Assumption \ref{ass:base}(iii)) applied to the martingale $\mathbb E^P[\zeta+\int_0^T\mathsf f_s\dd s\mid\mathcal F_t^P]$ yields a unique triplet $(Y',Z',U')\in\cS_P^2\times\cH_P^2\times\cJ_P^2$ such that
\[
Y'_t=\zeta+\int_t^T\mathsf f_s\dd s-\int_t^TZ'_s\dd B_s^P-\int_t^T\int_EU'_s(e)\widetilde N^P(\dd s,\dd e).
\]
Set $\Phi(y,z,u)=(Y'_-,Z',U')$, where $Y'_-$ is the left-limit process, with $Y'_{0-}=Y'_0$. This process is predictable and coincides with $Y'$ $\dd t\otimes\dd P$-a.e., since a càdlàg path has only countably many jumps; therefore the norms and time integrals below are unchanged.

\emph{Step 2: contraction.} Let $(y,z,u)$ and $(\bar y,\bar z,\bar u)$ be two inputs of $\mathcal H_P^2$, with frozen generators $\mathsf f$ and $\bar{\mathsf f}$, and images $(Y',Z',U')$ and $(Y'',Z'',U'')$. Set
\[
\delta Y_t=Y'_t-Y''_t,\qquad \delta Z_t=Z'_t-Z''_t,\qquad \delta U_t=U'_t-U''_t,\qquad \delta\mathsf f_t=\mathsf f_t-\bar{\mathsf f}_t .
\]
The source $\chi$ cancels in $\delta\mathsf f$. By Assumption \ref{ass:B} and the natural coupling,
\begin{equation}\label{eq:G-Picard}
\mathbb E^P[|\delta\mathsf f_t|^2]\le8L^2\,\mathbb E^P\bigl[|y_t-\bar y_t|^2+|z_t-\bar z_t|_{a_t^P}^2+\|u_t-\bar u_t\|_\Pi^2\bigr].
\end{equation}
For $\varpi>0$, Itô's formula applied to $e^{\varpi t}|\delta Y_t|^2$, with $\delta Y_T=0$, gives, after taking expectations and dropping the nonnegative term $\mathbb E^P[|\delta Y_0|^2]$ (the stochastic integrals are true martingales, since $\delta Y\in\cS_P^2$ and $\delta Z,\delta U$ are square integrable),
\begin{align*}
\mathbb E^P\!\left[\int_0^Te^{\varpi t}\bigl(\varpi|\delta Y_t|^2+|\delta Z_t|_{a_t^P}^2+\|\delta U_t\|_\Pi^2\bigr)\dd t\right]
&\le2\,\mathbb E^P\!\left[\int_0^Te^{\varpi t}\langle \delta Y_t,\delta\mathsf f_t\rangle\dd t\right]\\
&\le\frac{\varpi}{2}\mathbb E^P\!\left[\int_0^Te^{\varpi t}|\delta Y_t|^2\dd t\right]+\frac{2}{\varpi}\mathbb E^P\!\left[\int_0^Te^{\varpi t}|\delta\mathsf f_t|^2\dd t\right].
\end{align*}
Consider the norm $\|(y,z,u)\|_\varpi^2=\mathbb E^P\!\left[\int_0^Te^{\varpi t}(|y_t|^2+|z_t|_{a_t^P}^2+\|u_t\|_\Pi^2)\dd t\right]$, which is equivalent to the usual norm. The preceding inequality bounds $\mathbb E^P\!\left[\int_0^Te^{\varpi t}|\delta Y_t|^2\dd t\right]$ by $4\varpi^{-2}\mathbb E^P\!\left[\int_0^Te^{\varpi t}|\delta\mathsf f_t|^2\dd t\right]$ and the $(\delta Z,\delta U)$ part by $2\varpi^{-1}\mathbb E^P\!\left[\int_0^Te^{\varpi t}|\delta\mathsf f_t|^2\dd t\right]$. Whenever $\varpi\ge2$, so that $4\varpi^{-2}\le2\varpi^{-1}$, \eqref{eq:G-Picard} gives
\[
\|\Phi(y,z,u)-\Phi(\bar y,\bar z,\bar u)\|_\varpi^2\le\frac{4}{\varpi}\mathbb E^P\!\left[\int_0^Te^{\varpi t}|\delta\mathsf f_t|^2\dd t\right]\le\frac{32L^2}{\varpi}\|(y,z,u)-(\bar y,\bar z,\bar u)\|_\varpi^2 .
\]
For $\varpi=\max(64L^2,2)$, $\Phi$ is a contraction with ratio at most $2^{-1/2}$ on the complete space $(\mathcal H_P^2,\|\cdot\|_\varpi)$. Its unique fixed point $(Y,Z,U)$ coincides $\dd t\otimes\dd P$-a.e. with the càdlàg process in $\cS_P^2$ associated with it in Step 1; this càdlàg representative solves \eqref{eq:backward}. Conversely, any solution of \eqref{eq:backward} is a fixed point of $\Phi$; hence uniqueness in $\mathcal H_P^2$, and then in $\cS_P^2$ upon returning to the equation.

\emph{Step 3: a priori estimate.} Set
\[
\mathsf D_t=C_B+L\bigl(|X_{t-}|+M_2(\mu_{t-})+|\alpha_t|\bigr)+|\chi_t|,
\qquad F_t=f(t,X_{t-},Y_t,Z_t,U_t,\mu_{t-},\eta_t,\alpha_t)+\chi_t .
\]
Since $W_2(\eta_t,\mathfrak d_0)=(\mathbb E^P[|Y_t|^2])^{1/2}$,
\[
|F_t|\le \mathsf D_t+L\bigl(|Y_t|+(\mathbb E^P[|Y_t|^2])^{1/2}+|Z_t|_{a_t^P}+\|U_t\|_\Pi\bigr).
\]
The inequalities $2|Y|\mathsf D\le|Y|^2+\mathsf D^2$ and $2L|Y|(|Z|_{a^P}+\|U\|_\Pi)\le4L^2|Y|^2+\frac12(|Z|_{a^P}^2+\|U\|_\Pi^2)$, together with $\mathbb E^P\bigl[|Y_t|(\mathbb E^P[|Y_t|^2])^{1/2}\bigr]\le\mathbb E^P[|Y_t|^2]$, imply
\[
2\,\mathbb E^P[\langle Y_t,F_t\rangle]\le(1+4L+4L^2)\mathbb E^P[|Y_t|^2]+\mathbb E^P[\mathsf D_t^2]+\tfrac12\mathbb E^P\bigl[|Z_t|_{a_t^P}^2+\|U_t\|_\Pi^2\bigr].
\]
Applying Itô's formula to $e^{\varpi t}|Y_t|^2$ between $0$ and $T$, where the stochastic integrals are true martingales as in Step 2, gives
\[
\mathbb E^P[|Y_0|^2]+\mathbb E^P\!\left[\int_0^Te^{\varpi t}\bigl(\varpi|Y_t|^2+|Z_t|_{a_t^P}^2+\|U_t\|_\Pi^2\bigr)\dd t\right]=e^{\varpi T}\,\mathbb E^P[|\zeta|^2]+2\,\mathbb E^P\!\left[\int_0^Te^{\varpi t}\langle Y_t,F_t\rangle\dd t\right] .
\]
Using the previous bound with $\varpi=2+4L+4L^2$, and $1\le e^{\varpi t}\le e^{\varpi T}$, we obtain
\begin{equation}\label{eq:backward-energy-L2}
\mathbb E^P\!\left[\int_0^T\Bigl(|Y_t|^2+\tfrac12|Z_t|_{a_t^P}^2+\tfrac12\|U_t\|_\Pi^2\Bigr)\dd t\right]
\le e^{\varpi T}\,\mathbb E^P\Bigl[|\zeta|^2+\int_0^T\mathsf D_t^2\dd t\Bigr].
\end{equation}
For the supremum, $Y_t=\zeta+\int_t^TF_s\dd s-(\mathsf M_T-\mathsf M_t)$, where $\mathsf M=\int Z\dd B^P+\int\!\!\int U\dd\widetilde N^P$; Doob's inequality gives $\mathbb E^P[\sup_t|\mathsf M_t|^2]\le4\,\mathbb E^P\!\left[\int_0^T(|Z_t|_{a_t^P}^2+\|U_t\|_\Pi^2)\dd t\right]$. Together with $\mathbb E^P[|F_t|^2]\le5\,\mathbb E^P[\mathsf D_t^2]+5L^2\,\mathbb E^P[2|Y_t|^2+|Z_t|_{a_t^P}^2+\|U_t\|_\Pi^2]$ and \eqref{eq:backward-energy-L2}, this gives
\[
\mathbb E^P\Bigl[\sup_{t\le T}|Y_t|^2+\int_0^T\bigl(|Z_t|_{a_t^P}^2+\|U_t\|_\Pi^2\bigr)\dd t\Bigr]
\le C\,\mathbb E^P\Bigl[|\zeta|^2+\int_0^T\bigl(1+|X_{t-}|^2+M_2(\mu_{t-})^2+|\alpha_t|^2+|\chi_t|^2\bigr)\dd t\Bigr],
\]
where $C$ depends only on $L_B,C_B,T$. Finally, $|\zeta|^2\le C(1+|X_T|^2+M_2(\mu_T)^2)$, and \eqref{eq:forward-apriori} together with \eqref{eq:K-apriori} yields \eqref{eq:backward-apriori}.
\end{proof}

\begin{corollary}\label{cor:wellposed-selected}
Fix $P\in\PP$ and suppose the assumptions of Theorem \ref{thm:backward} hold. Let $\xi\in L^2(\Omega,\mathcal F_0,P;\R^d)$, $\alpha\in\mathcal A_2$, $\gamma\in(0,1]$, and let $q$ be an admissible epistemic selection at level $\gamma$. Then the triangular system \eqref{eq:forward}, \eqref{eq:K}, \eqref{eq:backward} has a unique selected solution in the sense of Definition \ref{def:solution},
\[
(X^P,Y^{P,q},Z^{P,q},U^{P,q},K^{P,q})
\in
\cS_P^2\times\cS_P^2\times\cH_P^2\times\cJ_P^2\times\cS_P^2.
\]
\end{corollary}

\begin{proof}
Theorem \ref{thm:forward}, Proposition \ref{prop:K}, and Theorem \ref{thm:backward}, applied successively along the composition displayed at the beginning of this section, give existence; since each step has a unique output, the selected solution is unique.
\end{proof}

\paragraph{Solution map and solution cuts}
For fixed $P$, set
\[
\mathfrak E_P^2=\cS_P^2\times\cH_P^2\times\cJ_P^2\times\cS_P^2,
\]
\[
\|(Y,Z,U,K)\|_{\mathfrak E_P^2}^2
=
\mathbb E^P\!\left[\sup_{t\le T}|Y_t|^2\right]
+\mathbb E^P\!\left[\int_0^T|Z_t|_{a_t^P}^2\dd t\right]
+\mathbb E^P\!\left[\int_0^T\|U_t\|_\Pi^2\dd t\right]
+\mathbb E^P\!\left[\sup_{t\le T}|K_t|^2\right].
\]
For $P\in\PP$, a datum $(\xi,\alpha)\in\mathfrak D_P$, and an admissible selection $q$, denote by $\mathcal Y^{P,q}=\mathcal T_P^{\xi,\alpha}(q)=(Y^{P,q},Z^{P,q},U^{P,q},K^{P,q})$ the selected solution from Corollary \ref{cor:wellposed-selected}, and, for the selection system $(\mathfrak A_{\gamma,M}^P)$ of Definition \ref{def:selsystem}, define the solution cuts
\begin{equation}\label{eq:cuts-def}
\cC_{\gamma,M}^P[\xi,\alpha]
=
\mathcal T_P^{\xi,\alpha}\bigl(\mathfrak A_{\gamma,M}^P[\xi,\alpha]\bigr),
\qquad \gamma\in(0,1].
\end{equation}
When $P$ and the data are fixed and no comparison between data sets is made, we write $\mathcal T_P$ and $\cC_{\gamma,M}^P$. By \eqref{eq:backward-apriori} and \eqref{eq:K-apriori}, there exists $C_{\rm ap}$, independent of $P,\gamma$, and $q$ for data in $\mathfrak D_P$ with common bounds, such that
\begin{equation}\label{eq:Cap}
\|\mathcal T_P(q)\|_{\mathfrak E_P^2}^2\le C_{\rm ap}.
\end{equation}
After identifying elements at zero distance, $\mathfrak E_P^2$ is a Banach space as a finite product of the Banach spaces introduced in Section~\ref{sec:base}. The map $\mathcal T_P$ depends only on the $\dd t\otimes\dd P$-equivalence class of $q$. Indeed, if $q=q'$ $\dd t\otimes\dd P$-a.e., then the source densities coincide $\dd t\otimes\dd P$-a.e., hence $K^{P,q}=K^{P,q'}$ up to indistinguishability, and uniqueness in Theorem~\ref{thm:backward} gives $\mathcal T_P(q)=\mathcal T_P(q')$. Thus $\mathcal T_P$ is well defined on subsets of $L^\infty(\dd t\otimes\dd P)$, and the cuts \eqref{eq:cuts-def} are subsets of $\mathfrak E_P^2$.

\subsection{Backward stability in the norm of the integrated source}

The next estimate is formulated directly in terms of the integrated source $K$. This is the form needed to preserve the rates coming from occupation estimates, and its constant is independent of the spatial regularity of the source density.

\begin{lemma}\label{lem:energy}
Fix $P\in\PP$ and suppose Assumptions \ref{ass:base} and \ref{ass:B} hold. For $i=1,2$, let $X^i\in\cS_P^2$ with $\mu_t^i=\Law^P(X_t^i)$, let $\alpha^i\in L^2(\dd t\otimes\dd P)$ be predictable and $\mathbb D$-valued, let $\zeta^i\in L^2(\Omega,\mathcal F_T,P;\R^d)$, and let $K^i$ be an adapted continuous finite-variation process with $K_0^i=0$ and $\mathbb E^P[\sup_{t\le T}|K_t^i|^2]<\infty$. Let $(Y^i,Z^i,U^i)\in\cS_P^2\times\cH_P^2\times\cJ_P^2$ solve
\[
Y_t^i=\zeta^i+\int_t^TF_s^i\dd s+K_T^i-K_t^i-\int_t^TZ_s^i\dd B_s^P-\int_t^T\int_EU_s^i(e)\widetilde N^P(\dd s,\dd e),
\]
where $F_s^i=f(s,X_{s-}^i,Y_s^i,Z_s^i,U_s^i,\mu_{s-}^i,\Law^P(Y_s^i),\alpha_s^i)$. Then
\begin{align}
&\mathbb E^P\!\left[\sup_{t\le T}|Y_t^1-Y_t^2|^2\right]+\mathbb E^P\!\left[\int_0^T\bigl(|Z_t^1-Z_t^2|_{a_t^P}^2+\|U_t^1-U_t^2\|_\Pi^2\bigr)\dd t\right]\notag\\
&\quad\le C_E \mathbb E^P\!\left[|\zeta^1-\zeta^2|^2+\sup_{t\le T}|K_t^1-K_t^2|^2\right.\notag\\
&\left.\qquad\qquad+\int_0^T\bigl(|X_{s-}^1-X_{s-}^2|^2+W_2^2(\mu_{s-}^1,\mu_{s-}^2)+|\alpha_s^1-\alpha_s^2|^2\bigr)\dd s\right],
\label{eq:backward-stability-source}
\end{align}
where $C_E$ depends only on $L_B$ and $T$; in particular, $C_E$ is independent of $P$ and of the regularity of $K^1,K^2$.
\end{lemma}

\begin{proof}
\emph{Step 1: translation and Lipschitz bound.} In this proof, $L=L_B$. Set $R_t=(Y_t^1-Y_t^2)+(K_t^1-K_t^2)$ and
\[
\mathsf M_t=\int_0^t(Z_s^1-Z_s^2)\dd B_s^P
+\int_0^t\int_E(U_s^1(e)-U_s^2(e))\widetilde N^P(\dd s,\dd e).
\]
Subtracting the two backward equations gives, $P$-a.s., for every $t\in[0,T]$,
\[
R_t=(\zeta^1-\zeta^2)+(K_T^1-K_T^2)
+\int_t^T(F_s^1-F_s^2)\dd s-(\mathsf M_T-\mathsf M_t).
\]
The natural coupling of the laws of $Y_s^1$ and $Y_s^2$ implies
\[
W_2(\Law^P(Y_s^1),\Law^P(Y_s^2))
\le(\mathbb E^P\!\left[|Y_s^1-Y_s^2|^2\right])^{1/2}
\le(\mathbb E^P\!\left[|R_s|^2\right])^{1/2}+(\mathbb E^P\!\left[|K_s^1-K_s^2|^2\right])^{1/2}.
\]
Set
\[
\Xi_s=|X_{s-}^1-X_{s-}^2|+
W_2(\mu_{s-}^1,\mu_{s-}^2)+|\alpha_s^1-\alpha_s^2|.
\]
Lipschitz continuity of the generator gives
\begin{align*}
|F_s^1-F_s^2|
\le L\Big(&|R_s|+|Z_s^1-Z_s^2|_{a_s^P}+\|U_s^1-U_s^2\|_\Pi+(\mathbb E^P\!\left[|R_s|^2\right])^{1/2}\\
&+|K_s^1-K_s^2|+(\mathbb E^P\!\left[|K_s^1-K_s^2|^2\right])^{1/2}+\Xi_s\Big).
\end{align*}

\emph{Step 2: It\^o's formula, martingale and jump terms.} The linear martingale terms in Itô's formula have zero expectation. Indeed, $R\in\cS_P^2$ and, with $\mathsf M$ defined above, BDG and Cauchy--Schwarz yield
\[
\mathbb E^P\!\left[\sup_{t\le T}
\left|\int_0^t e^{\varpi s}\langle R_{s-},\dd\mathsf M_s\rangle\right|\right]
\le C_\varpi
\left(\mathbb E^P\!\left[\sup_{t\le T}|R_t|^2\right]\right)^{1/2}
\left(\mathbb E^P\!\left[[\mathsf M]_T\right]\right)^{1/2}<\infty.
\]
Since a local martingale $\mathsf M'$ satisfying $\mathbb E^P[\sup_{t\le T}|\mathsf M'_t|]<\infty$ is a true martingale \cite[chap.~I]{Protter2005}, this bound implies the martingale property and hence zero expectation.
The quadratic jump term in Itô's formula is $\int_t^T\int_Ee^{\varpi s}|U_s^1(e)-U_s^2(e)|^2N(\dd s,\dd e)$. Since $\mathbb E^P\!\left[\int_0^T\int_E|U_s^1(e)-U_s^2(e)|^2\Pi(\dd e)\dd s\right]<\infty$, the compensation formula \cite[chap.~II]{IkedaWatanabe1989} gives
\begin{align*}
&\mathbb E^P\!\left[\int_t^T\int_Ee^{\varpi s}|U_s^1(e)-U_s^2(e)|^2N(\dd s,\dd e)\right]\\
&\qquad=\mathbb E^P\!\left[\int_t^Te^{\varpi s}\|U_s^1-U_s^2\|_\Pi^2\dd s\right],
\end{align*}
which explains the occurrence of $\|U_s^1-U_s^2\|_\Pi^2$ on the left-hand side below. Apply Itô's formula \cite[chap.~II]{Protter2005} to $e^{\varpi t}|R_t|^2$; since $R_s=R_{s-}$, $\dd s\otimes\dd P$-a.e., we write $R_s$ in the Lebesgue integrals. After taking expectations,
\begin{align*}
&\mathbb E^P\!\left[e^{\varpi t}|R_t|^2\right]+\mathbb E^P\!\left[\int_t^Te^{\varpi s}\Bigl(\varpi|R_s|^2+|Z_s^1-Z_s^2|_{a_s^P}^2+\|U_s^1-U_s^2\|_\Pi^2\Bigr)\dd s\right]\\
&\qquad=\mathbb E^P\!\left[e^{\varpi T}|(\zeta^1-\zeta^2)+(K_T^1-K_T^2)|^2\right]
+2\mathbb E^P\!\left[\int_t^Te^{\varpi s}\langle R_s,F_s^1-F_s^2\rangle\dd s\right].
\end{align*}

\emph{Step 3: absorption of the generator.} The terms involving $Z^1-Z^2$ and $U^1-U^2$ satisfy
\[
2L|R_s|\Bigl(|Z_s^1-Z_s^2|_{a_s^P}+\|U_s^1-U_s^2\|_\Pi\Bigr)
\le4L^2|R_s|^2+\frac12\Bigl(|Z_s^1-Z_s^2|_{a_s^P}^2+\|U_s^1-U_s^2\|_\Pi^2\Bigr).
\]
Moreover,
\begin{align*}
&2L|R_s|\Bigl(|K_s^1-K_s^2|+(\mathbb E^P\!\left[|K_s^1-K_s^2|^2\right])^{1/2}+\Xi_s\Bigr)\\
&\qquad\le3|R_s|^2+L^2\Bigl(|K_s^1-K_s^2|^2+\mathbb E^P\!\left[|K_s^1-K_s^2|^2\right]+\Xi_s^2\Bigr).
\end{align*}
For the remaining law term, Cauchy--Schwarz gives
\[
\mathbb E^P\!\left[
|R_s|\bigl(\mathbb E^P\!\left[|R_s|^2\right]\bigr)^{1/2}
\right]\le\mathbb E^P\!\left[|R_s|^2\right].
\]
Thus, after taking expectations, the contributions of $|R_s|$ and its $L^2(P)$ norm are bounded by $4L\mathbb E^P\!\left[|R_s|^2\right]$. Summing the three preceding estimates gives
\begin{align*}
2\,\mathbb E^P[\langle R_s,F_s^1-F_s^2\rangle]
\le{}&(4L+4L^2+3)\,\mathbb E^P[|R_s|^2]+\tfrac12\mathbb E^P\bigl[|Z_s^1-Z_s^2|_{a_s^P}^2+\|U_s^1-U_s^2\|_\Pi^2\bigr]\\
&+L^2\,\mathbb E^P\bigl[2|K_s^1-K_s^2|^2+\Xi_s^2\bigr].
\end{align*}
With $\varpi=4L+4L^2+4$, so that $\varpi-(4L+4L^2+3)=1$, the term $2\mathbb E^P\!\left[\int_t^Te^{\varpi s}\langle R_s,F_s^1-F_s^2\rangle\dd s\right]$ is absorbed into the left-hand side; coefficient one remains in front of the integral of $e^{\varpi s}|R_s|^2$ and coefficient $1/2$ remains in front of the integrals of $e^{\varpi s}|Z_s^1-Z_s^2|_{a_s^P}^2$ and $e^{\varpi s}\|U_s^1-U_s^2\|_\Pi^2$. Finally, $1\le e^{\varpi s}\le e^{\varpi T}$,
$|(\zeta^1-\zeta^2)+(K_T^1-K_T^2)|^2\le2|\zeta^1-\zeta^2|^2+2|K_T^1-K_T^2|^2$, and the integral of $|K_s^1-K_s^2|^2$ is bounded by $T\sup_{s\le T}|K_s^1-K_s^2|^2$. Hence
\begin{align*}
&\mathbb E^P\!\left[\int_0^T\Bigl(|R_s|^2+|Z_s^1-Z_s^2|_{a_s^P}^2+\|U_s^1-U_s^2\|_\Pi^2\Bigr)\dd s\right]\\
&\qquad\le C \mathbb E^P\!\left[|\zeta^1-\zeta^2|^2+\sup_{t\le T}|K_t^1-K_t^2|^2+\int_0^T\Xi_s^2\dd s\right].
\end{align*}

\emph{Step 4: supremum of $R$.} The representation of $R$ in Step~1 also gives
\[
\sup_{t\le T}|R_t|
\le |(\zeta^1-\zeta^2)+(K_T^1-K_T^2)|+\int_0^T|F_s^1-F_s^2|\dd s+2\sup_{t\le T}|\mathsf M_t|.
\]
To control the supremum, the inequality $(a_1+\cdots+a_7)^2\le7\sum_{i=1}^7 a_i^2$ and the Lipschitz bound give
\begin{align*}
\mathbb E^P\!\left[\int_0^T|F_s^1-F_s^2|^2\dd s\right]
&\le7L^2\mathbb E^P\!\left[\int_0^T
\left(2|R_s|^2+|Z_s^1-Z_s^2|_{a_s^P}^2+\|U_s^1-U_s^2\|_\Pi^2
\right.\right.\\
&\hspace{3cm}\left.\left.+2|K_s^1-K_s^2|^2+\Xi_s^2\right)\dd s\right].
\end{align*}
The isometries and Doob's inequality for the square-integrable martingale $\mathsf M$ yield
\[
\mathbb E^P\!\left[\sup_{t\le T}|\mathsf M_t|^2\right]
\le4\mathbb E^P\!\left[\int_0^T
\left(|Z_s^1-Z_s^2|_{a_s^P}^2+\|U_s^1-U_s^2\|_\Pi^2\right)\dd s\right].
\]
Squaring the preceding supremum representation, we use $(a+b+c)^2\le3(a^2+b^2+c^2)$ and $(\int_0^T|F_s^1-F_s^2|\dd s)^2\le T\int_0^T|F_s^1-F_s^2|^2\dd s$. The integral estimate already obtained then yields
\[
\mathbb E^P\!\left[\sup_{t\le T}|R_t|^2\right]
\le C \mathbb E^P\!\left[|\zeta^1-\zeta^2|^2+\sup_{t\le T}|K_t^1-K_t^2|^2+\int_0^T\Xi_s^2\dd s\right].
\]

\emph{Step 5: return to $Y^1-Y^2$ and conclusion.} Since $Y_t^1-Y_t^2=R_t-(K_t^1-K_t^2)$,
\[
\sup_{t\le T}|Y_t^1-Y_t^2|^2\le2\sup_{t\le T}|R_t|^2+2\sup_{t\le T}|K_t^1-K_t^2|^2,
\]
and
\[
\Xi_s^2\le3\Bigl(|X_{s-}^1-X_{s-}^2|^2+W_2^2(\mu_{s-}^1,\mu_{s-}^2)+|\alpha_s^1-\alpha_s^2|^2\Bigr),
\]
which gives \eqref{eq:backward-stability-source}. All constants $C$ above depend only on $L_B$ and $T$, hence $C_E=C_E(L_B,T)$: the martingale estimates are written in the norms $|\cdot|_{a_t^P}$ and $\|\cdot\|_\Pi$, the Doob and BDG constants are universal, and the compensator $\Pi(\dd e)\dd t$ is common to all $P\in\PP$.
\end{proof}
\section{Occupation estimates}\label{sec:occupation-stabilite}

Throughout this section we fix $P\in\PP$, $(\xi,\alpha)\in\mathfrak D_P$, $\gamma\in(0,1]$, and $q\in\mathfrak A_{\gamma,M}^P[\xi,\alpha]$, while keeping all constants uniform in these parameters. Here and below, $b_t$, $\sigma_t$, and $\beta_t(e)$ denote the coefficients evaluated along $(t,X_{t-}^P,\mu_{t-}^P,\alpha_t)$ and, for $\beta$, at $(t,X_{t-}^P,\mu_{t-}^P,\alpha_t,e)$. The argument has two parts: we first control the motion of the interface (Sections \ref{sec:lawstat}--\ref{sec:shift}) and then estimate the time spent by the forward state near that interface (Section \ref{sec:loctime}).

\subsection{Variation of the law statistics}\label{sec:lawstat}

For $j=1,\ldots,\ell$,
\[
m_j^P(t)=\mathbb E^P\!\left[\psi_j(X_t^P)\right].
\]

\begin{lemma}\label{lem:lawstatistics}
Fix $P\in\PP$ and suppose Assumptions \ref{ass:base}, \ref{ass:F}, and \ref{ass:finitestat} hold. Let $\xi\in L^2(\Omega,\mathcal F_0,P;\R^d)$ and $\alpha\in\mathcal A_2$, let $X^P$ be the corresponding forward solution, and set $m_j^P(t)=\mathbb E^P[\psi_j(X_t^P)]$, $j=1,\ldots,\ell$. Then each $m_j^P$ is absolutely continuous on $[0,T]$ and, for almost every $t$, with $b_t,\sigma_t,\beta_t(e)$ evaluated at $(t,X_{t-}^P,\mu_{t-}^P,\alpha_t)$,
\begin{align}
\dot m_j^P(t)
={}&\mathbb E^P\!\left[
D\psi_j(X_{t-}^P)\cdot b_t
+\frac12\operatorname{Tr}(\sigma_ta_t^P\sigma_t^\top D^2\psi_j(X_{t-}^P))\right.\notag\\
&\left.+\int_E\{\psi_j(X_{t-}^P+\beta_t(e))-\psi_j(X_{t-}^P)
-D\psi_j(X_{t-}^P)\cdot\beta_t(e)\}\Pi(\dd e)
\right].
\label{eq:law-stat-derivative}
\end{align}
Moreover, there exists $C_j$, independent of $P$ for data satisfying common quadratic bounds, such that
\begin{equation}\label{eq:law-stat-L1}
\int_0^T|\dot m_j^P(t)|\dd t\le C_j.
\end{equation}
\end{lemma}

\begin{proof}
Fix $P$ and $j$, and abbreviate the coefficients evaluated at $(t,X_{t-}^P,\mu_{t-}^P,\alpha_t)$ by $b_t,\sigma_t,\beta_t(e)$. Set $B_j=\|D\psi_j\|_\infty$ and $B_j^{(2)}=\|D^2\psi_j\|_\infty$. Taylor's formula with integral remainder gives, for all $x,z\in\R^d$,
\[
|\psi_j(x+z)-\psi_j(x)-D\psi_j(x)\cdot z|
\le \tfrac12B_j^{(2)}|z|^2.
\]
Hence the drift term $G_j(t)$ in Itô's formula satisfies
\[
|G_j(t)|\le B_j|b_t|+
\tfrac12B_j^{(2)}|\sigma_t|_{a_t^P}^2+
\tfrac12B_j^{(2)}\int_E|\beta_t(e)|^2\Pi(\dd e).
\]
The growth assumptions and \eqref{eq:forward-apriori} imply
\[
\mathbb E^P\!\left[\int_0^T\left(|b_t|^2+|\sigma_t|_{a_t^P}^2+
\int_E|\beta_t(e)|^2\Pi(\dd e)\right)\dd t\right]\le C,
\]
uniformly in $P$. Thus $G_j\in L^1(\dd t\otimes\dd P)$. The two martingale integrands in Itô's formula satisfy
\[
|D\psi_j(X_{t-}^P)\sigma_t|_{a_t^P}^2
+\int_E|\psi_j(X_{t-}^P+\beta_t(e))-\psi_j(X_{t-}^P)|^2\Pi(\dd e)
\le B_j^2\left(|\sigma_t|_{a_t^P}^2+
\int_E|\beta_t(e)|^2\Pi(\dd e)\right).
\]
The corresponding stochastic integrals are therefore square-integrable martingales with zero expectation. Itô's formula \cite[chap.~II]{Protter2005} and Fubini's theorem yield, for every $t\in[0,T]$,
\[
m_j^P(t)=m_j^P(0)+\int_0^t \mathbb E^P\!\left[G_j(s)\right]\dd s.
\]
Fubini's theorem applies because $\mathbb E^P\!\left[\int_0^T|G_j(s)|\dd s\right]<\infty$. This identity proves absolute continuity and \eqref{eq:law-stat-derivative} for almost every $t$. Finally,
\begin{align*}
\int_0^T|\dot m_j^P(t)|\dd t
&\le B_j\sqrt T\left(\mathbb E^P\!\left[\int_0^T|b_t|^2\dd t\right]\right)^{1/2}\\
&\quad+\tfrac12B_j^{(2)} \mathbb E^P\!\left[\int_0^T\left(|\sigma_t|_{a_t^P}^2+
\int_E|\beta_t(e)|^2\Pi(\dd e)\right)\dd t\right]\le C_j.
\end{align*}
The constant is uniform in $P$.
\end{proof}

\subsection{Total variation of the interface shift}\label{sec:shift}

Set
\[
H_t^{P,q}=h^{q_t}(\mu_t^P,\alpha_t).
\]

\begin{proposition}\label{prop:interface-motion}
Fix $P\in\PP$ and suppose Assumptions \ref{ass:base}, \ref{ass:F}, and \ref{ass:finitestat} hold. Let $(\xi,\alpha)\in\mathfrak D_P$, $\gamma\in(0,1]$, and $q\in\mathfrak A_{\gamma,M}^P[\xi,\alpha]$. Then the process $H_t^{P,q}=h^{q_t}(\mu_t^P,\alpha_t)$ is absolutely continuous and, for almost every $t$,
\[
\dot H_t^{P,q}
=
D_q\widehat h\,\dot q_t+D_\alpha\widehat h\,\dot\alpha_t
+\sum_{j=1}^{\ell}D_{m_j}\widehat h\,\dot m_j^P(t).
\]
Moreover, there exist constants $C_H$ and $C_H^{(2)}$, depending on $M$, such that
\begin{equation}\label{eq:H-var-L1}
\sup_{P,\gamma,q}
\mathbb E^P\!\left[\Var_{[0,T]}(H^{P,q})\right]
=
\sup_{P,\gamma,q}\mathbb E^P\!\left[\int_0^T|\dot H_t^{P,q}|\dd t\right]
\le C_H
\end{equation}
and
\begin{equation}\label{eq:H-var-L2}
\sup_{P,\gamma,q}\mathbb E^P\!\left[\Var_{[0,T]}(H^{P,q})^2\right]\le C_H^{(2)}.
\end{equation}
\end{proposition}

\begin{proof}
For almost every path, $(q,\alpha,m_1^P,\ldots,m_\ell^P)$ is absolutely continuous. Since $\widehat h$ is $C^1$ with bounded first derivatives, its composition with an absolutely continuous path is absolutely continuous, and the chain rule holds at every time at which all components of the path are differentiable, hence for almost every $t$. The derivatives $\dot q,\dot\alpha$ are taken in the predictable version of Definition \ref{def:regular}, and $\dot m_j^P$ is deterministic, so $\dot H^{P,q}$ is measurable and the expectations below are well defined. Since the derivatives of $\widehat h$ are bounded,
\[
|\dot H_t^{P,q}|
\le C\left(|\dot q_t|+|\dot\alpha_t|
+\sum_{j=1}^{\ell}|\dot m_j^P(t)|\right).
\]
Since $\Var_{[0,T]}(H^{P,q})=\int_0^T|\dot H_t^{P,q}|\dd t$, Cauchy--Schwarz in $(t,\omega)$ gives
\[
\begin{aligned}
\mathbb E^P\!\left[\Var_{[0,T]}(H^{P,q})\right]
\le C\Bigg(&T^{1/2}\Bigl(\mathbb E^P\!\left[\int_0^T|\dot q_t|^2\dd t\right]\Bigr)^{1/2}\\
&+T^{1/2}\Bigl(\mathbb E^P\!\left[\int_0^T|\dot\alpha_t|^2\dd t\right]\Bigr)^{1/2}
+\sum_{j=1}^{\ell}\int_0^T|\dot m_j^P(t)|\dd t\Bigg).
\end{aligned}
\]
The first two terms are controlled by the energy bounds of $q$ and $\alpha$, and the last one by Lemma \ref{lem:lawstatistics}. This proves \eqref{eq:H-var-L1}. Moreover,
\[
\left(\int_0^T|\dot q_t|\dd t\right)^2\le T\int_0^T|\dot q_t|^2\dd t,
\qquad
\left(\int_0^T|\dot\alpha_t|\dd t\right)^2\le T\int_0^T|\dot\alpha_t|^2\dd t.
\]
The functions $m_j^P$ are deterministic for fixed $P$, and Lemma \ref{lem:lawstatistics} uniformly bounds their total variation. More precisely,
\[
\Var_{[0,T]}(H^{P,q})^2\le3C_{\widehat h}^2\left[T\int_0^T|\dot q_t|^2\dd t+T\int_0^T|\dot\alpha_t|^2\dd t+\left(\sum_j\int_0^T|\dot m_j^P(t)|\dd t\right)^2\right].
\]
Taking expectations yields \eqref{eq:H-var-L2}; in particular, $C_H^{(2)}$ depends on $M$.
\end{proof}

\subsection{Local time and occupation}\label{sec:loctime}

Define
\[
S_t^{P,q,k}=X_t^{P,k}+H_t^{P,q,k}.
\]
Since $H^{P,q,k}$ is continuous and of finite variation,
\[
(S^{P,q,k})^c=(X^{P,k})^c=\int_0^\cdot \mathbf e_k^\top\sigma(s,X_{s-}^P,\mu_{s-}^P,\alpha_s)\dd B_s^P,
\]
hence
\[
\dd\langle(S^{P,q,k})^c\rangle_t
=
\mathbf e_k^\top\sigma_ta_t^P\sigma_t^\top \mathbf e_k\,\dd t.
\]

We use $\operatorname{sgn}(0)=-1$, corresponding to the left derivative of $|\cdot|$, and the version of local time $\mathsf L^y$ that is right-continuous in $y$. For a real-valued càdlàg semimartingale $S$ and $y\in\R$, we use the following two facts, with the convention of \cite[chap.~IV]{Protter2005}.
\begin{enumerate}[label=(\alph*)]
\item (Meyer--Tanaka formula.) There exists an adapted continuous increasing process $\mathsf L^y(S)$ with $\mathsf L_0^y(S)=0$ such that, $P$-a.s., for every $t\in[0,T]$,
\begin{align*}
|S_t-y|&=|S_0-y|+\int_0^t\operatorname{sgn}(S_{s-}-y)\dd S_s+\mathsf L_t^y(S)\\
&\quad+\sum_{0<s\le t}\Bigl(|S_s-y|-|S_{s-}-y|-\operatorname{sgn}(S_{s-}-y)(S_s-S_{s-})\Bigr),
\end{align*}
where the series is absolutely convergent.
\item (Occupation density formula.) There exists a jointly measurable version of $(y,t,\omega)\mapsto \mathsf L_t^y(S)(\omega)$ that is càdlàg in $y$ and such that, $P$-a.s., for every $t\in[0,T]$ and every bounded Borel function $\varphi:\R\to\R$,
\[
\int_0^t\varphi(S_{s-})\dd[S]_s^c=\int_\R\varphi(y)\mathsf L_t^y(S)\dd y,
\]
where $[S]^c$ is the continuous part of the quadratic variation, equal to the bracket $\langle S^c\rangle$ of the continuous martingale part of $S$.
\end{enumerate}
We always use this version of local time. For the convex function $x\mapsto|x-y|$, the sum of jump remainders is nonnegative and finite on each compact time interval, hence absolutely convergent.

\begin{lemma}\label{lem:localtime}\label{lem:localtime-second}
Suppose Assumptions \ref{ass:base}, \ref{ass:F}, \ref{ass:operator}, \ref{ass:finitestat}, and \ref{ass:elliptic} hold. For each admissible choice of $P\in\PP$, $(\xi,\alpha)\in\mathfrak D_P$, $\gamma\in(0,1]$, $q\in\mathfrak A_{\gamma,M}^P[\xi,\alpha]$, and $k\in\{1,\ldots,d\}$, set $S^{P,q,k}=X^{P,k}+H^{P,q,k}$. Then, for every $\bar y>0$, there exist $C_{\rm loc}(\bar y),C_{\rm loc}^{(2)}(\bar y)<\infty$, depending only on $\bar y$, $M$, and the structural constants, such that
\begin{gather*}
\sup_{P,\gamma,q}\ \max_{1\le k\le d}\ \sup_{|y|\le\bar y}
\mathbb E^P\!\left[\mathsf L_T^y(S^{P,q,k})\right]\le C_{\rm loc}(\bar y),\\
\sup_{P,\gamma,q}\ \max_{1\le k\le d}\ \sup_{|y|\le\bar y}
\mathbb E^P\!\left[\mathsf L_T^y(S^{P,q,k})^2\right]\le C_{\rm loc}^{(2)}(\bar y).
\end{gather*}
\end{lemma}

\begin{proof}
\emph{Step 1: decomposition and localization.} Fix $P\in\PP$, $(\xi,\alpha)\in\mathfrak D_P$, $\gamma\in(0,1]$, $q\in\mathfrak A_{\gamma,M}^P[\xi,\alpha]$, $k\in\{1,\ldots,d\}$, and $y\in[-\bar y,\bar y]$, and write $S_t=S_t^{P,q,k}=X_t^{P,k}+H_t^{P,q,k}$. The canonical decomposition is
\[
S_t=S_0+\int_0^t b_k(s,X_{s-}^P,\mu_{s-}^P,\alpha_s)\dd s
+\int_0^t\dot H_s^{P,q,k}\dd s+\mathsf M_t^c+\mathsf M_t^d,
\]
with
\[
\mathsf M_t^c=\int_0^t \mathbf e_k^\top\sigma(s,X_{s-}^P,\mu_{s-}^P,\alpha_s)\dd B_s^P,
\]
\[
\mathsf M_t^d=\int_0^t\int_E\beta_k(s,X_{s-}^P,\mu_{s-}^P,\alpha_s,e)\widetilde N^P(\dd s,\dd e).
\]
Choose
\[
\tau_n=\inf\left\{t:\langle \mathsf M^c\rangle_t+\int_0^t\int_E|\beta_k(s,X_{s-}^P,\mu_{s-}^P,\alpha_s,e)|^2\Pi(\dd e)\dd s\ge n\right\}\wedge T.
\]
The integral defining $\tau_n$ is continuous, increasing, and finite at $T$, $P$-a.s., by the quadratic bounds for the forward state. Hence, for almost every path, $\tau_n=T$ for all sufficiently large $n$. The stopping times depend on $P$; only the constant in the estimates below will be uniform in $P$.

\emph{Step 2: Tanaka's formula.} The integrands $\operatorname{sgn}(S_{t-}-y)$ are predictable and bounded by one in absolute value. The isometries therefore imply that the two stopped martingale integrals belong to $L^2(P)$. Tanaka's formula applied to $x\mapsto|x-y|$ gives
\begin{align}
|S_{T\wedge\tau_n}-y|
={}&|S_0-y|
+\int_0^{T\wedge\tau_n}\operatorname{sgn}(S_{t-}-y)b_k(t,X_{t-}^P,\mu_{t-}^P,\alpha_t)\dd t\notag\\
&+\int_0^{T\wedge\tau_n}\operatorname{sgn}(S_{t-}-y)\dot H_t^{P,q,k}\dd t
+\int_0^{T\wedge\tau_n}\operatorname{sgn}(S_{t-}-y)\dd \mathsf M_t^c\notag\\
&+\int_0^{T\wedge\tau_n}\operatorname{sgn}(S_{t-}-y)\dd \mathsf M_t^d
+\mathsf L_{T\wedge\tau_n}^y(S)+R_{T\wedge\tau_n}^y,
\label{eq:tanaka-localtime-full}
\end{align}
where, by convexity of $x\mapsto|x-y|$,
\[
R_t^y=\sum_{0<s\le t}\Bigl(|S_s-y|-|S_{s-}-y|-\operatorname{sgn}(S_{s-}-y)(S_s-S_{s-})\Bigr)\ge0.
\]
Isolating $\mathsf L^y$ in \eqref{eq:tanaka-localtime-full}, and using $R^y\ge0$ and $|\operatorname{sgn}|\le1$,
\begin{align*}
\mathsf L_{T\wedge\tau_n}^y(S)
\le{}&|S_{T\wedge\tau_n}-y|+|S_0-y|
+\int_0^{T\wedge\tau_n}|b_k(t,X_{t-}^P,\mu_{t-}^P,\alpha_t)|\dd t
+\int_0^{T\wedge\tau_n}|\dot H_t^{P,q,k}|\dd t\\
&-\int_0^{T\wedge\tau_n}\operatorname{sgn}(S_{t-}-y)\dd \mathsf M_t^c
-\int_0^{T\wedge\tau_n}\operatorname{sgn}(S_{t-}-y)\dd \mathsf M_t^d.
\end{align*}
The two stopped martingales are square integrable and have zero expectation. Therefore,
\begin{align}
\mathbb E^P\!\left[\mathsf L_{T\wedge\tau_n}^y(S)\right]
\le{}&\mathbb E^P\!\left[|S_{T\wedge\tau_n}-y|\right]+\mathbb E^P\!\left[|S_0-y|\right]\\
&+\mathbb E^P\!\left[\int_0^T|b_k(t,X_{t-}^P,\mu_{t-}^P,\alpha_t)|\dd t\right]
+\mathbb E^P\!\left[\Var_{[0,T]}(H^{P,q,k})\right].
\label{eq:localtime-L1-components}
\end{align}

\emph{Step 3: first moment.} The four terms in \eqref{eq:localtime-L1-components} are bounded separately. First,
\[
\mathbb E^P\!\left[|S_{T\wedge\tau_n}-y|\right]
\le (\mathbb E^P\!\left[\sup_{t\le T}|X_t^P|^2\right])^{1/2}
+(\mathbb E^P\!\left[\sup_{t\le T}|H_t^{P,q,k}|^2\right])^{1/2}+\bar y.
\]
Next,
\[
\mathbb E^P\!\left[|S_0-y|\right]\le(\mathbb E^P\!\left[|\xi|^2\right])^{1/2}+(\mathbb E^P\!\left[|H_0^{P,q,k}|^2\right])^{1/2}+\bar y.
\]
The growth of $h$ (Assumption \ref{ass:operator}) gives $|H_0^{P,q,k}|\le C_{hc}(1+|\alpha_0|+M_2(\mu_0^P))$, with $M_2(\mu_0^P)^2=\mathbb E^P[|\xi|^2]$. Hence, by \eqref{eq:alpha0-bound},
\[
\mathbb E^P[|H_0^{P,q,k}|^2]\le3C_{hc}^2\bigl(1+\mathbb E^P[|\alpha_0|^2]+\mathbb E^P[|\xi|^2]\bigr)<\infty .
\]
This bound, the inequality
\[
\sup_{t\le T}|H_t^{P,q,k}|\le |H_0^{P,q,k}|+\Var_{[0,T]}(H^{P,q,k})
\]
and \eqref{eq:H-var-L2} give
\[
\mathbb E^P\!\left[\sup_{t\le T}|H_t^{P,q,k}|^2\right]\le2\mathbb E^P\!\left[|H_0^{P,q,k}|^2\right]+2C_H^{(2)}<\infty.
\]
For the drift,
\begin{align*}
\mathbb E^P\!\left[\int_0^T|b_k(t,X_{t-}^P,\mu_{t-}^P,\alpha_t)|\dd t\right]
&\le T^{1/2}\left(\mathbb E^P\!\left[\int_0^T|b_k(t,X_{t-}^P,\mu_{t-}^P,\alpha_t)|^2\dd t\right]\right)^{1/2}\\
&\le C\left(1+\mathbb E^P\!\left[\sup_{t\le T}|X_t^P|^2\right]+\mathbb E^P\!\left[\int_0^T|\alpha_t|^2\dd t\right]\right)^{1/2}.
\end{align*}
Finally,
\[
\mathbb E^P\!\left[\Var_{[0,T]}(H^{P,q,k})\right]\le C_H.
\]
Inserting these estimates into \eqref{eq:localtime-L1-components} gives
\[
\sup_{n}\,\sup_{P,\gamma,q}\,\max_{1\le k\le d}\,\sup_{|y|\le\bar y}\mathbb E^P\!\left[\mathsf L_{T\wedge\tau_n}^y(S^{P,q,k})\right]\le C_{\rm loc}(\bar y).
\]
Since $\mathsf L_{T\wedge\tau_n}^y(S)\to \mathsf L_T^y(S)$, $P$-a.s., Fatou's lemma yields
\[
\mathbb E^P\!\left[\mathsf L_T^y(S)\right]\le\liminf_{n\to\infty}\mathbb E^P\!\left[\mathsf L_{T\wedge\tau_n}^y(S)\right]\le C_{\rm loc}(\bar y).
\]

\emph{Step 4: second moment.} With the notation of \eqref{eq:tanaka-localtime-full},
\begin{align*}
\mathsf L_{T\wedge\tau_n}^y(S)
\le{}&|S_{T\wedge\tau_n}-y|+|S_0-y|+\int_0^T|b_k(t,X_{t-}^P,\mu_{t-}^P,\alpha_t)|\dd t+\Var_{[0,T]}(H^{P,q,k})\\
&+\left|\int_0^{T\wedge\tau_n}\operatorname{sgn}(S_{t-}-y)\dd \mathsf M_t^c\right|
+\left|\int_0^{T\wedge\tau_n}\operatorname{sgn}(S_{t-}-y)\dd \mathsf M_t^d\right|.
\end{align*}
Using $(x_1+x_2+x_3+x_4+x_5+x_6)^2\le6(x_1^2+x_2^2+x_3^2+x_4^2+x_5^2+x_6^2)$,
\begin{align}
\mathbb E^P\!\left[\mathsf L_{T\wedge\tau_n}^y(S)^2\right]
\le{}&6\mathbb E^P\!\left[|S_{T\wedge\tau_n}-y|^2\right]+6\mathbb E^P\!\left[|S_0-y|^2\right]
+6T \mathbb E^P\!\left[\int_0^T|b_k(t,X_{t-}^P,\mu_{t-}^P,\alpha_t)|^2\dd t\right]\notag\\
&+6\mathbb E^P\!\left[\Var_{[0,T]}(H^{P,q,k})^2\right]+6\mathbb E^P\!\left[|\mathsf M_{T\wedge\tau_n}^{c,y}|^2\right]+6\mathbb E^P\!\left[|\mathsf M_{T\wedge\tau_n}^{d,y}|^2\right],
\label{eq:localtime-L2-components}
\end{align}
where
\[
\mathsf M_t^{c,y}=\int_0^t\operatorname{sgn}(S_{s-}-y)\dd \mathsf M_s^c,
\qquad
\mathsf M_t^{d,y}=\int_0^t\operatorname{sgn}(S_{s-}-y)\dd \mathsf M_s^d.
\]
The isometries give
\begin{align*}
\mathbb E^P\!\left[|\mathsf M_{T\wedge\tau_n}^{c,y}|^2\right]
&=\mathbb E^P\!\left[\int_0^{T\wedge\tau_n}
\mathbf e_k^\top\sigma(s,X_{s-}^P,\mu_{s-}^P,\alpha_s)a_s^P
\sigma(s,X_{s-}^P,\mu_{s-}^P,\alpha_s)^\top \mathbf e_k\dd s\right]\\
&\le \mathbb E^P\!\left[\int_0^T
|\sigma(s,X_{s-}^P,\mu_{s-}^P,\alpha_s)|_{a_s^P}^2\dd s\right].
\end{align*}
\begin{align*}
\mathbb E^P\!\left[|\mathsf M_{T\wedge\tau_n}^{d,y}|^2\right]
&=\mathbb E^P\!\left[\int_0^{T\wedge\tau_n}\int_E
|\beta_k(s,X_{s-}^P,\mu_{s-}^P,\alpha_s,e)|^2\Pi(\dd e)\dd s\right]\\
&\le \mathbb E^P\!\left[\int_0^T\int_E
|\beta(s,X_{s-}^P,\mu_{s-}^P,\alpha_s,e)|^2\Pi(\dd e)\dd s\right].
\end{align*}
By \ref{ass:F} and \eqref{eq:forward-apriori},
\begin{align*}
&\mathbb E^P\!\left[\int_0^T\Bigg(
|b(s,X_{s-}^P,\mu_{s-}^P,\alpha_s)|^2
+|\sigma(s,X_{s-}^P,\mu_{s-}^P,\alpha_s)|_{a_s^P}^2\Bigg.\right.\\
&\qquad\left.\Bigg.+\int_E|\beta(s,X_{s-}^P,\mu_{s-}^P,\alpha_s,e)|^2\Pi(\dd e)
\Bigg)\dd s\right]\le C.
\end{align*}
By the growth of $h$ and \eqref{eq:H-var-L2},
\[
\mathbb E^P\!\left[|S_0-y|^2\right]+\mathbb E^P\!\left[\sup_{t\le T}|S_t-y|^2\right]+\mathbb E^P\!\left[\Var_{[0,T]}(H^{P,q,k})^2\right]\le C(\bar y,M).
\]
To make the uniformity explicit, set
\[
C_{S,2}=\sup_{P,\gamma,q}\max_{1\le k\le d}\mathbb E^P\!\left[\sup_{t\le T}|S_t^{P,q,k}|^2\right],
\]
and let $C_b,C_\sigma,C_\beta$ denote the uniform bounds for the three corresponding terms in the preceding estimate. The bound \eqref{eq:alpha0-bound}, the growth of $h$, and \eqref{eq:H-var-L2} yield $C_{S,2}<\infty$. Since $|S_t-y|^2\le2|S_t|^2+2\bar y^2$, one admissible constant is
\[
C_{\rm loc}^{(2)}(\bar y)=
24(C_{S,2}+\bar y^2)+6T C_b+6C_H^{(2)}+6C_\sigma+6C_\beta.
\]
It depends neither on $n$ nor on $P,\gamma,q,k,y$ with $|y|\le \bar y$. Substituting into \eqref{eq:localtime-L2-components} gives
\[
\sup_{n}\,\sup_{P,\gamma,q}\,\max_{1\le k\le d}\,\sup_{|y|\le\bar y}\mathbb E^P\!\left[\mathsf L_{T\wedge\tau_n}^y(S^{P,q,k})^2\right]\le C_{\rm loc}^{(2)}(\bar y).
\]
Fatou's lemma implies
\[
\mathbb E^P\!\left[\mathsf L_T^y(S)^2\right]\le\liminf_{n\to\infty}\mathbb E^P\!\left[\mathsf L_{T\wedge\tau_n}^y(S)^2\right]\le C_{\rm loc}^{(2)}(\bar y).
\]
\end{proof}

\begin{theorem}\label{thm:occupation}
Suppose \ref{ass:base}, \ref{ass:F}, \ref{ass:operator}, \ref{ass:finitestat}, \ref{ass:elliptic}, and \ref{ass:finite-branches}. For each admissible choice of $P\in\PP$, $(\xi,\alpha)\in\mathfrak D_P$, $\gamma\in(0,1]$, and $q\in\mathfrak A_{\gamma,M}^P[\xi,\alpha]$, and for every $r>0$, set
\[
\mathsf O_r^{P,q}[\xi,\alpha]
=\int_0^T\1_{\{\rho^{q_t}(X_{t-}^P,\mu_{t-}^P,\alpha_t)\le r\}}\dd t,
\]
where $X^P$ is the forward solution for the datum $(\xi,\alpha)$ and $\mu^P$ is its flow of laws; we write $\mathsf O_r^{P,q}$ when the data are fixed. Then there exist constants $C_{\rm occ},C_{\rm occ}^{(2)}$, depending only on $M$ and the structural constants, and therefore independent of these indices, such that, for every $r\in(0,1]$,
\[
\sup_{P,\gamma,q}\mathbb E^P\!\left[\mathsf O_r^{P,q}\right]\le C_{\rm occ}r
\]
and
\begin{equation}\label{eq:occupation-second}
\sup_{P,\gamma,q}\mathbb E^P\!\left[\bigl(\mathsf O_r^{P,q}\bigr)^2\right]\le C_{\rm occ}^{(2)}r^2.
\end{equation}
Moreover, for all such $P$, $(\xi,\alpha)$, $\gamma$, and $q$, $\operatorname{Leb}\{t\in[0,T]:\rho^{q_t}(X_{t-}^P,\mu_{t-}^P,\alpha_t)=0\}=0$, $P$-a.s.
\end{theorem}

\begin{proof}
\emph{Step 1: reduction to single thresholds.} Fix $P$, $(\xi,\alpha)$, $\gamma$, $q$, and $0<r\le1$, and set $R_\Theta=\Theta_*+1$. The process $H^{P,q}$ is continuous because it is absolutely continuous, and the flow $\mu^P$ is continuous (Lemma \ref{lem:flowcontinuity}); hence $S_{t-}^{P,q,k}=X_{t-}^{P,k}+h_k^{q_t}(\mu_{t-}^P,\alpha_t)$ and
\[
\rho^{q_t}(X_{t-}^P,\mu_{t-}^P,\alpha_t)=\min_{k,j}|S_{t-}^{P,q,k}-\theta_{k,j}|.
\]
For a threshold $\theta=\theta_{k,j}$, set $\mathsf O_r^{k,j}=\int_0^T\1_{\{|S_{t-}^{P,q,k}-\theta|\le r\}}\dd t$, so that the occupation time $\mathsf O_r=\mathsf O_r^{P,q}$ of $\{\rho^{q_t}\le r\}$ satisfies $\mathsf O_r\le\sum_{k,j}\mathsf O_r^{k,j}$. Since $H^{P,q}$ is of finite variation, $\dd\langle(S^{P,q,k})^c\rangle_t=\mathbf e_k^\top\sigma_ta_t^P\sigma_t^\top \mathbf e_k\,\dd t\ge\underline\sigma^2\dd t$; and since a càdlàg path has only countably many jumps, $P$-a.s., $S_t^{P,q,k}=S_{t-}^{P,q,k}$ for almost every $t\in[0,T]$. The occupation density formula therefore gives, $P$-a.s.,
\[
\underline\sigma^2\mathsf O_r^{k,j}\le\int_{\theta-r}^{\theta+r}\mathsf L_T^y(S^{P,q,k})\dd y,
\qquad [\theta-r,\theta+r]\subset[-R_\Theta,R_\Theta].
\]

\emph{Step 2: first moment.} Taking expectations, Tonelli's theorem and the first bound of Lemma \ref{lem:localtime}, with $\bar y=R_\Theta$, give
\[
\mathbb E^P\!\left[\mathsf O_r^{k,j}\right]
\le\underline\sigma^{-2}\int_{\theta-r}^{\theta+r}\mathbb E^P\!\left[\mathsf L_T^y(S^{P,q,k})\right]\dd y
\le2\underline\sigma^{-2}C_{\rm loc}(R_\Theta)\,r .
\]

\emph{Step 3: second moment.} By Cauchy--Schwarz on the interval $[\theta-r,\theta+r]$, of length $2r$,
\[
\bigl(\mathsf O_r^{k,j}\bigr)^2
\le\underline\sigma^{-4}\left(\int_{\theta-r}^{\theta+r}\mathsf L_T^y(S^{P,q,k})\dd y\right)^2
\le2r\,\underline\sigma^{-4}\int_{\theta-r}^{\theta+r}\mathsf L_T^y(S^{P,q,k})^2\dd y .
\]
Taking expectations, Tonelli's theorem and the second bound of Lemma \ref{lem:localtime} give
\[
\mathbb E^P\!\left[\bigl(\mathsf O_r^{k,j}\bigr)^2\right]
\le2r\,\underline\sigma^{-4}\int_{\theta-r}^{\theta+r}\mathbb E^P\!\left[\mathsf L_T^y(S^{P,q,k})^2\right]\dd y
\le4\underline\sigma^{-4}C_{\rm loc}^{(2)}(R_\Theta)\,r^2 .
\]

\emph{Step 4: summation over the thresholds.} Recall that $N_\Theta$ is the total number of thresholds. Since $\mathsf O_r\le\sum_{k,j}\mathsf O_r^{k,j}$ and $(\sum_{k,j}x_{k,j})^2\le N_\Theta\sum_{k,j}x_{k,j}^2$,
\[
\mathbb E^P\!\left[\mathsf O_r\right]\le\sum_{k,j}\mathbb E^P\!\left[\mathsf O_r^{k,j}\right]\le2N_\Theta\underline\sigma^{-2}C_{\rm loc}(R_\Theta)\,r,
\qquad
\mathbb E^P\!\left[\mathsf O_r^2\right]\le N_\Theta\sum_{k,j}\mathbb E^P\!\left[\bigl(\mathsf O_r^{k,j}\bigr)^2\right]\le4N_\Theta^2\underline\sigma^{-4}C_{\rm loc}^{(2)}(R_\Theta)\,r^2 .
\]
The two estimates therefore hold with
\[
C_{\rm occ}=2N_\Theta\underline\sigma^{-2}C_{\rm loc}(R_\Theta),\qquad C_{\rm occ}^{(2)}=4N_\Theta^2\underline\sigma^{-4}C_{\rm loc}^{(2)}(R_\Theta),
\]
uniformly in $P,\gamma,q$.

\emph{Step 5: the interface has zero occupation.} As $r\downarrow0$, $\1_{\{\rho^{q_t}\le r\}}\downarrow\1_{\{\rho^{q_t}=0\}}$, and these indicators are bounded by $1$. Dominated convergence together with the first estimate gives
\[
\mathbb E^P\!\left[\int_0^T
\1_{\{\rho^{q_t}(X_{t-}^P,\mu_{t-}^P,\alpha_t)=0\}}\dd t\right]=0.
\]
\end{proof}

\begin{remark}\label{rem:occupation-scope}
The proofs of Proposition \ref{prop:interface-motion}, Lemma \ref{lem:localtime}, and Theorem \ref{thm:occupation} use the process $q$ only through the following properties: $q$ is predictable, $\Q$-valued, has absolutely continuous paths, and satisfies $\mathbb E^P\!\left[\int_0^T|\dot q_t|^2\dd t\right]\le M^2$. Membership in a cut $\Q_\gamma$ plays no role. The three results, with the same constants, therefore hold for every such process, in particular for every regular admissible selection with energy bound $M$ and for the elements of the ambient class $\overline{\mathfrak A}_M^P$ in (S5).
\end{remark}
\section{Intrinsic selection systems}\label{sec:feedback}

Remark \ref{rem:noncompact} shows that the class of all regular selections need not be compact. We therefore construct \emph{intrinsic} selection systems directly from the geometry of the fuzzy data and a regularity constant $L$. For Lipschitz feedback selections of convex cuts, (S1)--(S3) and (S5) hold in every dimension, while (S4) is proved when the fuzzy parameter is scalar (Section \ref{sec:fb-convex}). For center--radius cuts, all properties (S1)--(S5) hold in arbitrary dimension (Section \ref{sec:fb-cr}). The results of Sections \ref{sec:stability}--\ref{sec:yosida} can then be applied without postulating these properties separately. Throughout the section, $\bar\alpha$ denotes a generic element of the control set $\mathbb D$, to distinguish it from the control process $\alpha$ in the data.

\subsection{Lipschitz feedback selections of convex cuts}\label{sec:fb-convex}

\begin{assumption}[Intrinsic feedback regime]\label{ass:feedback}
\begin{enumerate}[label=(\roman*)]
\item $\mathbb D$ is closed, and there exists $R_\alpha>0$ such that, for every $P\in\PP$ and every $(\xi,\alpha)\in\mathfrak D_P$, $P$-a.s., $|\alpha_t|\le R_\alpha$ for all $t\in[0,T]$; the controls satisfy the time regularity of Assumption \ref{ass:finitestat}. Set $\mathbb D_{R_\alpha}=\{\bar\alpha\in\mathbb D:|\bar\alpha|\le R_\alpha\}$ and $\mathsf G_{R_\alpha}=[0,T]\times\mathbb D_{R_\alpha}$; both are compact.
\item For every $(\gamma,\mu,\bar\alpha)$, the set $\Q_\gamma(\mu,\bar\alpha)$ is convex. The cuts are left-continuous in the sense that
\[
\Q_\gamma(\mu,\bar\alpha)
=\bigcap_{0<\iota<\gamma}\Q_\iota(\mu,\bar\alpha),
\qquad \gamma\in(0,1].
\]
\item There exist maps $\widehat\Q_\gamma$ from $\mathbb D\times\R^\ell$ into the nonempty compact convex subsets of $\Q$, with the functions $\psi_j$ from Assumption \ref{ass:finitestat}, such that $\Q_\gamma(\mu,\bar\alpha)=\widehat\Q_\gamma(\bar\alpha,m_1(\mu),\ldots,m_\ell(\mu))$ and
\[
d_H\bigl(\widehat\Q_\gamma(\bar\alpha,y),\widehat\Q_{\gamma'}(\bar\alpha',y')\bigr)\le L_{\widehat\Q}\bigl(|\bar\alpha-\bar\alpha'|+|y-y'|\bigr)+C_{\widehat\Q}|\gamma-\gamma'|^\nu .
\]
\end{enumerate}
\end{assumption}

Item (ii) actually follows from (iii) and closedness of the cuts: if $x\in\Q_\iota(\mu,\bar\alpha)$ for every $\iota<\gamma$, then $\operatorname{dist}(x,\Q_\gamma(\mu,\bar\alpha))\le C_{\widehat\Q}|\gamma-\iota|^\nu$ for every $\iota<\gamma$, hence $x\in\Q_\gamma(\mu,\bar\alpha)$. We state it separately because it is the property used below; it also holds for the upper level sets of any membership function, since $\{u\ge\gamma\}=\bigcap_{0<\iota<\gamma}\{u\ge\iota\}$ \cite{DiamondKloeden1994}.

\begin{lemma}\label{lem:mflow-lip}
Fix $P\in\PP$ and suppose \ref{ass:base}, \ref{ass:F}, \ref{ass:finitestat}, and \ref{ass:feedback}(i) hold. Let $(\xi,\alpha)\in\mathfrak D_P$ have forward flow $\mu^P$, and set $y(t)=(m_1(\mu_t^P),\ldots,m_\ell(\mu_t^P))$. Then $y$ is Lipschitz on $[0,T]$, with a constant $C_y$ depending only on $\mathfrak C_{\rm str}$ and $R_\alpha$.
\end{lemma}

\begin{proof}
By the linear-growth bound in the proof of Theorem \ref{thm:forward} and $M_2(\mu_{t-}^P)^2=\mathbb E^P[|X_{t-}^P|^2]$,
\[
\mathbb E^P\Bigl[|b_t|^2+|\sigma_t|_{a_t^P}^2+\int_E|\beta_t(e)|^2\Pi(\dd e)\Bigr]\le2C_F^2+2L_F^2\Bigl(2\,\mathbb E^P[\sup_{s\le T}|X_s^P|^2]+R_\alpha^2\Bigr)=:\Lambda,
\]
and this is bounded by \eqref{eq:forward-apriori}. Formula \eqref{eq:law-stat-derivative} then gives, for almost every $t$, $|\dot m_j^P(t)|\le\|D\psi_j\|_\infty\Lambda^{1/2}+\frac12\|D^2\psi_j\|_\infty\Lambda$. Since each $m_j^P=m_j(\mu_\cdot^P)$ is absolutely continuous, $y$ is Lipschitz with constant
\[
C_y=\Bigl(\sum_j\bigl(\|D\psi_j\|_\infty\Lambda^{1/2}+\tfrac12\|D^2\psi_j\|_\infty\Lambda\bigr)^2\Bigr)^{1/2}.
\]
\end{proof}

For a nonempty compact convex set $\Sigma\subset\R^{d_q}$, let $\operatorname{st}(\Sigma)\in \Sigma$ denote its Steiner point. There exists a constant $\mathfrak s_{d_q}$, depending only on $d_q$, such that $|\operatorname{st}(\Sigma)-\operatorname{st}(\Sigma')|\le \mathfrak s_{d_q}d_H(\Sigma,\Sigma')$ \cite{AubinCellina1984}; for $d_q=1$, $\operatorname{st}([l,u])=(l+u)/2$ and $\mathfrak s_1=1$. Set
\[
L_*=\mathfrak s_{d_q}L_{\widehat\Q}\max(1,C_y).
\]

\begin{definition}\label{def:feedback}
Fix $P\in\PP$, $(\xi,\alpha)\in\mathfrak D_P$ with forward flow $\mu^P$, $\gamma\in(0,1]$, and $L>0$. Let $\mathcal S_{\gamma,L}^P[\xi,\alpha]$ be the set of maps $s:[0,T]\times\mathbb D_{R_\alpha}\to\R^{d_q}$ such that
\[
|s(t,\bar\alpha)-s(t',\bar\alpha')|\le L\bigl(|t-t'|+|\bar\alpha-\bar\alpha'|\bigr),\qquad s(t,\bar\alpha)\in\Q_\gamma(\mu_t^P,\bar\alpha),
\]
for all $(t,\bar\alpha),(t',\bar\alpha')$. The class of \emph{feedback selections} is $\mathfrak A_{\gamma,L}^{{\rm fb},P}[\xi,\alpha]=\{(s(t,\alpha_t))_{t\in[0,T]}:s\in\mathcal S_{\gamma,L}^P[\xi,\alpha]\}$, with the convention $s(t,\alpha_t):=s(0,\alpha_\star)$, for a fixed $\alpha_\star\in\mathbb D_{R_\alpha}$, on the $P$-null set where $\alpha_t\notin\mathbb D_{R_\alpha}$ for some $t$; this set belongs to $\mathcal F_0^P$, so the convention preserves predictability and does not affect the equivalence classes in $L^\infty(\dd t\otimes\dd P)$. When the data are fixed, we omit $[\xi,\alpha]$.
\end{definition}

\begin{theorem}\label{thm:feedback-class}
Fix $P\in\PP$ and $(\xi,\alpha)\in\mathfrak D_P$, suppose \ref{ass:base}, \ref{ass:F}, \ref{ass:fuzzycuts}, \ref{ass:finitestat}, and \ref{ass:feedback}, and let $L\ge L_*$ and $M_L^2=2L^2(T+C_{\dot\alpha})$. Then, for every $\gamma\in(0,1]$:
\begin{enumerate}[label=(\alph*)]
\item $\mathfrak A_{\gamma,L}^{{\rm fb},P}$ is nonempty and compact in $(L^\infty(\dd t\otimes\dd P),\|\cdot\|_{\infty,P})$, and each of its elements is a regular admissible selection at level $\gamma$ with $|\dot q_t|\le L(1+|\dot\alpha_t|)$, hence with energy bound $M_L$;
\item $\mathfrak A_{\gamma_2,L}^{{\rm fb},P}\subseteq\mathfrak A_{\gamma_1,L}^{{\rm fb},P}$ for $0<\gamma_1\le\gamma_2\le1$;
\item $\mathfrak A_{\gamma,L}^{{\rm fb},P}=\bigcap_{0<\iota<\gamma}\mathfrak A_{\iota,L}^{{\rm fb},P}$;
\item if $d_q=1$, then for all $\gamma,\gamma'\in(0,1]$, every $(\xi',\alpha')\in\mathfrak D_P$ with flow $\mu'^P$, and every $q\in\mathfrak A_{\gamma,L}^{{\rm fb},P}[\xi,\alpha]$, there exists $q'\in\mathfrak A_{\gamma',L}^{{\rm fb},P}[\xi',\alpha']$ such that, $P$-a.s., for every $t\in[0,T]$,
\begin{equation}\label{eq:fb-coupling}
|q_t-q'_t|\le L|\alpha_t-\alpha'_t|+L_{\widehat\Q}B_\psi W_2(\mu_t^P,\mu_t'^P)+C_{\widehat\Q}|\gamma-\gamma'|^\nu.
\end{equation}
\item the set $\overline{\mathfrak A}_L^{{\rm fb},P}[\xi,\alpha]$ of processes $(s(t,\alpha_t))_t$, where $s:[0,T]\times\mathbb D_{R_\alpha}\to\Q$ is $L$-Lipschitz, is compact in $L^\infty(\dd t\otimes\dd P)$ and contains $\mathfrak A_{\gamma,L}^{{\rm fb},P}$ for every $\gamma$; its elements are predictable, $\Q$-valued, have absolutely continuous paths, and have energy bound $M_L$.
\end{enumerate}
\end{theorem}

\begin{proof}
Set $\Phi(s)=(s(t,\alpha_t))_t$. Since, $P$-a.s., $\alpha_t\in\mathbb D_{R_\alpha}$ for all $t\in[0,T]$,
\begin{equation}\label{eq:fb-Phi}
\|\Phi(s)-\Phi(s')\|_{\infty,P}\le\sup_{\mathsf G_{R_\alpha}}|s-s'| .
\end{equation}

(a) \emph{Step 1: nonemptiness.} By Assumption \ref{ass:feedback}(iii) and Lemma \ref{lem:mflow-lip}, the map $(t,\bar\alpha)\mapsto\Q_\gamma(\mu_t^P,\bar\alpha)=\widehat\Q_\gamma(\bar\alpha,y(t))$ satisfies $d_H\le L_{\widehat\Q}\max(1,C_y)(|t-t'|+|\bar\alpha-\bar\alpha'|)$. Hence $s_\gamma(t,\bar\alpha)=\operatorname{st}(\Q_\gamma(\mu_t^P,\bar\alpha))$ belongs to $\mathcal S_{\gamma,L_*}^P\subseteq\mathcal S_{\gamma,L}^P$.

\emph{Step 2: compactness.} The set $\mathcal S_{\gamma,L}^P\subset C(\mathsf G_{R_\alpha};\R^{d_q})$ is uniformly bounded, because its elements take values in the compact set $\Q$, and equi-Lipschitz; by the Arzelà--Ascoli theorem, it is relatively compact in the uniform norm. It is closed because a uniform limit of $L$-Lipschitz maps is $L$-Lipschitz and the sets $\Q_\gamma(\mu_t^P,\bar\alpha)$ are closed. Thus $\mathcal S_{\gamma,L}^P$ is compact, and so is its image under $\Phi$, which is continuous by \eqref{eq:fb-Phi}.

\emph{Step 3: regularity and admissibility.} Let $q=\Phi(s)$. Since $s$ is continuous and $\alpha$ is predictable, $q$ is predictable; since $\mu^P$ is continuous (Lemma \ref{lem:flowcontinuity}), $q_t\in\Q_\gamma(\mu_{t-}^P,\alpha_t)$. Moreover, $|q_t-q_u|\le L(|t-u|+|\alpha_t-\alpha_u|)$, so $q$ has absolutely continuous paths with $|\dot q_t|\le L(1+|\dot\alpha_t|)$ for almost every $t$, and $\mathbb E^P\!\left[\int_0^T|\dot q_t|^2\dd t\right]\le2L^2(T+C_{\dot\alpha})=M_L^2$.

(b) If $\gamma_1\le\gamma_2$, then $\Q_{\gamma_2}\subseteq\Q_{\gamma_1}$ by Assumption \ref{ass:fuzzycuts}, hence $\mathcal S_{\gamma_2,L}^P\subseteq\mathcal S_{\gamma_1,L}^P$.

(c) The inclusion $\subseteq$ follows from (b). Conversely, let $q\in\bigcap_{0<\iota<\gamma}\mathfrak A_{\iota,L}^{{\rm fb},P}$, let $\iota_n\uparrow\gamma$, and choose $s_n\in\mathcal S_{\iota_n,L}^P$ such that $q=\Phi(s_n)$ in $L^\infty(\dd t\otimes\dd P)$. All $s_n$ belong to the compact set $\mathcal S_{\iota_1,L}^P$, so a subsequence converges uniformly to an $L$-Lipschitz map $s$. For fixed $j$ and $n\ge j$, $s_n(t,\bar\alpha)\in\Q_{\iota_n}(\mu_t^P,\bar\alpha)\subseteq\Q_{\iota_j}(\mu_t^P,\bar\alpha)$, which is closed; hence $s(t,\bar\alpha)\in\bigcap_j\Q_{\iota_j}(\mu_t^P,\bar\alpha)=\Q_\gamma(\mu_t^P,\bar\alpha)$ by Assumption \ref{ass:feedback}(ii). Thus $s\in\mathcal S_{\gamma,L}^P$ and, by \eqref{eq:fb-Phi}, $q=\Phi(s)$.

(d) Let $d_q=1$ and $q=\Phi(s)$ with $s\in\mathcal S_{\gamma,L}^P[\xi,\alpha]$. Write $[l'(t,\bar\alpha),u'(t,\bar\alpha)]=\Q_{\gamma'}(\mu_t'^P,\bar\alpha)$. For intervals, $\max(|l-\tilde l|,|u-\tilde u|)=d_H([l,u],[\tilde l,\tilde u])$; hence, as in (a), $l'$ and $u'$ are $L_*$-Lipschitz. The map $s'=\min\{\max\{s,l'\},u'\}$, that is, the projection of $s(t,\bar\alpha)$ onto $\Q_{\gamma'}(\mu_t'^P,\bar\alpha)$, is $L$-Lipschitz and belongs to $\mathcal S_{\gamma',L}^P[\xi',\alpha']$. Since $s(t,\bar\alpha)\in\Q_\gamma(\mu_t^P,\bar\alpha)$,
\[
|s'(t,\bar\alpha)-s(t,\bar\alpha)|\le d_H\bigl(\Q_\gamma(\mu_t^P,\bar\alpha),\Q_{\gamma'}(\mu_t'^P,\bar\alpha)\bigr)\le L_{\widehat\Q}B_\psi W_2(\mu_t^P,\mu_t'^P)+C_{\widehat\Q}|\gamma-\gamma'|^\nu,
\]
where we used $|m_j(\mu)-m_j(\mu')|\le\|D\psi_j\|_\infty W_2(\mu,\mu')$. With $q'=(s'(t,\alpha'_t))_t$, the estimate $|q_t-q'_t|\le|s(t,\alpha_t)-s(t,\alpha'_t)|+|s(t,\alpha'_t)-s'(t,\alpha'_t)|$ gives \eqref{eq:fb-coupling}.
(e) The set of $L$-Lipschitz maps $\mathsf G_{R_\alpha}\to\Q$ is compact in the uniform norm by the Arzelà--Ascoli theorem, and so is its image under $\Phi$, by \eqref{eq:fb-Phi}; this image contains $\Phi(\mathcal S_{\gamma,L}^P)$ for every $\gamma$. Predictability, absolute continuity, and the energy bound are obtained as in (a).
\end{proof}

\begin{corollary}\label{cor:feedback}
Suppose \ref{ass:base}, \ref{ass:F}, \ref{ass:fuzzycuts}, \ref{ass:finitestat}, and \ref{ass:feedback}, and let $L\ge L_*$. Then, for every $P\in\PP$, the classes $\mathfrak A_{\gamma,L}^{{\rm fb},P}[\xi,\alpha]$, $(\xi,\alpha)\in\mathfrak D_P$, $\gamma\in(0,1]$, satisfy (S1)--(S3) and (S5) of Definition \ref{def:selsystem} with $M=M_L$. If $d_q=1$, they also satisfy (S4) with $L_Q=\max(L,L_{\widehat\Q}B_\psi,C_{\widehat\Q})$, and therefore form a selection system; under the remaining assumptions of each result, the conclusions of Sections \ref{sec:stability}--\ref{sec:yosida} apply to these selection classes. If $d_q\ge2$, under the analytic assumptions of $(H_{\rm quant})$ and without requiring (S4), the results in these sections that do not use (S4) remain applicable, namely Lemma \ref{lem:commonforward}, the first assertion of Corollary \ref{cor:lipschitzT}, Theorem \ref{thm:cutcompact}, Proposition \ref{prop:leftcuts}, Lemma \ref{lem:membership}, Proposition \ref{prop:fuzzy-standard}, and the results of Section \ref{sec:yosida} other than Corollary \ref{cor:Y-joint}.
\end{corollary}

\begin{proof}
(S1)--(S3) are Theorem \ref{thm:feedback-class}(a)--(c), and (S5) is item (e). For $d_q=1$, (S4) is item (d), since $W_2(\mu_t,\mu'_t)=W_2(\mu_{t-},\mu'_{t-})$ by Lemma \ref{lem:flowcontinuity}. The proofs of the cited results use only (S1)--(S3) and (S5).
\end{proof}

\subsection{Center--radius cuts in arbitrary dimension}\label{sec:fb-cr}

For $d_q\ge2$, the transport used in Theorem \ref{thm:feedback-class}(d) fails for general convex cuts: the projection $s'$ of $s(t,\bar\alpha)$ onto $\Q_{\gamma'}(\mu_t'^P,\bar\alpha)$ still satisfies $|s'-s|\le d_H$, but is generally not Lipschitz in $(t,\bar\alpha)$, since metric projection onto a moving convex set is only $\frac12$-Hölder with respect to the Hausdorff distance. The following structure provides a canonical transport in every dimension.

\begin{assumption}[Center--radius cuts]\label{ass:center-radius}
Assumption \ref{ass:feedback}(i) holds, and there exist a compact set $\mathcal K\subset\R^{d_q}$ star-shaped with respect to $0$, a decreasing function $\varrho:(0,1]\to[0,\overline\varrho ]$ with $|\varrho_\gamma-\varrho_{\gamma'}|\le C_\varrho|\gamma-\gamma'|^\nu$, and a map $\widehat c:\mathbb D\times\R^\ell\to\R^{d_q}$ satisfying $|\widehat c(\bar\alpha,y)-\widehat c(\bar\alpha',y')|\le L_{\widehat c}(|\bar\alpha-\bar\alpha'|+|y-y'|)$ such that, with the functions $\psi_j$ from Assumption \ref{ass:finitestat},
\[
\Q_\gamma(\mu,\bar\alpha)=\widehat c\bigl(\bar\alpha,m_1(\mu),\ldots,m_\ell(\mu)\bigr)+\varrho_\gamma \mathcal K\subseteq\Q .
\]
\end{assumption}

This covers fuzzy balls, with $\mathcal K$ the closed unit ball, and, for $d_q=1$, symmetric fuzzy numbers $[\widehat c-\varrho_\gamma,\widehat c+\varrho_\gamma]$ with $\mathcal K=[-1,1]$, including symmetric triangular and trapezoidal fuzzy numbers; the cuts need not be convex. Set $|\mathcal K|_\infty=\max_{x\in \mathcal K}|x|$, let $\mathfrak K_L$ be the set of $L$-Lipschitz maps $\varsigma:[0,T]\times\mathbb D_{R_\alpha}\to \mathcal K$, endowed with the uniform norm, and, with $y(t)=(m_1(\mu_t),\ldots,m_\ell(\mu_t))$, define
\begin{equation}\label{eq:cr-selection}
q_t^{\gamma,\varsigma}[\mu,\alpha]=\widehat c(\alpha_t,y(t))+\varrho_\gamma\,\varsigma(t,\alpha_t),\qquad \varsigma\in\mathfrak K_L .
\end{equation}

\begin{proposition}\label{prop:center-radius}
Suppose \ref{ass:base}, \ref{ass:F}, \ref{ass:finitestat}, and \ref{ass:center-radius}, and let $L>0$. Then Assumption \ref{ass:paramclass} is satisfied with $\mathfrak P=\mathfrak K_L$, the selections \eqref{eq:cr-selection},
\begin{gather*}
M^2=2(L_{\widehat c}+\overline\varrho L)^2C_{\dot\alpha}+2T(L_{\widehat c}C_y+\overline\varrho L)^2,
\qquad C_Q=C_\varrho|\mathcal K|_\infty,\\
L_Q=\max\bigl(L_{\widehat c}+\overline\varrho L,\ L_{\widehat c}B_\psi,\ C_\varrho|\mathcal K|_\infty\bigr),
\end{gather*}
where $C_y$ is the constant from Lemma \ref{lem:mflow-lip}. In particular, by Proposition \ref{prop:paramclass}, these classes form a selection system in every dimension $d_q$. Under the remaining assumptions of each result, the conclusions of Sections \ref{sec:stability}--\ref{sec:yosida} apply to these selection classes.
\end{proposition}

\begin{proof}
Fix $P\in\PP$ and $(\xi,\alpha)\in\mathfrak D_P$; all constants below are independent of this choice.
\emph{Compactness of $\mathfrak K_L$.} Its elements take values in the compact set $\mathcal K$ and are equi-Lipschitz on the compact set $\mathsf G_{R_\alpha}$; by the Arzelà--Ascoli theorem and closedness of $\mathcal K$, $\mathfrak K_L$ is compact.

(i) The process $q^{\gamma,\varsigma}$ is continuous and adapted, hence predictable, and $q_t^{\gamma,\varsigma}\in\Q_\gamma(\mu_t,\alpha_t)=\Q_\gamma(\mu_{t-},\alpha_t)$ by Lemma \ref{lem:flowcontinuity}.

(ii) $\|q^{\gamma,\varsigma}-q^{\gamma,\varsigma'}\|_{\infty,P}\le\overline\varrho \sup_{\mathsf G_{R_\alpha}}|\varsigma-\varsigma'|$.

(iii) By Lemma \ref{lem:mflow-lip}, $y$ is $C_y$-Lipschitz, and $\alpha$ is absolutely continuous. Hence $q^{\gamma,\varsigma}$ is absolutely continuous and, for almost every $t$, $|\dot q_t^{\gamma,\varsigma}|\le L_{\widehat c}(|\dot\alpha_t|+C_y)+\overline\varrho L(1+|\dot\alpha_t|)$. Squaring and integrating gives the bound $M^2$.

(iv) $\|q^{\gamma,\varsigma}-q^{\gamma',\varsigma}\|_{\infty,P}\le|\varrho_\gamma-\varrho_{\gamma'}|\,|\mathcal K|_\infty\le C_\varrho|\mathcal K|_\infty|\gamma-\gamma'|^\nu$.

(v) Let $\gamma_1\le\gamma_2$, so that $\varrho_{\gamma_2}\le \varrho_{\gamma_1}$. If $\varrho_{\gamma_1}=0$, then $q^{\gamma_2,\varsigma}=q^{\gamma_1,\varsigma}$. Otherwise, $\varsigma'=(\varrho_{\gamma_2}/\varrho_{\gamma_1})\varsigma$ takes values in $\mathcal K$, because $\mathcal K$ is star-shaped with respect to $0$, is $L$-Lipschitz, and satisfies $q^{\gamma_2,\varsigma}=q^{\gamma_1,\varsigma'}$.

(vi) For two pairs $(\mu,\alpha)$, $(\mu',\alpha')$ and the same $\varsigma$, using $|m_j(\mu_t)-m_j(\mu'_t)|\le\|D\psi_j\|_\infty W_2(\mu_t,\mu'_t)$,
\[
|q_t^{\gamma,\varsigma}[\mu,\alpha]-q_t^{\gamma',\varsigma}[\mu',\alpha']|\le L_{\widehat c}\bigl(|\alpha_t-\alpha'_t|+B_\psi W_2(\mu_t,\mu'_t)\bigr)+C_\varrho|\mathcal K|_\infty|\gamma-\gamma'|^\nu+\overline\varrho L|\alpha_t-\alpha'_t|,
\]
which is \eqref{eq:param-cross-coupling} with the stated constant $L_Q$, since $\mu_t=\mu_{t-}$ and $\mu'_t=\mu'_{t-}$.
\end{proof}

\begin{remark}
The selections \eqref{eq:example-selection} in Section \ref{sec:example} are center--radius selections, with $\mathcal K=[-1,1]$, $\widehat c(\bar\alpha,y)=\mathfrak q_{\rm c}+\varepsilon_\mu\tanh(y)+\varepsilon_\alpha\tanh(\bar\alpha)$, $\varrho_\gamma=\varepsilon_Q(1-\gamma)$, and $\varsigma(t,\bar\alpha)=\vartheta_t$, for bounded controls such as the concrete control given there. They are also feedback selections in the sense of Definition \ref{def:feedback} for sufficiently large $L$.
\end{remark}
\section{Stability with respect to the selection, the level, and the data}\label{sec:stability}

We compare two selected systems driven by the same pair $(B^P,\widetilde N^P)$. A geometric branch-switching estimate and the occupation bounds from Section \ref{sec:occupation-stabilite} control the difference between their sources; Lemma \ref{lem:energy} then transfers that control to the backward variables.

\subsection{Quantitative stability and coupling of the data}
The argument first controls the forward perturbation, then the source difference created by branch changes, and finally the BSDE.

Under the same $P$ and on the same stochastic basis, consider two data sets
\[
(\xi,\alpha,\gamma)
\qquad\text{and}\qquad
(\xi',\alpha',\gamma').
\]
The two forward equations are driven by the same $(B^P,\widetilde N^P)$. We directly consider two regular admissible selections
\[
q_t\in\Q_\gamma(\mu_{t-},\alpha_t),
\qquad
q'_t\in\Q_{\gamma'}(\mu'_{t-},\alpha'_t),
\qquad \dd t\otimes\dd P\text{-a.e.},
\]
and assume that they are coupled as in (S4):
\begin{equation}\label{eq:metriccoupling}
|q_t-q'_t|
\le
L_Q\Bigl(
|\gamma-\gamma'|^\nu
+W_2(\mu_{t-},\mu'_{t-})
+|\alpha_t-\alpha'_t|
\Bigr),
\qquad \dd t\otimes\dd P\text{-a.e.}
\end{equation}

Set
\begin{equation}\label{eq:R-stab}
\mathfrak R
=
\mathbb E^P\!\left[|\xi-\xi'|^2\right]
+
\mathbb E^P\!\left[\int_0^T|\alpha_t-\alpha'_t|^2\dd t\right]
+
|\gamma-\gamma'|^{2\nu}.
\end{equation}

\subsection{Stability of the forward state}

\begin{proposition}\label{prop:forward-stability}
Fix $P\in\PP$ and suppose that Assumptions \ref{ass:base} and \ref{ass:F} hold. Let $\xi,\xi'\in L^2(\Omega,\mathcal F_0,P;\R^d)$ and $\alpha,\alpha'\in\mathcal A_2$, and let $X,X'\in\cS_P^2$ be the corresponding solutions of \eqref{eq:forward}, driven by the same $(B^P,\widetilde N^P)$, with law flows $\mu_t=\Law^P(X_t)$ and $\mu'_t=\Law^P(X'_t)$. Then there exists $C$, depending only on $T$, $\overline a$, and $L_F$, such that
\begin{align}
&\mathbb E^P\!\left[\sup_{t\le T}|X_t-X'_t|^2\right]
+
\sup_{t\le T}W_2^2(\mu_t,\mu'_t)\notag\\
&\qquad\le
C\left(
\mathbb E^P\!\left[|\xi-\xi'|^2\right]
+
\mathbb E^P\!\left[\int_0^T|\alpha_t-\alpha'_t|^2\dd t\right]
\right).
\label{eq:forwardstab}
\end{align}
\end{proposition}

\begin{proof}
Subtracting the two forward equations gives, $P$-a.s., for every $t\in[0,T]$,
\begin{align*}
X_t-X'_t
={}&\xi-\xi'\\
&+\int_0^t\Bigl[b(s,X_{s-},\mu_{s-},\alpha_s)-b(s,X'_{s-},\mu'_{s-},\alpha'_s)\Bigr]\dd s\\
&+\int_0^t\Bigl[\sigma(s,X_{s-},\mu_{s-},\alpha_s)-\sigma(s,X'_{s-},\mu'_{s-},\alpha'_s)\Bigr]\dd B_s^P\\
&+\int_0^t\int_E\Bigl[\beta(s,X_{s-},\mu_{s-},\alpha_s,e)-\beta(s,X'_{s-},\mu'_{s-},\alpha'_s,e)\Bigr]\widetilde N^P(\dd s,\dd e).
\end{align*}
For $t\in[0,T]$, the inequality $(x_1+x_2+x_3+x_4)^2\le4(x_1^2+x_2^2+x_3^2+x_4^2)$, Cauchy--Schwarz, and BDG yield
\begin{align*}
&\mathbb E^P\!\left[\sup_{r\le t}|X_r-X'_r|^2\right]\\
&\le4\mathbb E^P\!\left[|\xi-\xi'|^2\right]\\
&\quad+4t\mathbb E^P\!\left[\int_0^t|b(s,X_{s-},\mu_{s-},\alpha_s)-b(s,X'_{s-},\mu'_{s-},\alpha'_s)|^2\dd s\right]\\
&\quad+4C_{\rm BDG}\mathbb E^P\!\left[\int_0^t|\sigma(s,X_{s-},\mu_{s-},\alpha_s)-\sigma(s,X'_{s-},\mu'_{s-},\alpha'_s)|_{a_s^P}^2\dd s\right]\\
&\quad+4C_{\rm BDG}\mathbb E^P\!\left[\int_0^t\int_E|\beta(s,X_{s-},\mu_{s-},\alpha_s,e)-\beta(s,X'_{s-},\mu'_{s-},\alpha'_s,e)|^2\Pi(\dd e)\dd s\right].
\end{align*}
By Assumption \ref{ass:F},
\begin{align*}
&|b(s,X_{s-},\mu_{s-},\alpha_s)-b(s,X'_{s-},\mu'_{s-},\alpha'_s)|^2\\
&\quad+|\sigma(s,X_{s-},\mu_{s-},\alpha_s)-\sigma(s,X'_{s-},\mu'_{s-},\alpha'_s)|_{a_s^P}^2\\
&\quad+\int_E|\beta(s,X_{s-},\mu_{s-},\alpha_s,e)-\beta(s,X'_{s-},\mu'_{s-},\alpha'_s,e)|^2\Pi(\dd e)\\
&\le L_F^2\Bigl(|X_{s-}-X'_{s-}|^2+W_2^2(\mu_{s-},\mu'_{s-})+|\alpha_s-\alpha'_s|^2\Bigr).
\end{align*}
Since $(X_s,X'_s)$ is a coupling of $(\mu_s,\mu'_s)$,
\[
W_2^2(\mu_s,\mu'_s)\le \mathbb E^P\!\left[|X_s-X'_s|^2\right]\le \mathbb E^P\!\left[\sup_{u\le s}|X_u-X'_u|^2\right].
\]
Hence
\[
\mathbb E^P\!\left[\sup_{r\le t}|X_r-X'_r|^2\right]
\le4\mathbb E^P\!\left[|\xi-\xi'|^2\right]+C \mathbb E^P\!\left[\int_0^t|\alpha_s-\alpha'_s|^2\dd s\right]
+C\int_0^t\mathbb E^P\!\left[\sup_{u\le s}|X_u-X'_u|^2\right]\dd s.
\]
Gronwall's lemma gives
\[
\mathbb E^P\!\left[\sup_{t\le T}|X_t-X'_t|^2\right]
\le C\left(\mathbb E^P\!\left[|\xi-\xi'|^2\right]+\mathbb E^P\!\left[\int_0^T|\alpha_s-\alpha'_s|^2\dd s\right]\right),
\]
and
\[
\sup_{t\le T}W_2^2(\mu_t,\mu'_t)
\le \mathbb E^P\!\left[\sup_{t\le T}|X_t-X'_t|^2\right].
\]
\end{proof}

\subsection{Geometric branch-switching lemma}

Under the quantitative assumptions, the value of the minimal selection at the thresholds does not affect the source: by the last assertion of Theorem \ref{thm:occupation}, two choices on $\{\rho^{q_t}=0\}$ yield source densities that coincide $\dd t\otimes\dd P$-a.e., and hence the same $K$ up to indistinguishability.

\begin{lemma}\label{lem:branch}
Suppose that Assumptions \ref{ass:operator} and \ref{ass:finite-branches} hold. For $(x,q,\mu,\alpha)$ and $(x',q',\mu',\alpha')$ in $\R^d\times\Q\times\Ptwo(\R^d)\times\mathbb D$, set $D=|x-x'|+|q-q'|+W_2(\mu,\mu')+|\alpha-\alpha'|$. Then there exist constants $C_1,C_{\rm geo}>0$, independent of these points, such that, for every $r>0$,
\begin{equation}\label{eq:branch-comparison}
|(A^q)^0(x,\mu,\alpha)-(A^{q'})^0(x',\mu',\alpha')|
\le C_1D+2\sqrt d\,\overline v\bigl(\1_{\{\rho^q(x,\mu,\alpha)\le r\}}+\1_{\{D\ge r/C_{\rm geo}\}}\bigr).
\end{equation}
\end{lemma}

\begin{proof}
Set $z=x+h^q(\mu,\alpha)$ and $z'=x'+h^{q'}(\mu',\alpha')$; then $|z-z'|\le C_{\rm geo}D$ with $C_{\rm geo}=\max(1,L_h)$. Write
\[
(A^q)^0(x,\mu,\alpha)-(A^{q'})^0(x',\mu',\alpha')=\alpha_A(z-z')+\bigl(c^q(\mu,\alpha)-c^{q'}(\mu',\alpha')\bigr)+(v^0-v'^0),
\]
where $v^0\in\prod_k\partial \mathfrak g_k^q(z_k)$ and $v'^0\in\prod_k\partial \mathfrak g_k^{q'}(z'_k)$ are the slope parts of the minimal selections. The first two terms are bounded by $(\alpha_AC_{\rm geo}+L_c)D$. On the event $\{\rho^q>r\}\cap\{D<r/C_{\rm geo}\}$, each $z_k$ is at distance strictly greater than $r$ from $\Theta_k$ and $|z_k-z'_k|<r$: the segment $[z_k,z'_k]$ crosses no threshold, and $z_k,z'_k$ belong to the same interval $I_{k,j}$. Hence $v^0_k=v^q_{k,j}$ and $v'^0_k=v^{q'}_{k,j}$, so $|v^0-v'^0|\le\sqrt d\,L_v|q-q'|\le\sqrt d\,L_vD$. On the complement, which is contained in $\{\rho^q\le r\}\cup\{D\ge r/C_{\rm geo}\}$, we use $|v^0-v'^0|\le2\sqrt d\,\overline v$. This proves the claim with $C_1=\alpha_AC_{\rm geo}+L_c+\sqrt d\,L_v$.
\end{proof}

\subsection{Estimate for the difference of the sources}

Set
\[
\chi_t=(A^{q_t})^0(X_{t-},\mu_{t-},\alpha_t),
\qquad
\chi'_t=(A^{q'_t})^0(X'_{t-},\mu'_{t-},\alpha'_t).
\]

\begin{corollary}\label{cor:K-stability}
Fix $P\in\PP$ and suppose that $(H_{\rm quant})$ holds. Let $(\xi,\alpha),(\xi',\alpha')\in\mathfrak D_P$, $\gamma,\gamma'\in(0,1]$, and let $q\in\mathfrak A_{\gamma,M}^P[\xi,\alpha]$ and $q'\in\mathfrak A_{\gamma',M}^P[\xi',\alpha']$ satisfy \eqref{eq:metriccoupling}, with sources $K=K^{P,q}$ and $K'=K^{P,q'}$ constructed with the same $(B^P,\widetilde N^P)$. Let $\mathfrak R$ be defined by \eqref{eq:R-stab}. If $0<\mathfrak R\le 1$, then
\begin{equation}\label{eq:K-stability}
\mathbb E^P\!\left[\sup_{t\le T}|K_t-K'_t|^2\right]\le C\mathfrak R^{1/2},
\end{equation}
where $C$ depends only on $M$ and the structural constants.
\end{corollary}

\begin{proof}
\emph{Step 1: energy of the distance $D$.} Set
\[
D_t=|X_{t-}-X'_{t-}|+|q_t-q'_t|+W_2(\mu_{t-},\mu'_{t-})+|\alpha_t-\alpha'_t|.
\]
By \eqref{eq:metriccoupling},
\begin{align*}
D_t^2
\le C\Bigl(&|X_{t-}-X'_{t-}|^2+W_2^2(\mu_{t-},\mu'_{t-})+|\alpha_t-\alpha'_t|^2+|\gamma-\gamma'|^{2\nu}\Bigr).
\end{align*}
Integrating and using \eqref{eq:forwardstab},
\begin{equation}\label{eq:D-energy}
\mathbb E^P\!\left[\int_0^T D_t^2\dd t\right]\le C\mathfrak R.
\end{equation}

\emph{Step 2: pointwise comparison of the densities.} For $r\in(0,1]$, \eqref{eq:branch-comparison} gives $|\chi_t-\chi'_t|\le C_1D_t+C_2I_t^r+C_2\widetilde I_t^r,$
where
\[
I_t^r=\1_{\{\rho^{q_t}(X_{t-},\mu_{t-},\alpha_t)\le r\}},
\qquad
\widetilde I_t^r=\1_{\{D_t\ge r/C_{\rm geo}\}},
\qquad C_2=2\sqrt d\,\overline v.
\]
Since
\[
K_t-K'_t=\int_0^t(\chi_s-\chi'_s)\dd s,
\]
we have
\begin{align*}
\sup_{t\le T}|K_t-K'_t|^2
&\le3C_1^2\left(\int_0^TD_t\dd t\right)^2
+3C_2^2\left(\int_0^TI_t^r\dd t\right)^2
+3C_2^2\left(\int_0^T\widetilde I_t^r\dd t\right)^2.
\end{align*}

\emph{Step 3: the three contributions.} The first term satisfies
\[
\mathbb E^P\!\left[\left(\int_0^TD_t\dd t\right)^2\right]
\le T \mathbb E^P\!\left[\int_0^TD_t^2\dd t\right]
\le CT\mathfrak R.
\]
Since $q\in\mathfrak A_{\gamma,M}^P$ is a regular selection with energy bound $M$ (Definition \ref{def:regular}) and $r\le1=r_0$, the second term satisfies, by \eqref{eq:occupation-second},
\[
\mathbb E^P\!\left[\left(\int_0^TI_t^r\dd t\right)^2\right]\le C_{\rm occ}^{(2)}r^2.
\]
For the third term,
\[
\left(\int_0^T\widetilde I_t^r\dd t\right)^2\le T\int_0^T\widetilde I_t^r\dd t,
\]
and
\[
\widetilde I_t^r\le \frac{C_{\rm geo}^2}{r^2}D_t^2.
\]
Therefore
\[
\mathbb E^P\!\left[\left(\int_0^T\widetilde I_t^r\dd t\right)^2\right]
\le \frac{TC_{\rm geo}^2}{r^2}\mathbb E^P\!\left[\int_0^TD_t^2\dd t\right]
\le C\mathfrak R r^{-2}.
\]
Consequently,
\[
\mathbb E^P\!\left[\sup_{t\le T}|K_t-K'_t|^2\right]
\le C\left(\mathfrak R+r^2+\mathfrak R r^{-2}\right).
\]

\emph{Step 4: choice of the layer width.} For $0<\mathfrak R\le1$, take $r=\mathfrak R^{1/4}$. Then
\[
r\le 1,
\qquad r^2=\mathfrak R^{1/2},
\qquad \mathfrak Rr^{-2}=\mathfrak R^{1/2},
\qquad \mathfrak R\le\mathfrak R^{1/2},
\]
and hence
\[
\mathbb E^P\!\left[\sup_{t\le T}|K_t-K'_t|^2\right]\le C\mathfrak R^{1/2}.
\]
\end{proof}

\subsection{Transfer to the McKean--Vlasov BSDE}

\begin{theorem}\label{thm:stability}
Fix $P\in\PP$ and suppose that $(H_{\rm quant})$ holds. Let $(\xi,\alpha),(\xi',\alpha')\in\mathfrak D_P$, $\gamma,\gamma'\in(0,1]$, let $q\in\mathfrak A_{\gamma,M}^P[\xi,\alpha]$ and $q'\in\mathfrak A_{\gamma',M}^P[\xi',\alpha']$ satisfy \eqref{eq:metriccoupling}, and let $(Y,Z,U,K)$ and $(Y',Z',U',K')$ be the corresponding selected solutions, constructed with the same $(B^P,\widetilde N^P)$. Let $\mathfrak R$ be defined by \eqref{eq:R-stab}. Then there exists $C=C(M;\mathfrak C_{\rm str})$ such that
\begin{align}
&\mathbb E^P\!\left[\sup_{t\le T}|Y_t-Y'_t|^2\right]
+\mathbb E^P\!\left[\int_0^T|Z_t-Z'_t|_{a_t^P}^2\dd t\right]
+\mathbb E^P\!\left[\int_0^T\|U_t-U'_t\|_\Pi^2\dd t\right]\notag\\
&\qquad
+\mathbb E^P\!\left[\sup_{t\le T}|K_t-K'_t|^2\right]
\le
C\mathfrak R^{1/2}.
\label{eq:full-stability}
\end{align}
\end{theorem}

\begin{proof}
\emph{Step 1: terminal condition.} Set
\[
\zeta=g(X_T,\mu_T),\qquad \zeta'=g(X'_T,\mu'_T).
\]
By the Lipschitz continuity of $g$,
\[
|\zeta-\zeta'|^2\le2L_B^2\left(|X_T-X'_T|^2+W_2^2(\mu_T,\mu'_T)\right),
\]
and \eqref{eq:forwardstab} gives
\begin{equation}\label{eq:terminal-stab}
\mathbb E^P\!\left[|\zeta-\zeta'|^2\right]\le C\mathfrak R.
\end{equation}

\emph{Step 2: the case $\mathfrak R=0$.} If $\mathfrak R=0$, then $\xi=\xi'$ $P$-a.s., $\alpha=\alpha'$ $\dd t\otimes\dd P$-a.e., and $\gamma=\gamma'$. Proposition \ref{prop:forward-stability} yields $X=X'$ and $\mu=\mu'$, while \eqref{eq:metriccoupling} yields $q=q'$ $\dd t\otimes\dd P$-a.e. Hence $K=K'$, uniqueness in Theorem \ref{thm:backward} gives $(Y,Z,U)=(Y',Z',U')$, and \eqref{eq:full-stability} is trivial.

\emph{Step 3: the case $0<\mathfrak R\le1$.} Assume next that $0<\mathfrak R\le1$. Then \eqref{eq:K-stability} gives
\[
\mathbb E^P\!\left[\sup_{t\le T}|K_t-K'_t|^2\right]\le C\mathfrak R^{1/2}.
\]
Applying Lemma \ref{lem:energy} to the two BSDEs gives
\begin{align*}
&\mathbb E^P\!\left[\sup_{t\le T}|Y_t-Y'_t|^2\right]
+\mathbb E^P\!\left[\int_0^T\left(|Z_t-Z'_t|_{a_t^P}^2+\|U_t-U'_t\|_\Pi^2\right)\dd t\right]\\
&\le C_E\Bigg(\mathbb E^P\!\left[|\zeta-\zeta'|^2\right]+\mathbb E^P\!\left[\sup_{t\le T}|K_t-K'_t|^2\right]\\
&\hspace{2.4cm}+\mathbb E^P\!\left[\int_0^T\left(|X_{t-}-X'_{t-}|^2+W_2^2(\mu_{t-},\mu'_{t-})+|\alpha_t-\alpha'_t|^2\right)\dd t\right]\Bigg).
\end{align*}
By \eqref{eq:forwardstab}, \eqref{eq:terminal-stab}, and $\mathfrak R\le\mathfrak R^{1/2}$,
\[
\mathbb E^P\!\left[\sup_{t\le T}|Y_t-Y'_t|^2\right]
+\mathbb E^P\!\left[\int_0^T\left(|Z_t-Z'_t|_{a_t^P}^2+\|U_t-U'_t\|_\Pi^2\right)\dd t\right]
\le C\mathfrak R^{1/2}.
\]
Adding \eqref{eq:K-stability} yields \eqref{eq:full-stability}.

\emph{Step 4: the case $\mathfrak R>1$.} If $\mathfrak R>1$, let $\mathcal Y=\mathcal T_P(q)$ and $\mathcal Y'=\mathcal T_P(q')$. The inequality $\|\mathcal Y-\mathcal Y'\|^2\le2\|\mathcal Y\|^2+2\|\mathcal Y'\|^2$ and \eqref{eq:Cap} give
\[
\|\mathcal Y-\mathcal Y'\|_{\mathfrak E_P^2}^2\le4C_{\rm ap}\le4C_{\rm ap}\mathfrak R^{1/2}.
\]
After replacing $C$ by $C\vee4C_{\rm ap}$, \eqref{eq:full-stability} holds for every $\mathfrak R>0$. Taking square roots gives $\|\mathcal Y-\mathcal Y'\|_{\mathfrak E_P^2}\le C\mathfrak R^{1/4}$.
\end{proof}

\subsection{Variation of the fuzzy level and Hausdorff distance}

For fixed initial data and control, the forward system and its law coincide. The coupling \eqref{eq:metriccoupling} gives
\[
|q_t-q'_t|\le L_Q|\gamma-\gamma'|^\nu\qquad \dd t\otimes\dd P\text{-a.e.}
\]
This bound is uniform in $(t,\omega)$, which yields a Lipschitz rate in $|\gamma-\gamma'|^\nu$ instead of the exponent $\nu/2$ that would follow from Theorem \ref{thm:stability}.

\begin{lemma}\label{lem:commonforward}
Suppose that $(H_{\rm quant})$ holds. There exists $C_\star=C_\star(M;\mathfrak C_{\rm str})$ such that, for every $P\in\PP$, $(\xi,\alpha)\in\mathfrak D_P$, with forward solution $X^P$ and law flow $\mu^P$, every $\gamma,\gamma'\in(0,1]$, every $\varepsilon\ge0$, and all regular admissible selections $q,q'$ at levels $\gamma,\gamma'$ for these data, with energy bound $M$, satisfying $\|q-q'\|_{\infty,P}\le\varepsilon$, one has
\begin{equation}\label{eq:commonforward-lip}
\|\mathcal T_P(q)-\mathcal T_P(q')\|_{\mathfrak E_P^2}\le C_\star\varepsilon.
\end{equation}
\end{lemma}
\begin{proof}
If $\varepsilon=0$, the sources coincide $\dd t\otimes\dd P$-a.e. and uniqueness gives the result. Assume henceforth that $\varepsilon>0$. If $C_{\rm geo}\varepsilon\le 1/2$, take $r=2C_{\rm geo}\varepsilon$. Then $D_t\le\varepsilon<r/C_{\rm geo}$ and the crossing indicator in Lemma \ref{lem:branch} vanishes. Hence
\[
|\chi_t-\chi'_t|\le C_1\varepsilon+2\sqrt d\,\overline v\1_{\{\rho^{q_t}\le2C_{\rm geo}\varepsilon\}}.
\]
Integrating in time,
\[
\sup_{t\le T}|K_t-K'_t|\le C_1T\varepsilon+2\sqrt d\,\overline v\int_0^T\1_{\{\rho^{q_t}\le2C_{\rm geo}\varepsilon\}}\dd t .
\]
Squaring, using $(x_1+x_2)^2\le2x_1^2+2x_2^2$ and \eqref{eq:occupation-second} with $r=2C_{\rm geo}\varepsilon\le1$, which applies to $q$ by Remark \ref{rem:occupation-scope},
\[
\mathbb E^P\!\left[\sup_{t\le T}|K_t-K'_t|^2\right]
\le2C_1^2T^2\varepsilon^2+8d\,\overline v^2C_{\rm occ}^{(2)}(2C_{\rm geo}\varepsilon)^2
=C(M;\mathfrak C_{\rm str})\varepsilon^2,
\]
with $C(M;\mathfrak C_{\rm str})=2C_1^2T^2+32d\,\overline v^2C_{\rm geo}^2C_{\rm occ}^{(2)}$.
Lemma \ref{lem:energy}, with identical forward state and control, gives \eqref{eq:commonforward-lip}. If $C_{\rm geo}\varepsilon>1/2$, then \eqref{eq:Cap} gives $\|\mathcal T_P(q)-\mathcal T_P(q')\|_{\mathfrak E_P^2}\le2C_{\rm ap}^{1/2}<4C_{\rm geo}C_{\rm ap}^{1/2}\varepsilon$. Thus \eqref{eq:commonforward-lip} holds in all cases with $C_\star(M)=\max\{((1+C_E)C)^{1/2},4C_{\rm geo}C_{\rm ap}^{1/2}\}$, where $C$ is the constant in the preceding display.
\end{proof}

\begin{corollary}\label{cor:lipschitzT}
Fix $P\in\PP$ and data $(\xi,\alpha)\in\mathfrak D_P$, and suppose that $(H_{\rm quant})$ holds. Then, for every $\gamma\in(0,1]$, the solution map $\mathcal T_P:(\mathfrak A_{\gamma,M}^P,\|\cdot\|_{\infty,P})\to\mathfrak E_P^2$ is Lipschitz with constant $C_\star$. Moreover, for every $\gamma,\gamma'\in(0,1]$ and every $q\in\mathfrak A_{\gamma,M}^P$, there exists $q'\in\mathfrak A_{\gamma',M}^P$ such that
\begin{equation}\label{eq:gamma-lip}
\|\mathcal T_P(q)-\mathcal T_P(q')\|_{\mathfrak E_P^2}\le C_\star L_Q|\gamma-\gamma'|^\nu.
\end{equation}
\end{corollary}

\begin{proof}
The elements of $\mathfrak A_{\gamma,M}^P$ are associated with the same data $(\xi,\alpha)$, hence with the same $(X^P,\mu^P)$, and satisfy the same energy bound by (S1); Lemma \ref{lem:commonforward} gives the first assertion. For the second, (S4) applied with the same data provides $q'\in\mathfrak A_{\gamma',M}^P$ with $\|q-q'\|_{\infty,P}\le L_Q|\gamma-\gamma'|^\nu$, and Lemma \ref{lem:commonforward} with $\varepsilon=L_Q|\gamma-\gamma'|^\nu$ concludes.
\end{proof}

\begin{proposition}\label{prop:hausdorff-stability}
Fix $P\in\PP$ and data $(\xi,\alpha)\in\mathfrak D_P$, and suppose that $(H_{\rm quant})$ holds. Then, for all $\gamma,\gamma'\in(0,1]$,
\begin{equation}\label{eq:hausdorff-gamma}
d_H^{\mathfrak E_P^2}(\cC_{\gamma,M}^P,\cC_{\gamma',M}^P)\le C_\star L_Q|\gamma-\gamma'|^\nu.
\end{equation}
\end{proposition}
\begin{proof}
Let $\mathcal Y=\mathcal T_P(q)\in\cC_{\gamma,M}^P$ with $q\in\mathfrak A_{\gamma,M}^P$. By \eqref{eq:gamma-lip}, there exists $q'\in\mathfrak A_{\gamma',M}^P$ such that $\operatorname{dist}(\mathcal Y,\cC_{\gamma',M}^P)\le\|\mathcal Y-\mathcal T_P(q')\|_{\mathfrak E_P^2}\le C_\star L_Q|\gamma-\gamma'|^\nu$. Taking the supremum over $\mathcal Y$ bounds the excess $e(\cC_{\gamma,M}^P,\cC_{\gamma',M}^P)$; exchanging $\gamma$ and $\gamma'$ bounds the reverse excess, which proves \eqref{eq:hausdorff-gamma}.
\end{proof}
\section{The fuzzy set of solutions}\label{sec:coupes}

Recall that $\cC_{\gamma,M}^P=\mathcal T_P(\mathfrak A_{\gamma,M}^P)$; see \eqref{eq:cuts-def}.

\begin{proposition}
\label{prop:hausdorff-data}
Suppose that $(H_{\rm quant})$ holds. For each $P\in\PP$ and $(\xi,\alpha),(\xi',\alpha')\in\mathfrak D_P$, whose forward systems are constructed with the same $(B^P,\widetilde N^P)$, denote by $\cC_{\gamma,M}^P[\xi,\alpha]$ and $\cC_{\gamma,M}^P[\xi',\alpha']$ the corresponding solution cuts in $\mathfrak E_P^2$, and set
\[
\delta_P=\mathbb E^P\!\left[|\xi-\xi'|^2\right]
+\mathbb E^P\!\left[\int_0^T|\alpha_t-\alpha'_t|^2\dd t\right].
\]
Then there exists $C=C(M;\mathfrak C_{\rm str})$, independent of $P$ and of the two data sets, such that, for all $\gamma,\gamma'\in(0,1]$,
\[
d_H^{\mathfrak E_P^2}
\bigl(\cC_{\gamma,M}^P[\xi,\alpha],
\cC_{\gamma',M}^P[\xi',\alpha']\bigr)
\le C\delta_P^{1/4}+C|\gamma-\gamma'|^\nu.
\]
\end{proposition}
\begin{proof}
Fix $P$ and the two data sets, and first consider the same level $\gamma$. Let $\mathcal Y=\mathcal T_P(q)\in\cC_{\gamma,M}^P[\xi,\alpha]$ with $q\in\mathfrak A_{\gamma,M}^P[\xi,\alpha]$. By (S4), there exists $q'\in\mathfrak A_{\gamma,M}^P[\xi',\alpha']$ satisfying \eqref{eq:metriccoupling} with $\gamma'=\gamma$; denote its solution by $\mathcal Y'$. Theorem \ref{thm:stability}, with $\mathfrak R=\delta_P$, implies $\|\mathcal Y-\mathcal Y'\|_{\mathfrak E_P^2}^2\le C\delta_P^{1/2}$, with a constant independent of $q$. Hence
\[
\sup_{\mathcal Y\in\cC_{\gamma,M}^P[\xi,\alpha]}
\inf_{\mathcal V\in\cC_{\gamma,M}^P[\xi',\alpha']}
\|\mathcal Y-\mathcal V\|_{\mathfrak E_P^2}
\le C\delta_P^{1/4}.
\]
Exchanging the two data sets yields the reverse directed distance, and therefore the Hausdorff bound at the same level. To change the level, apply \eqref{eq:hausdorff-gamma} to the primed data and then use the triangle inequality:
\begin{align*}
d_H\bigl(\cC_{\gamma,M}^P[\xi,\alpha],
\cC_{\gamma',M}^P[\xi',\alpha']\bigr)
&\le d_H\bigl(\cC_{\gamma,M}^P[\xi,\alpha],
\cC_{\gamma,M}^P[\xi',\alpha']\bigr)\\
&\quad+d_H\bigl(\cC_{\gamma,M}^P[\xi',\alpha'],
\cC_{\gamma',M}^P[\xi',\alpha']\bigr)\\
&\le C\delta_P^{1/4}+C|\gamma-\gamma'|^\nu.
\end{align*}
\end{proof}

\begin{theorem}\label{thm:cutcompact}
Fix $P\in\PP$ and data $(\xi,\alpha)\in\mathfrak D_P$, and suppose that $(H_{\rm quant})$ holds. Then, for every $\gamma\in(0,1]$, the solution cut $\cC_{\gamma,M}^P$ is nonempty and compact in $\mathfrak E_P^2$.
\end{theorem}

\begin{proof}
By (S1),
\[
\varnothing\ne\mathfrak A_{\gamma,M}^P\subset L^\infty(\dd t\otimes\dd P)
\]
is compact. For every $q\in\mathfrak A_{\gamma,M}^P$, Corollary \ref{cor:wellposed-selected} defines $\mathcal T_P(q)\in\mathfrak E_P^2$; hence $\cC_{\gamma,M}^P=\mathcal T_P(\mathfrak A_{\gamma,M}^P)\ne\varnothing.$
Let $(\mathcal Y^n)_{n\ge1}\subset\cC_{\gamma,M}^P$. Choose $q^n\in\mathfrak A_{\gamma,M}^P$ such that $\mathcal Y^n=\mathcal T_P(q^n).$ By compactness, there exist a subsequence $(q^{n_k})$ and $q\in\mathfrak A_{\gamma,M}^P$ such that
\[
\|q^{n_k}-q\|_{\infty,P}\longrightarrow0.
\]
Corollary \ref{cor:lipschitzT} gives
\[
\|\mathcal Y^{n_k}-\mathcal T_P(q)\|_{\mathfrak E_P^2}
\le C_\star\|q^{n_k}-q\|_{\infty,P}
\longrightarrow0.
\]
Thus every sequence in $\cC_{\gamma,M}^P$ has a subsequence converging to an element of $\cC_{\gamma,M}^P$; in the metric space $\mathfrak E_P^2$, this is equivalent to compactness.
\end{proof}

The nesting property
\[
\gamma_1\le\gamma_2
\quad\Longrightarrow\quad
\mathfrak A_{\gamma_2,M}^P\subseteq\mathfrak A_{\gamma_1,M}^P
\]
immediately implies $\cC_{\gamma_2,M}^P\subseteq\cC_{\gamma_1,M}^P.$

\begin{proposition}\label{prop:leftcuts}
Fix $P\in\PP$ and data $(\xi,\alpha)\in\mathfrak D_P$, and suppose that $(H_{\rm quant})$ holds. Then, for every $\gamma\in(0,1]$,
\[
\mathfrak A_{\gamma,M}^P
=
\bigcap_{0<\iota<\gamma}\mathfrak A_{\iota,M}^P
\qquad\text{and}\qquad
\cC_{\gamma,M}^P
=
\bigcap_{0<\iota<\gamma}\cC_{\iota,M}^P.
\]
\end{proposition}

\begin{proof}
The first identity is (S3). Fix a sequence $\iota_n\uparrow\gamma$. Nesting gives
\[
\cC_{\gamma,M}^P\subseteq\cC_{\iota_n,M}^P,
\qquad n\ge1,
\]
and therefore
\[
\cC_{\gamma,M}^P\subseteq\bigcap_{n\ge1}\cC_{\iota_n,M}^P.
\]

Conversely, let
\[
\mathcal Y\in\bigcap_{n\ge1}\cC_{\iota_n,M}^P.
\]
For each $n$, choose $q^n\in\mathfrak A_{\iota_n,M}^P$ such that $\mathcal T_P(q^n)=\mathcal Y.$ Since $\iota_1\le\iota_n$, $q^n\in\mathfrak A_{\iota_n,M}^P\subseteq\mathfrak A_{\iota_1,M}^P.$ Compactness of $\mathfrak A_{\iota_1,M}^P$ provides a subsequence $(q^{n_k})$ and $q\in\mathfrak A_{\iota_1,M}^P$ such that
\[
\|q^{n_k}-q\|_{\infty,P}\longrightarrow0.
\]
Fix $j\ge1$. For $k$ sufficiently large, $n_k\ge j$ and hence
\[
q^{n_k}\in\mathfrak A_{\iota_{n_k},M}^P
\subseteq\mathfrak A_{\iota_j,M}^P.
\]
Compactness of $\mathfrak A_{\iota_j,M}^P$ implies closedness; passing to the limit yields $q\in\mathfrak A_{\iota_j,M}^P.$ Since $j$ is arbitrary,
\[
q\in\bigcap_{j\ge1}\mathfrak A_{\iota_j,M}^P
=\mathfrak A_{\gamma,M}^P.
\]
Continuity of $\mathcal T_P$ gives
\[
\mathcal T_P(q^{n_k})\longrightarrow\mathcal T_P(q)
\quad\text{in }\mathfrak E_P^2.
\]
But $\mathcal T_P(q^{n_k})=\mathcal Y$ for every $k$; hence $\mathcal Y=\mathcal T_P(q)\in\cC_{\gamma,M}^P.$ Consequently,
\[
\cC_{\gamma,M}^P=\bigcap_{n\ge1}\cC_{\iota_n,M}^P
=\bigcap_{0<\iota<\gamma}\cC_{\iota,M}^P.
\]
\end{proof}

We can therefore define the fuzzy set of stochastic trajectories $\widetilde{\mathcal Y}^{\,P,M}[\xi,\alpha]$, denoted by $\widetilde{\mathcal Y}^{\,P,M}$ when $P$ and the data are fixed, by
\[
\mathfrak m_{\widetilde{\mathcal Y}^{\,P,M}[\xi,\alpha]}(\mathcal Y)
=
\sup\{\gamma\in(0,1]:\mathcal Y\in\cC_{\gamma,M}^P[\xi,\alpha]\},\qquad \sup\varnothing=0.
\]
For $\gamma\in(0,1]$, the $\gamma$-cut of $\widetilde{\mathcal Y}^{\,P,M}[\xi,\alpha]$ is $[\widetilde{\mathcal Y}^{\,P,M}[\xi,\alpha]]_\gamma:=\{\mathcal Y\in\mathfrak E_P^2:\mathfrak m_{\widetilde{\mathcal Y}^{\,P,M}[\xi,\alpha]}(\mathcal Y)\ge\gamma\}$.

\begin{lemma}\label{lem:membership}
Fix $P\in\PP$ and data $(\xi,\alpha)\in\mathfrak D_P$, and suppose that $(H_{\rm quant})$ holds. Then, for every $\gamma\in(0,1]$, $[\widetilde{\mathcal Y}^{\,P,M}]_\gamma=\cC_{\gamma,M}^P.$ Moreover, if $\gamma_{\mathcal Y}:=\mathfrak m_{\widetilde{\mathcal Y}^{\,P,M}}(\mathcal Y)>0$, the supremum defining it is attained, i.e. $\mathcal Y\in\cC_{\gamma_{\mathcal Y},M}^P$.
\end{lemma}

\begin{proof}
If $\mathcal Y\in\cC_{\gamma,M}^P$, then $\mathfrak m_{\widetilde{\mathcal Y}^{\,P,M}}(\mathcal Y)\ge\gamma$ by definition. Conversely, suppose that $\mathfrak m_{\widetilde{\mathcal Y}^{\,P,M}}(\mathcal Y)\ge\gamma>0$ and let $\iota\in(0,\gamma)$. Since $\iota$ is strictly below the supremum, there exists $\iota'>\iota$ such that $\mathcal Y\in\cC_{\iota',M}^P$, and nesting gives $\mathcal Y\in\cC_{\iota,M}^P$. Hence $\mathcal Y\in\bigcap_{0<\iota<\gamma}\cC_{\iota,M}^P=\cC_{\gamma,M}^P$ by Proposition~\ref{prop:leftcuts}. The second assertion is the case $\gamma=\gamma_{\mathcal Y}$.
\end{proof}
Here the term \emph{fuzzy set} is used in the sense of \cite{DiamondKloeden1994}, namely a fuzzy set in the metric space $\mathfrak E_P^2$ whose positive cuts are nonempty, compact, nested, and left-continuous. The cuts need not be convex, since the solution map is nonlinear.

\begin{proposition}
\label{prop:fuzzy-standard}
Fix $P\in\PP$ and data $(\xi,\alpha)\in\mathfrak D_P$, and suppose that $(H_{\rm quant})$ holds. Let $\widetilde{\mathcal Y}^{\,P,M}$ be the fuzzy set on $\mathfrak E_P^2$ with membership function $\mathfrak m_{\widetilde{\mathcal Y}^{\,P,M}}$ defined above. Then $\widetilde{\mathcal Y}^{\,P,M}$ is normal and upper semicontinuous, and its support
\[
[\widetilde{\mathcal Y}^{\,P,M}]_0
:=\overline{\bigcup_{\gamma\in(0,1]}\cC_{\gamma,M}^P}
\]
is compact in $\mathfrak E_P^2$.
\end{proposition}

\begin{proof}
For every predictable process $q$ with values in $\Q$, define $\mathcal T_P(q)$ through the same source and the same BSDE: the proofs of Proposition \ref{prop:K} and Theorem \ref{thm:backward} use only these two properties of $q$, not its membership in a cut. Likewise, the proofs of Theorem \ref{thm:occupation} and Lemma \ref{lem:commonforward} use only that $q$ takes values in $\Q$ and satisfies the energy bound $M$, which hold for the elements of $\overline{\mathfrak A}_M^P$ by (S5). The map $\mathcal T_P$ is therefore Lipschitz, and in particular continuous, on the compact set $\overline{\mathfrak A}_M^P$; hence $\mathcal T_P(\overline{\mathfrak A}_M^P)$ is a compact subset of $\mathfrak E_P^2$ containing every cut $\cC_{\gamma,M}^P$. Its support, being a closed subset of this compact set, is therefore compact. Nonemptiness of $\cC_{1,M}^P$ gives normality. Finally, the positive cuts are closed by Theorem \ref{thm:cutcompact}, and $[\widetilde{\mathcal Y}^{\,P,M}]_\gamma=\cC_{\gamma,M}^P$ by Lemma \ref{lem:membership}. Since all its positive upper level sets are closed, the membership function is upper semicontinuous.
\end{proof}
\section{Yosida approximation and approximation of the cuts}\label{sec:yosida}

For a fixed selection, the regularization acts only on the operator: the forward state $X^P$ and its law $\mu^P$ remain unchanged.

\subsection{Resolvent, translation, and minimal selection}

Fix $(q,\mu,\alpha)$ and write
\[
h=h^q(\mu,\alpha),\qquad
c=c^q(\mu,\alpha).
\]
Then $A^q(x,\mu,\alpha)=\partial\phi^q(x+h)+c.$ For $\lambda>0$, $J_\lambda^{q,\mu,\alpha}$ and $A_\lambda^q$ are defined in \eqref{eq:resolvent-def}.

\begin{lemma}\label{lem:resolventidentity}
Suppose that Assumption \ref{ass:operator} holds. Let $q\in\Q$, $\mu\in\Ptwo(\R^d)$, $\alpha\in\mathbb D$, and set $h=h^q(\mu,\alpha)$ and $c=c^q(\mu,\alpha)$. Then, for every $\lambda>0$ and every $x\in\R^d$,
\begin{align}
J_\lambda^{q,\mu,\alpha}(x)
&=
J_{\lambda\partial\phi^q}(x+h-\lambda c)-h,
\label{eq:resolvent-translation}\\
A_\lambda^q(x,\mu,\alpha)
&=
(\partial\phi^q)_\lambda(x+h-\lambda c)+c.
\label{eq:yosida-translation}
\end{align}
\end{lemma}

\begin{proof}
Set $y=J_\lambda^{q,\mu,\alpha}(x)$. By definition,
\[
x-y\in\lambda A^q(y,\mu,\alpha)
=
\lambda\bigl(\partial\phi^q(y+h)+c\bigr).
\]
Hence $x+h-\lambda c-(y+h)\in\lambda\partial\phi^q(y+h)$, that is, $y+h=J_{\lambda\partial\phi^q}(x+h-\lambda c)$. This proves \eqref{eq:resolvent-translation}. Next,
\begin{align*}
A_\lambda^q(x,\mu,\alpha)
&=\frac{x-y}{\lambda}\\
&=\frac{x+h-\lambda c-J_{\lambda\partial\phi^q}(x+h-\lambda c)}{\lambda}+c\\
&=(\partial\phi^q)_\lambda(x+h-\lambda c)+c.
\end{align*}
\end{proof}

\begin{lemma}\label{lem:yosida-minimal}
Suppose that Assumption \ref{ass:operator} holds. Let $q\in\Q$, $\mu\in\Ptwo(\R^d)$, $\alpha\in\mathbb D$, and $x\in\R^d$. Then, for every $\lambda>0$,
\begin{equation}\label{eq:Y-min-bound}
|A_\lambda^q(x,\mu,\alpha)|
\le |(A^q)^0(x,\mu,\alpha)|,
\end{equation}
and
\begin{equation}\label{eq:Y-point-conv}
A_\lambda^q(x,\mu,\alpha)
\longrightarrow
(A^q)^0(x,\mu,\alpha)
\qquad\text{as }\lambda\downarrow0.
\end{equation}
\end{lemma}

\begin{proof}
Since $A^q(\cdot,\mu,\alpha)$ is maximal monotone on $\R^d$, both assertions are classical properties of the Yosida approximation \cite[chap.~II]{Brezis1973}; we include the proof. Set $y_\lambda=J_\lambda^{q,\mu,\alpha}(x)$, $v_\lambda=A_\lambda^q(x,\mu,\alpha)$, and $v^0=(A^q)^0(x,\mu,\alpha)$. Then $v_\lambda\in A^q(y_\lambda,\mu,\alpha)$ and $x-y_\lambda=\lambda v_\lambda$, so monotonicity gives $\langle v_\lambda-v^0,y_\lambda-x\rangle=-\lambda\langle v_\lambda-v^0,v_\lambda\rangle\ge0$, that is, $|v_\lambda|^2\le\langle v^0,v_\lambda\rangle\le|v^0||v_\lambda|$. For \eqref{eq:Y-point-conv}, the preceding bound gives $y_\lambda=x-\lambda v_\lambda\to x$. Let $\lambda_n\downarrow0$; the bounded sequence $(v_{\lambda_n})$ has a subsequence converging to some $v$. For every $(z,w)$ in the graph of $A^q(\cdot,\mu,\alpha)$, monotonicity yields $\langle v_{\lambda_n}-w,y_{\lambda_n}-z\rangle\ge0$, hence $\langle v-w,x-z\rangle\ge0$ in the limit; maximality implies $v\in A^q(x,\mu,\alpha)$. Since $|v|\le|v^0|$ and the minimum-norm element of the closed convex set $A^q(x,\mu,\alpha)$ is unique, $v=v^0$. Since every cluster point equals $v^0$, we obtain $v_\lambda\to v^0$.
\end{proof}

\begin{corollary}\label{cor:regularized-wellposed}
Fix $P\in\PP$ and suppose that the assumptions of Theorem \ref{thm:backward} hold. Let $\xi\in L^2(\Omega,\mathcal F_0,P;\R^d)$, $\alpha\in\mathcal A_2$, $\gamma\in(0,1]$, let $q$ be an admissible epistemic selection at level $\gamma$, and let $\lambda>0$. Then the process $A_\lambda^{q_t}(X_{t-}^P,\mu_{t-}^P,\alpha_t)$ is predictable and satisfies
\[
|A_\lambda^{q_t}(X_{t-}^P,\mu_{t-}^P,\alpha_t)|\le|(A^{q_t})^0(X_{t-}^P,\mu_{t-}^P,\alpha_t)|,
\qquad\dd t\otimes\dd P\text{-a.e.}
\]
Consequently, the system admits a unique regularized solution in the sense of Definition \ref{def:solution}, and this solution satisfies \eqref{eq:K-apriori} and \eqref{eq:backward-apriori} with the same constants, independently of $\lambda$.
\end{corollary}

\begin{proof}
Predictability follows from Assumption \ref{ass:measop} and Remark \ref{rem:measop}(iii), while the domination follows from \eqref{eq:Y-min-bound}. Hence the $L^2(\dd t\otimes\dd P)$ norm of the regularized density is bounded by the right-hand side of \eqref{eq:K-apriori}, and its time integral $K^\lambda$ is adapted, continuous, of finite variation, and belongs to $\cS_P^2$, with bound \eqref{eq:K-apriori}. The proof of Theorem \ref{thm:backward} uses the source only through these properties and through the term $|\chi_t|$ in the process $\mathsf D_t$ of Step 3, which is replaced by the smaller regularized density. It therefore applies verbatim and yields existence, uniqueness, and \eqref{eq:backward-apriori}. The forward component is that of Theorem \ref{thm:forward}, which does not depend on the source.
\end{proof}

\subsection{Qualitative convergence for a fixed selection}

For an admissible selection $q$, define
\[
\chi_t^\lambda
=
A_\lambda^{q_t}(X_{t-}^P,\mu_{t-}^P,\alpha_t),
\qquad
\chi_t
=
(A^{q_t})^0(X_{t-}^P,\mu_{t-}^P,\alpha_t),
\]
and
\[
K_t^\lambda=\int_0^t\chi_s^\lambda\dd s,
\qquad
K_t=\int_0^t\chi_s\dd s.
\]

\begin{theorem}\label{thm:yosidaqual}
Fix $P\in\PP$ and suppose that the assumptions of Theorem \ref{thm:backward} hold. Let $\xi\in L^2(\Omega,\mathcal F_0,P;\R^d)$, $\alpha\in\mathcal A_2$, $\gamma\in(0,1]$, and let $q$ be an admissible epistemic selection at level $\gamma$. For $\lambda>0$, let $(Y^\lambda,Z^\lambda,U^\lambda,K^\lambda)$ be the regularized solution from Corollary \ref{cor:regularized-wellposed} and $(Y,Z,U,K)$ the selected solution (Definition \ref{def:solution}), and let $\chi^\lambda,\chi$ be the corresponding source densities. Then, as $\lambda\downarrow0$,
\begin{equation}\label{eq:Y-qual-source}
\mathbb E^P\!\left[\int_0^T|\chi_t^\lambda-\chi_t|^2\dd t\right]\longrightarrow0,
\end{equation}
\begin{equation}\label{eq:Y-qual-K}
\mathbb E^P\!\left[\sup_{t\le T}|K_t^\lambda-K_t|^2\right]\longrightarrow0,
\end{equation}
and
\begin{equation}\label{eq:Y-qual-BSDE}
\|(Y^\lambda,Z^\lambda,U^\lambda,K^\lambda)
-(Y,Z,U,K)\|_{\mathfrak E_P^2}\longrightarrow0.
\end{equation}
\end{theorem}

\begin{proof}
For $\dd t\otimes\dd P$-almost every $(t,\omega)$, Lemma \ref{lem:yosida-minimal} gives
\[
\chi_t^\lambda\longrightarrow\chi_t,
\qquad
|\chi_t^\lambda|\le|\chi_t|.
\]
Thus $|\chi_t^\lambda-\chi_t|^2\le4|\chi_t|^2.$ Proposition \ref{prop:K} ensures
\[
\mathbb E^P\!\left[\int_0^T|\chi_t|^2\dd t\right]<\infty.
\]
Dominated convergence yields \eqref{eq:Y-qual-source}. Next,
\[
\sup_{t\le T}|K_t^\lambda-K_t|^2
\le
T\int_0^T|\chi_t^\lambda-\chi_t|^2\dd t,
\]
which proves \eqref{eq:Y-qual-K}.

It remains to transfer this convergence to the BSDE. The two equations have the same forward system, the same flow $\mu^P$, the same control, and the same terminal condition. Lemma \ref{lem:energy}, applied with
\[
\zeta^1=\zeta^2=g(X_T^P,\mu_T^P),\qquad
X^1=X^2=X^P,\qquad \mu^1=\mu^2=\mu^P,\qquad \alpha^1=\alpha^2=\alpha,
\]
gives directly
\[
\|(Y^\lambda-Y,Z^\lambda-Z,U^\lambda-U)\|_{\cS_P^2\times\cH_P^2\times\cJ_P^2}^2
\le C \mathbb E^P\!\left[\sup_{t\le T}|K_t^\lambda-K_t|^2\right].
\]
Adding the $K$ component and using \eqref{eq:Y-qual-K} yields \eqref{eq:Y-qual-BSDE}.
\end{proof}

\subsection{Same-branch analysis}

\begin{lemma}\label{lem:samebranch}\label{lem:pointY}
Suppose that Assumptions \ref{ass:operator} and \ref{ass:finite-branches} hold, and set $M_\Theta=1+\alpha_A\Theta_*+C_c+\overline v$. Let $\lambda>0$ and $(q,x,\mu,\alpha)\in\Q\times\R^d\times\Ptwo(\R^d)\times\mathbb D$. Then:
\begin{enumerate}[label=(\alph*)]
\item if $\rho^q(x,\mu,\alpha)>M_\Theta\lambda$, then $A_\lambda^q(x,\mu,\alpha)=(1+\alpha_A\lambda)^{-1}(A^q)^0(x,\mu,\alpha)$;
\item in all cases,
\begin{equation}\label{eq:Y-point}
|A_\lambda^q(x,\mu,\alpha)-(A^q)^0(x,\mu,\alpha)|\le\alpha_A\lambda\,|(A^q)^0(x,\mu,\alpha)|+2\sqrt d\,\overline v\,\1_{\{\rho^q(x,\mu,\alpha)\le M_\Theta\lambda\}} .
\end{equation}
\end{enumerate}
\end{lemma}

\begin{proof}
Set $z=x+h^q(\mu,\alpha)$, $c=c^q(\mu,\alpha)$, and $w=J_{\lambda\partial\phi^q}(z-\lambda c)$. By Lemma \ref{lem:resolventidentity}, $J_\lambda^{q,\mu,\alpha}(x)=w-h^q(\mu,\alpha)$, hence $A_\lambda^q(x,\mu,\alpha)=\lambda^{-1}(z-w)$. The definition of the resolvent provides $v\in\prod_k\partial \mathfrak g_k^q(w_k)$, with $|v_k|\le\overline v$, such that
\begin{equation}\label{eq:poly-resolvent}
z=(1+\lambda\alpha_A)w+\lambda(c+v),
\qquad\text{that is,}\qquad
A_\lambda^q(x,\mu,\alpha)=\alpha_Aw+c+v .
\end{equation}
Let $\theta\in\Theta_k$. If $z_k>\theta+M_\Theta\lambda$ and $w_k\le\theta$, then \eqref{eq:poly-resolvent} gives
\[
z_k\le(1+\lambda\alpha_A)\theta+\lambda(C_c+\overline v)\le\theta+\lambda(\alpha_A\Theta_*+C_c+\overline v)<\theta+M_\Theta\lambda,
\]
a contradiction. Likewise, $z_k<\theta-M_\Theta\lambda$ and $w_k\ge\theta$ imply $z_k\ge\theta-\lambda(\alpha_A\Theta_*+C_c+\overline v)>\theta-M_\Theta\lambda$, again a contradiction. Under the assumption in (a), $z_k$ and $w_k$ are therefore strictly on the same side of every threshold: they belong to the same interval $I_{k,j}$, on which $\partial \mathfrak g_k^q$ is the singleton $\{v_{k,j}^q\}$. Thus $v_k=v^q_{k,j}$ and, since $z_k$ is not a threshold, $(A^q)^0_k(x,\mu,\alpha)=\alpha_Az_k+c_k+v^q_{k,j}$. Hence $A_\lambda^q-(A^q)^0=\alpha_A(w-z)=-\alpha_A\lambda A_\lambda^q$, which proves (a).

In general, $(A^q)^0(x,\mu,\alpha)=\alpha_Az+c+v^0$ with $v^0\in\prod_k\partial \mathfrak g_k^q(z_k)$, so $|v^0|\le\sqrt d\,\overline v$. Hence $A_\lambda^q-(A^q)^0=-\alpha_A\lambda A_\lambda^q+(v-v^0)$, where $v=v^0$ outside the layer $\{\rho^q\le M_\Theta\lambda\}$ by the preceding argument, and $|v-v^0|\le2\sqrt d\,\overline v$ inside the layer. Lemma \ref{lem:yosida-minimal} gives $|A_\lambda^q|\le|(A^q)^0|$, which proves (b).
\end{proof}

\subsection{Uniform Yosida rate}

\begin{theorem}\label{thm:yosidarate}
Suppose that $(H_{\rm quant})$ holds. For each admissible choice of $P\in\PP$, $(\xi,\alpha)\in\mathfrak D_P$, $\gamma\in(0,1]$, $q\in\mathfrak A_{\gamma,M}^P[\xi,\alpha]$, and $\lambda\in(0,1]$, let $\mathcal Y^{P,q}=(Y,Z,U,K)$ be the selected solution and $\mathcal Y^{P,q,\lambda}=(Y^\lambda,Z^\lambda,U^\lambda,K^\lambda)$ the regularized solution. Then there exists $C=C(M;\mathfrak C_{\rm str})$, independent of $P$, $(\xi,\alpha)$, $\gamma$, $q$, and $\lambda$, such that
\begin{equation}\label{eq:Y-rate-K-strong}
\sup_{P,\gamma,q}\mathbb E^P\!\left[\sup_{t\le T}|K_t^\lambda-K_t|^2\right]\le C\lambda^2,\qquad 0<\lambda\le1,
\end{equation}
and
\begin{equation}\label{eq:Y-linear-occ}
\sup_{P,\gamma,q}
\|\mathcal Y^{P,q,\lambda}-\mathcal Y^{P,q}\|_{\mathfrak E_P^2}
\le C\lambda,\qquad 0<\lambda\le1.
\end{equation}
\end{theorem}

\begin{proof}
\emph{Step 1: uniform bounds on the source densities.} Fix $P\in\PP$, $(\xi,\alpha)\in\mathfrak D_P$, $\gamma\in(0,1]$, and $q\in\mathfrak A_{\gamma,M}^P[\xi,\alpha]$. Keep the notation $\chi^\lambda,\chi,K^\lambda,K$ introduced in the qualitative convergence result. Proposition \ref{prop:K}, \eqref{eq:forward-apriori}, and the uniform bounds on the data ensure that
\[
C_\chi:=\sup_{P,\gamma,q}\mathbb E^P\!\left[
\int_0^T|(A^{q_t})^0(X_{t-}^P,\mu_{t-}^P,\alpha_t)|^2\dd t\right]<\infty.
\]
Lemma \ref{lem:yosida-minimal}, applied for $\dd t\otimes\dd P$-almost every $(t,\omega)$, gives
\begin{equation}\label{eq:Y-minimal-domination}
|A_\lambda^{q_t}(X_{t-}^P,\mu_{t-}^P,\alpha_t)|
\le |(A^{q_t})^0(X_{t-}^P,\mu_{t-}^P,\alpha_t)|.
\end{equation}
This domination yields the same bound $C_\chi$ for $\mathbb E^P\!\left[\int_0^T|\chi_t^\lambda|^2\dd t\right]$, uniformly in $\lambda$.

\emph{Step 2: small $\lambda$: integrate before squaring.} Set $C_2=2\sqrt d\,\overline v$ and $\lambda_0=\min\{1,M_\Theta^{-1}\}>0$. Lemma \ref{lem:samebranch}(b) gives
\[
|\chi_t^\lambda-\chi_t|
\le\alpha_A\lambda|\chi_t|
+C_2\1_{\{\rho^{q_t}(X_{t-}^P,\mu_{t-}^P,\alpha_t)\le M_\Theta\lambda\}}.
\]
For $0<\lambda\le\lambda_0$, integrate this inequality before squaring it. By the definition of $K^\lambda$ and $K$,
\[
\sup_{u\le T}|K_u^\lambda-K_u|
\le \alpha_A\lambda\int_0^T|\chi_t|\dd t
+C_2\int_0^T
\1_{\{\rho^{q_t}(X_{t-}^P,\mu_{t-}^P,\alpha_t)\le M_\Theta\lambda\}}\dd t.
\]
The inequality $(a+b)^2\le2a^2+2b^2$ yields
\begin{align}
\mathbb E^P\!\left[\sup_{u\le T}|K_u^\lambda-K_u|^2\right]
\le{}&2\alpha_A^2\lambda^2 \mathbb E^P\!\left[
\left(\int_0^T|\chi_t|\dd t\right)^2\right]\notag\\
&+2C_2^2 \mathbb E^P\!\left[\left(\int_0^T
\1_{\{\rho^{q_t}(X_{t-}^P,\mu_{t-}^P,\alpha_t)\le M_\Theta\lambda\}}\dd t
\right)^2\right].
\label{eq:Y-K-secondmoment}
\end{align}
By Cauchy--Schwarz in time and the definition of $C_\chi$,
\[
\mathbb E^P\!\left[\left(\int_0^T|\chi_t|\dd t\right)^2\right]\le T\,\mathbb E^P\!\left[\int_0^T|\chi_t|^2\dd t\right]\le TC_\chi ,
\]
and, by the second-moment estimate of Theorem \ref{thm:occupation} applied with $r=M_\Theta\lambda\le1$,
\[
\mathbb E^P\!\left[\left(\int_0^T\1_{\{\rho^{q_t}(X_{t-}^P,\mu_{t-}^P,\alpha_t)\le M_\Theta\lambda\}}\dd t\right)^2\right]\le C_{\rm occ}^{(2)}M_\Theta^2\lambda^2 .
\]
Substituting into \eqref{eq:Y-K-secondmoment},
\[
\mathbb E^P\!\left[\sup_{u\le T}|K_u^\lambda-K_u|^2\right]\le C_K^{\rm small}\lambda^2,
\quad
C_K^{\rm small}=2\alpha_A^2TC_\chi+
2C_2^2C_{\rm occ}^{(2)}M_\Theta^2.
\]

\emph{Step 3: the range $\lambda_0<\lambda\le1$.} If $\lambda_0<\lambda\le1$, the domination \eqref{eq:Y-minimal-domination} gives $|\chi_t^\lambda-\chi_t|\le|\chi_t^\lambda|+|\chi_t|\le2|\chi_t|$, and therefore directly
\[
\mathbb E^P\!\left[\sup_{u\le T}|K_u^\lambda-K_u|^2\right]
\le4T \mathbb E^P\!\left[\int_0^T|\chi_t|^2\dd t\right]
\le4TC_\chi\lambda_0^{-2}\lambda^2.
\]
Consequently, \eqref{eq:Y-rate-K-strong} holds for every $\lambda\in(0,1]$, with $C_K=\max(C_K^{\rm small},4TC_\chi\lambda_0^{-2})$.

\emph{Step 4: transfer to the backward solution.} It remains to transfer this estimate to the backward solution without reintroducing the $L^2(\dd t\otimes\dd P)$ norm of $\chi^\lambda-\chi$. In Lemma \ref{lem:energy}, take $\zeta^1=\zeta^2=g(X_T^P,\mu_T^P)$, $X^1=X^2=X^P$, $\mu^1=\mu^2=\mu^P$, and $\alpha^1=\alpha^2=\alpha$. Its constant $C_E$ does not depend on $\lambda$. We obtain
\begin{align*}
&\mathbb E^P\!\left[\sup_{t\le T}|Y_t^\lambda-Y_t|^2\right]\\
&\quad+\mathbb E^P\!\left[\int_0^T
\left(|Z_t^\lambda-Z_t|_{a_t^P}^2+
\|U_t^\lambda-U_t\|_\Pi^2\right)\dd t\right]
\le C_E C_K\lambda^2.
\end{align*}
Adding the $K$ component in the product norm defined in Section~\ref{sec:bienpose},
\[
\|\mathcal Y^{P,q,\lambda}-\mathcal Y^{P,q}\|_{\mathfrak E_P^2}^2
\le(1+C_E)C_K\lambda^2.
\]
Taking the supremum over $P,\gamma,q$ and then the square root proves \eqref{eq:Y-linear-occ}.
\end{proof}

\begin{remark}
The $O(\lambda)$ rate in norm relies on the second moment of the occupation time. To identify the gain precisely, if only the first-moment bound were available, squaring \eqref{eq:Y-point} would give
\[
\mathbb E^P\!\left[\int_0^T|\chi_t^\lambda-\chi_t|^2
\dd t\right]\le C(\lambda^2+\lambda).
\]
The inequality
\(\sup_{t\le T}|K_t^\lambda-K_t|^2
\le T\int_0^T|\chi_s^\lambda-\chi_s|^2\dd s\)
would then yield only a squared error of order $O(\lambda)$ and a bound of order $O(\lambda^{1/2})$ in norm for the backward solution. By contrast, the second-moment occupation estimate allows the source difference to be integrated \emph{before} it is squared, yielding the $O(\lambda^2)$ bound for the integrated source. Thus the $O(\lambda)$ rate is a consequence of the finite-branch geometry and \eqref{eq:occupation-second}, not a generic rate for maximal monotone operators.
\end{remark}

\subsection{Approximation of the cuts}

For $\lambda>0$, $P\in\PP$, data $(\xi,\alpha)\in\mathfrak D_P$, $\gamma\in(0,1]$, and $q\in\mathfrak A_{\gamma,M}^P[\xi,\alpha]$, denote by $\mathcal T_{P,\lambda}^{\xi,\alpha}(q)=\mathcal Y^{P,q,\lambda}$ the solution obtained by replacing $(A^q)^0$ with $A_\lambda^q$ in \eqref{eq:K}, and define the regularized cuts
\[
\cC_{\gamma,M,\lambda}^P[\xi,\alpha]
=
\bigl\{\mathcal T_{P,\lambda}^{\xi,\alpha}(q):q\in\mathfrak A_{\gamma,M}^P[\xi,\alpha]\bigr\}.
\]
When $P$ and the data are fixed, we write $\mathcal T_{P,\lambda}$ and $\cC_{\gamma,M,\lambda}^P$.

\begin{lemma}\label{lem:resolvent-slopes}
Let $\alpha_A>0$ and, for $k=1,\ldots,d$, let $\mathfrak g_k,\mathfrak g_k':\R\to\R$ be two continuous convex functions, affine on each connected component of $\R\setminus\Theta_k$ for the same finite set $\Theta_k$, with respective slopes $v_{k,j}$ and $v'_{k,j}$ on the $j$-th component, and suppose that $\max_j|v_{k,j}-v'_{k,j}|\le\delta_k$. Set $\phi(s)=\frac{\alpha_A}{2}|s|^2+\sum_k\mathfrak g_k(s_k)$ and $\phi'(s)=\frac{\alpha_A}{2}|s|^2+\sum_k\mathfrak g'_k(s_k)$. Then, for every $\lambda>0$ and $u,u'\in\R^d$,
\[
|J_{\lambda\partial\phi}(u)-J_{\lambda\partial\phi'}(u')|
\le|u-u'|+\lambda\Bigl(\sum_{k=1}^d\delta_k^2\Bigr)^{1/2}.
\]
\end{lemma}

\begin{proof}
Since the potential is separable, the resolvent acts coordinatewise: $J_{\lambda\partial\phi}(u)_k=j_k(u_k)$, where $j_k$ is the resolvent of $\lambda\partial(\frac{\alpha_A}{2}s^2+\mathfrak g_k)$ on $\R$; similarly for $j'_k$.

We first show that, for every $s\in\R$ and every $p'\in\partial \mathfrak g'_k(s)$, there exists $\tilde p\in\partial \mathfrak g_k(s)$ such that $|p'-\tilde p|\le\delta_k$. If $s$ is not a threshold and belongs to the $j$-th component, then $\partial \mathfrak g'_k(s)=\{v'_{k,j}\}$ and $\partial \mathfrak g_k(s)=\{v_{k,j}\}$. If $s=\theta_{k,j}$, then $p'=\tau v'_{k,j-1}+(1-\tau)v'_{k,j}$ for some $\tau\in[0,1]$, and $\tilde p=\tau v_{k,j-1}+(1-\tau)v_{k,j}$ works.

Fix $r\in\R$ and set $s=j_k(r)$, $s'=j'_k(r)$. By definition of the resolvent,
\[
p:=\frac{r-s}{\lambda}-\alpha_As\in\partial \mathfrak g_k(s),
\qquad
p':=\frac{r-s'}{\lambda}-\alpha_As'\in\partial \mathfrak g'_k(s'),
\]
so $p-p'=-(\lambda^{-1}+\alpha_A)(s-s')$. Choose $\tilde p\in\partial \mathfrak g_k(s')$ with $|p'-\tilde p|\le\delta_k$. Monotonicity of $\partial \mathfrak g_k$ gives
\[
0\le(p-\tilde p)(s-s')
=-(\lambda^{-1}+\alpha_A)|s-s'|^2+(p'-\tilde p)(s-s'),
\]
hence $|s-s'|\le\delta_k/(\lambda^{-1}+\alpha_A)\le\lambda\delta_k$. Thus $|J_{\lambda\partial\phi}(u')-J_{\lambda\partial\phi'}(u')|\le\lambda(\sum_k\delta_k^2)^{1/2}$. The nonexpansiveness of $J_{\lambda\partial\phi}$ \cite[chap.~II]{Brezis1973} and the triangle inequality conclude the proof.
\end{proof}

\begin{lemma}\label{lem:regularized-solution-continuity}
Fix $P\in\PP$, data $(\xi,\alpha)\in\mathfrak D_P$, $\gamma\in(0,1]$, and $\lambda>0$, and suppose that $(H_{\rm quant})$ holds. Let $q^n,q\in\mathfrak A_{\gamma,M}^P$ with $\|q^n-q\|_{\infty,P}\to0$. Then
\[
\mathcal T_{P,\lambda}(q^n)\longrightarrow\mathcal T_{P,\lambda}(q)
\qquad\text{in }\mathfrak E_P^2.
\]
\end{lemma}

\begin{proof}
Under Assumption \ref{ass:finite-branches}, $\phi^q$ has the form considered in Lemma \ref{lem:resolvent-slopes}, and the slopes of $\phi^q$ and $\phi^{q'}$ differ by at most $L_v|q-q'|$. Lemma \ref{lem:resolvent-slopes} gives
\[
|J_{\lambda\partial\phi^q}(u)-J_{\lambda\partial\phi^{q'}}(u')|
\le|u-u'|+\lambda\sqrt d\,L_v|q-q'|.
\]
Set $u^q=x+h^q(\mu,\alpha)-\lambda c^q(\mu,\alpha).$ Then
\[
|u^q-u^{q'}|
\le (L_h+\lambda L_c)|q-q'|.
\]
Identity \eqref{eq:yosida-translation} yields
\begin{align*}
|A_\lambda^q(x,\mu,\alpha)-A_\lambda^{q'}(x,\mu,\alpha)|
&\le \lambda^{-1}|u^q-u^{q'}|
+\lambda^{-1}|J_{\lambda\partial\phi^q}(u^q)-J_{\lambda\partial\phi^{q'}}(u^{q'})|
+L_c|q-q'|\\
&\le\left[\frac{2(L_h+\lambda L_c)}{\lambda}+\sqrt d\,L_v+L_c\right]|q-q'|
=:C_\lambda|q-q'|.
\end{align*}

Let $\varepsilon_n=\|q^n-q\|_{\infty,P}$. For every $t\in[0,T]$,
\begin{align*}
|K_t^{\lambda,q^n}-K_t^{\lambda,q}|
&\le\int_0^t
|A_\lambda^{q_s^n}(X_{s-}^P,\mu_{s-}^P,\alpha_s)
-A_\lambda^{q_s}(X_{s-}^P,\mu_{s-}^P,\alpha_s)|\dd s\\
&\le C_\lambda T\varepsilon_n.
\end{align*}
Consequently,
\[
\mathbb E^P\!\left[\sup_{t\le T}|K_t^{\lambda,q^n}-K_t^{\lambda,q}|^2\right]
\le C_\lambda^2T^2\varepsilon_n^2.
\]
The two systems have the same $X^P$, $\mu^P$, $\alpha$, and terminal condition. Lemma \ref{lem:energy} therefore gives
\[
\|\mathcal T_{P,\lambda}(q^n)-\mathcal T_{P,\lambda}(q)\|_{\mathfrak E_P^2}^2
\le (1+C_E)C_\lambda^2T^2\varepsilon_n^2
\longrightarrow0.
\]
\end{proof}

\begin{proposition}\label{prop:cutsY}
Fix $P\in\PP$, data $(\xi,\alpha)\in\mathfrak D_P$, and $\lambda>0$, and suppose that $(H_{\rm quant})$ holds. Then, for every $\gamma\in(0,1]$, the regularized cut $\cC_{\gamma,M,\lambda}^P$ is nonempty and compact in $\mathfrak E_P^2$, and $\cC_{\gamma_2,M,\lambda}^P\subseteq\cC_{\gamma_1,M,\lambda}^P$ for $0<\gamma_1\le\gamma_2\le1$.
\end{proposition}

\begin{proof}
Nonemptiness and nesting follow from those of $\mathfrak A_{\gamma,M}^P$ ((S1)--(S2)) by applying $\mathcal T_{P,\lambda}$. If $(q^n)\subset\mathfrak A_{\gamma,M}^P$, compactness of $\mathfrak A_{\gamma,M}^P$ yields a subsequence $q^{n_k}\to q\in\mathfrak A_{\gamma,M}^P$ in $\|\cdot\|_{\infty,P}$, and Lemma \ref{lem:regularized-solution-continuity} gives $\mathcal T_{P,\lambda}(q^{n_k})\to\mathcal T_{P,\lambda}(q)$ in $\mathfrak E_P^2$; compactness follows. The corresponding properties of $\cC_{\gamma,M}^P$ are Theorem \ref{thm:cutcompact} and the nesting property stated immediately after it.
\end{proof}

\begin{theorem}\label{thm:hausY}
Suppose that $(H_{\rm quant})$ holds. Then there exists $C=C(M;\mathfrak C_{\rm str})$ such that, for every $\lambda\in(0,1]$,
\begin{equation}\label{eq:Y-Haus-linear}
\sup_{P\in\PP}\sup_{(\xi,\alpha)\in\mathfrak D_P}\sup_{\gamma\in(0,1]}
d_H^{\mathfrak E_P^2}\bigl(\cC_{\gamma,M,\lambda}^P[\xi,\alpha],\cC_{\gamma,M}^P[\xi,\alpha]\bigr)
\le C\lambda.
\end{equation}
\end{theorem}

\begin{proof}
Fix $P\in\PP$, $(\xi,\alpha)\in\mathfrak D_P$, and $\gamma\in(0,1]$. For every
\[
\mathcal Y^{P,q,\lambda}
=
\mathcal T_{P,\lambda}(q)
\in\cC_{\gamma,M,\lambda}^P,
\]
the same parameter $q$ belongs to $\mathfrak A_{\gamma,M}^P$ and $\mathcal Y^{P,q}=\mathcal T_P(q)\in\cC_{\gamma,M}^P.$ Hence
\begin{align*}
\operatorname{dist}_{\mathfrak E_P^2}
(\mathcal Y^{P,q,\lambda},\cC_{\gamma,M}^P)
&\le
\|\mathcal Y^{P,q,\lambda}-\mathcal Y^{P,q}\|_{\mathfrak E_P^2}\\
&\le
C\lambda.
\end{align*}
Taking the supremum over $q$ gives
\[
e(\cC_{\gamma,M,\lambda}^P,\cC_{\gamma,M}^P)
\le
C\lambda.
\]
Conversely, for every $\mathcal Y^{P,q}\in\cC_{\gamma,M}^P$, the same $q$ produces $\mathcal Y^{P,q,\lambda}\in\cC_{\gamma,M,\lambda}^P$, and therefore
\[
e(\cC_{\gamma,M}^P,\cC_{\gamma,M,\lambda}^P)
\le
C\lambda.
\]
Since the Hausdorff distance is the maximum of the two excesses,
\[
d_H^{\mathfrak E_P^2}
(\cC_{\gamma,M,\lambda}^P,\cC_{\gamma,M}^P)
\le
C\lambda.
\]
Taking the suprema over $P$, $(\xi,\alpha)$, and $\gamma$ gives \eqref{eq:Y-Haus-linear}.
\end{proof}

\subsection{Joint stability in \texorpdfstring{$\lambda$}{lambda} and the fuzzy level}

\begin{corollary}\label{cor:Y-joint}
Suppose that $(H_{\rm quant})$ holds. There exists $C=C(M;\mathfrak C_{\rm str})$ such that, for every $P\in\PP$, $(\xi,\alpha)\in\mathfrak D_P$, $\lambda\in(0,1]$, and $\gamma,\gamma'\in(0,1]$,
\[
d_H^{\mathfrak E_P^2}\bigl(\cC_{\gamma,M,\lambda}^P[\xi,\alpha],\cC_{\gamma',M}^P[\xi,\alpha]\bigr)
\le C\left(\lambda+|\gamma-\gamma'|^\nu\right).
\]
\end{corollary}

\begin{proof}
Fix $P\in\PP$ and $(\xi,\alpha)\in\mathfrak D_P$, and omit $[\xi,\alpha]$ from the notation. The triangle inequality for the Hausdorff distance gives
\begin{align*}
d_H^{\mathfrak E_P^2}(\cC_{\gamma,M,\lambda}^P,\cC_{\gamma',M}^P)
&\le d_H^{\mathfrak E_P^2}(\cC_{\gamma,M,\lambda}^P,\cC_{\gamma,M}^P)
+d_H^{\mathfrak E_P^2}(\cC_{\gamma,M}^P,\cC_{\gamma',M}^P)\\
&\le C_Y\lambda+C_\gamma|\gamma-\gamma'|^\nu\\
&\le (C_Y+C_\gamma)\left(\lambda+|\gamma-\gamma'|^\nu\right),
\end{align*}
where $C_Y$ and $C_\gamma$ are the constants in \eqref{eq:Y-Haus-linear} and \eqref{eq:hausdorff-gamma}, respectively, used in the second line.
\end{proof}

\subsection{Optimality of the rate}

To show that the upper rate in Theorem \ref{thm:yosidarate} is sharp, we now give a model in the admissible class for which a matching lower bound holds uniformly over $\PP$. Here $n=1$, so $\underline a\le a_t^P\le\overline a$ are scalars.

\begin{theorem}\label{thm:sharp}
Suppose that $\PP$ satisfies Assumption \ref{ass:base} with $n=1$, and let $d=d_q=d_\alpha=1$, $\mathbb D=\{0\}$, $\Q=\{0\}$, $\Q_\gamma\equiv\{0\}$, $h\equiv0$, $c^q\equiv c\in\R$, and $\phi^q(s)=\kappa_+s^++\kappa_-s^-+\frac{\alpha_A}{2}s^2$ with $\kappa_\pm>0$ and $\alpha_A>0$. Let $b\equiv0$, $\beta\equiv0$, $\sigma\equiv\sigma_0>0$, $f\equiv0$, $g\equiv0$, and $\mathfrak D_P=\{(x_0,0)\}$ with $x_0\in\R$. Set $\overline\kappa=\max(\kappa_+,\kappa_-)$, $M_0=1+|c|+\overline\kappa$,
\[
C_\ell=\frac{2M_0(\overline aT)^{1/2}}{\sigma_0\underline a},\qquad
\overline\chi=\alpha_AM_0+\overline\kappa+|c|,\qquad
c_*=\frac{\alpha_AT(\alpha_Ax_0+c-\overline\kappa)}{1+\alpha_A}-4\overline\kappa C_\ell ,
\]
and suppose that $c_*>0$, which holds for $x_0$ sufficiently large since $C_\ell$ does not depend on $x_0$. Then $(H_{\rm quant})$ holds and, for every $\lambda\in(0,\lambda_*]$ with $\lambda_*=\min\{1,\overline\kappa/(\alpha_A\overline\chi)\}$,
\begin{equation}\label{eq:sharp}
c_*\lambda\le\inf_{P\in\PP}\|\mathcal Y^{P,\lambda}-\mathcal Y^P\|_{\mathfrak E_P^2}
\le\sup_{P\in\PP}\|\mathcal Y^{P,\lambda}-\mathcal Y^P\|_{\mathfrak E_P^2}\le C\lambda ,
\end{equation}
where $\mathcal Y^P$ and $\mathcal Y^{P,\lambda}$ are the selected and regularized solutions for the unique selection $q\equiv0$; the constant $C$ depends on $x_0$ through $\mathfrak M_\xi$, and both $c_*$ and $C$ are fixed once $x_0$ is chosen. Moreover, $\inf_{P\in\PP}\mathbb E^P[\sup_{t\le T}|K_t^\lambda-K_t|^2]\ge c_*^2\lambda^2$.
\end{theorem}

\begin{proof}
\emph{Assumptions.} Assumptions \ref{ass:F} and \ref{ass:B} are immediate; Assumption \ref{ass:fuzzycuts} holds with constant cuts, the parametric class reduced to $q\equiv0$ satisfies Assumption \ref{ass:paramclass} and is therefore a selection system by Proposition \ref{prop:paramclass}; Assumption \ref{ass:finitestat} holds with $\ell=1$, $\psi_1=0$, $\widehat h=0$, and $\dot\alpha=0$; Assumption \ref{ass:ALA} holds with constant slopes; Assumption \ref{ass:measop} follows from Remark \ref{rem:measop}(ii); and Assumption \ref{ass:elliptic} holds with $\underline\sigma^2=\sigma_0^2\underline a$. The upper bound in \eqref{eq:sharp} is Theorem \ref{thm:yosidarate}.

\emph{Step 1: reduction to the source.} Fix $P\in\PP$. Here $X_t=x_0+\sigma_0B_t^P$ is a continuous square-integrable martingale, $\chi_t=A^0(X_t)$ and $\chi_t^\lambda=A_\lambda(X_t)$, where $A=\partial\phi+c$. Since $f=g=0$, \eqref{eq:backward} becomes $Y_t=K_T-K_t-\int_t^TZ_s\dd B_s^P-\int_t^T\int_EU_s(e)\widetilde N^P(\dd s,\dd e)$, and the stochastic integrals are martingales; hence $Y_t=\mathbb E^P[K_T-K_t\mid\mathcal F_t^P]$, and likewise for the regularized solution. Consequently,
\[
\|\mathcal Y^{P,\lambda}-\mathcal Y^P\|_{\mathfrak E_P^2}^2\ge\mathbb E^P[|Y_0^\lambda-Y_0|^2]\ge\bigl(\mathbb E^P[Y_0-Y_0^\lambda]\bigr)^2
=\Bigl(\mathbb E^P\!\left[\int_0^T(\chi_t-\chi_t^\lambda)\dd t\right]\Bigr)^2 .
\]

\emph{Step 2: bulk and layer contributions.} Set $I_t=\1_{\{|X_t|\le M_0\lambda\}}$; here $\rho(x)=|x|$. On $\{I_t=0\}$, Lemma \ref{lem:samebranch}(a), whose constant $M_\Theta$ equals $M_0$ here, gives $\chi_t-\chi_t^\lambda=\alpha_A\lambda(1+\alpha_A\lambda)^{-1}\chi_t$. On $\{I_t=1\}$, Lemma \ref{lem:samebranch}(b) gives $|\chi_t-\chi_t^\lambda|\le\alpha_A\lambda|\chi_t|+2\overline\kappa$, and $|\chi_t|\le\alpha_AM_0\lambda+\overline\kappa+|c|\le\overline\chi$ since $\lambda\le1$. Thus, provided that $\mathbb E^P\!\left[\int_0^T\chi_t\dd t\right]\ge0$,
\[
\mathbb E^P\!\left[\int_0^T(\chi_t-\chi_t^\lambda)\dd t\right]
\ge\frac{\alpha_A\lambda}{1+\alpha_A}\,\mathbb E^P\!\left[\int_0^T\chi_t\dd t\right]-\bigl(2\alpha_A\lambda\overline\chi+2\overline\kappa\bigr)\,\mathbb E^P\!\left[\int_0^TI_t\dd t\right] .
\]

\emph{Step 3: the bulk term.} Here $A^0(x)=\alpha_Ax+\kappa(x)+c$, with $\kappa(x)=\kappa_+$ for $x>0$, $\kappa(x)=-\kappa_-$ for $x<0$, and $\kappa(0)=\proj_{[-\kappa_-,\kappa_+]}(-c)$. Since $\kappa(x)\ge-\kappa_-\ge-\overline\kappa$ for every $x$ and $\mathbb E^P[X_t]=x_0$,
\[
\mathbb E^P\!\left[\int_0^T\chi_t\dd t\right]=\mathbb E^P\!\left[\int_0^T\bigl(\alpha_AX_t+\kappa(X_t)+c\bigr)\dd t\right]\ge T(\alpha_Ax_0+c-\overline\kappa)>0,
\]
where positivity follows from $c_*>0$.

\emph{Step 4: the layer term.} The occupation density formula and $\dd\langle X\rangle_t=\sigma_0^2a_t^P\dd t\ge\sigma_0^2\underline a\dd t$ give $\sigma_0^2\underline a\int_0^TI_t\dd t\le\int_{-M_0\lambda}^{M_0\lambda}\mathsf L_T^y(X)\dd y$. Since $X$ is a continuous martingale, the Meyer--Tanaka formula gives $\mathbb E^P[\mathsf L_T^y(X)]=\mathbb E^P[|X_T-y|]-|x_0-y|\le\mathbb E^P[|X_T-x_0|]\le\sigma_0(\overline aT)^{1/2}$. Hence $\mathbb E^P\!\left[\int_0^TI_t\dd t\right]\le C_\ell\lambda$, with $C_\ell$ independent of $x_0$ and $P$.

\emph{Step 5: conclusion.} Steps 2--4 give
\[
\mathbb E^P\!\left[\int_0^T(\chi_t-\chi_t^\lambda)\dd t\right]\ge\lambda\Bigl[\frac{\alpha_AT(\alpha_Ax_0+c-\overline\kappa)}{1+\alpha_A}-2\overline\kappa C_\ell-2\alpha_A\overline\chi C_\ell\lambda\Bigr]\ge c_*\lambda,
\]
since $2\alpha_A\overline\chi\lambda\le2\overline\kappa$ for $\lambda\le\lambda_*$. Step 1 gives the lower bound in \eqref{eq:sharp}, uniformly in $P$. Finally, $\mathbb E^P[\sup_t|K_t^\lambda-K_t|^2]\ge(\mathbb E^P[K_T-K_T^\lambda])^2\ge c_*^2\lambda^2$.
\end{proof}

\begin{remark}
Outside the layer, the regularization contracts the source by the factor $(1+\alpha_A\lambda)^{-1}$ and creates an error of order $\lambda$; inside the layer, the error is bounded but its occupation time is of order $\lambda$. The condition $c_*>0$ prevents cancellation between these two contributions, and Theorem~\ref{thm:sharp} shows the optimality of \eqref{eq:Y-rate-K-strong}.
\end{remark}
\section{A verifiable example}\label{sec:example}

Take $d=d_\alpha=d_q=1$, $\mathbb D=\R$, $\Q=[-R,R]$, data and controls with common quadratic bounds, with absolutely continuous controls satisfying
$\sup_{P\in\PP}\sup_{(\xi,\alpha)\in\mathfrak D_P}\mathbb E^P[\int_0^T|\dot\alpha_t|^2\dd t]<\infty$, together with backward coefficients satisfying \ref{ass:B}. Set
\[
m(\mu)=\int_\R\psi(x)\mu(\dd x),\qquad \psi\in C_b^2(\R),
\]
\[
\mathfrak q(\mu,\alpha)=\mathfrak q_{\rm c}+\varepsilon_\mu\tanh(m(\mu))+\varepsilon_\alpha\tanh(\alpha),
\qquad
\Q_\gamma(\mu,\alpha)=[\mathfrak q(\mu,\alpha)-\varepsilon_Q(1-\gamma),\mathfrak q(\mu,\alpha)+\varepsilon_Q(1-\gamma)],
\]
with $\varepsilon_Q>0$ and
$|\mathfrak q_{\rm c}|+|\varepsilon_\mu|+|\varepsilon_\alpha|+\varepsilon_Q\le R$. Fix $L>0$ and
\[
\mathfrak T_L=\{\vartheta\in C([0,T]): |\vartheta_t|\le1,\ |\vartheta_t-\vartheta_s|\le L|t-s|\}.
\]
By Arzelà--Ascoli, $\mathfrak T_L$ is compact and stable under multiplication by $s\in[0,1]$. With $\mathfrak P=\mathfrak T_L$, define
\begin{equation}\label{eq:example-selection}
q_t^{\gamma,\vartheta}[\mu,\alpha]=\mathfrak q(\mu_t,\alpha_t)+\varepsilon_Q(1-\gamma)\vartheta_t.
\end{equation}
Then $q_t^{\gamma,\vartheta}[\mu,\alpha]\in\Q_\gamma(\mu_t,\alpha_t)$ and
\[
\|q^{\gamma,\vartheta}[\mu,\alpha]-q^{\gamma,\vartheta'}[\mu,\alpha]\|_{\infty,P}=\varepsilon_Q(1-\gamma)\|\vartheta-\vartheta'\|_\infty,
\qquad
\|q^{\gamma,\vartheta}[\mu,\alpha]-q^{\gamma',\vartheta}[\mu,\alpha]\|_{\infty,P}\le\varepsilon_Q|\gamma-\gamma'|.
\]
Since $\tanh$ is $1$-Lipschitz and $|\tanh'|\le1$,
\[
\Bigl|\frac{\mathrm d}{\mathrm dt}\mathfrak q(\mu_t^P,\alpha_t)\Bigr|\le|\varepsilon_\mu||\dot m^P(t)|+|\varepsilon_\alpha||\dot\alpha_t|,
\qquad m^P(t)=\mathbb E^P[\psi(X_t^P)],
\]
for almost every $t$. Since the $L^1$ bound of Lemma~\ref{lem:lawstatistics} does not control the energy of $\dot q$, choose the additive dynamics below. With $B_\psi=\|\psi'\|_\infty$, $D_\psi=\|\psi''\|_\infty$, and $B_\beta=\int_E|\beta_0(e)|^2\Pi(\dd e)$, \eqref{eq:law-stat-derivative} gives
\[
|\dot m^P(t)|\le B_\psi\mathbb E^P\!\left[|b_0(X_{t-}^P,m(\mu_t^P),\alpha_t)|\right]
+\tfrac12D_\psi(\sigma_0^2\overline a+B_\beta),
\]
and the linear growth of $b_0$, the bound $|m(\mu_t^P)|\le\|\psi\|_\infty$, and Jensen's inequality imply
\[
|\dot m^P(t)|^2\le C\left(1+\mathbb E^P[|X_{t-}^P|^2]+\mathbb E^P[|\alpha_t|^2]\right).
\]
By \eqref{eq:forward-apriori},
\[
C_m:=\sup_{P\in\PP}\sup_{(\xi,\alpha)\in\mathfrak D_P}\int_0^T|\dot m^P(t)|^2\dd t<\infty.
\]
Moreover,
\[
|\dot q_t^{\gamma,\vartheta}|^2
\le3|\varepsilon_\mu|^2|\dot m^P(t)|^2+
3|\varepsilon_\alpha|^2|\dot\alpha_t|^2+3\varepsilon_Q^2L^2,
\]
so, with $C_{\dot\alpha}=\sup_{P}\sup_{(\xi,\alpha)\in\mathfrak D_P}\mathbb E^P[\int_0^T|\dot\alpha_t|^2\dd t]$,
\[
M^2=3|\varepsilon_\mu|^2C_m+3|\varepsilon_\alpha|^2C_{\dot\alpha}+3\varepsilon_Q^2L^2T.
\]
This bound holds along the flows $\mu^P[\xi,\alpha]$. If $\gamma_1\le\gamma_2$ and $\gamma_1<1$, then
\[
\vartheta'=\frac{1-\gamma_2}{1-\gamma_1}\vartheta\in\mathfrak T_L,
\]
and $q^{\gamma_2,\vartheta}[\mu,\alpha]=q^{\gamma_1,\vartheta'}[\mu,\alpha]$; the case $\gamma_1=1$ is trivial. Thus \ref{ass:paramclass} holds with $\nu=1$. For the same $\vartheta$ and two pairs $(\mu,\alpha)$, $(\mu',\alpha')$,
\[
|q_t^{\gamma,\vartheta}[\mu,\alpha]-q_t^{\gamma',\vartheta}[\mu',\alpha']|
\le |\varepsilon_\mu|\|\psi'\|_\infty W_2(\mu_t,\mu'_t)+|\varepsilon_\alpha||\alpha_t-\alpha'_t|+\varepsilon_Q|\gamma-\gamma'|,
\]
which verifies the cross-coupling condition (vi). These selections are continuous and adapted, hence predictable, and absolutely continuous; continuity of $\mu^P$ allows one to use either $\mu_t$ or $\mu_{t-}$ in \ref{ass:paramclass}.

Choose
\[
h^q(\mu,\alpha)=q+\ell_\mu m(\mu)+\ell_\alpha \alpha,
\qquad c^q(\mu,\alpha)=c_0+c_1\tanh(q),
\]
\[
\phi^q(s)=\kappa_+^q s^++\kappa_-^q s^-+\frac{\alpha_A}{2}s^2,
\qquad
\kappa_\pm^q=\underline\kappa_\pm+\varepsilon_\pm\tanh(q),
\]
with $\alpha_A>0$ and $|\varepsilon_\pm|<\underline\kappa_\pm$, so that
$\underline\kappa=\min_\pm(\underline\kappa_\pm-|\varepsilon_\pm|)>0$ and
$\overline\kappa=\max_\pm(\underline\kappa_\pm+|\varepsilon_\pm|)$. Finally, take
\[
\dd X_t^P=b_0(X_t^P,m(\mu_t^P),\alpha_t)\dd t+\sigma_0\dd B_t^P+\int_E\beta_0(e)\widetilde N^P(\dd t,\dd e),
\]
where $\sigma_0\ne0$, $b_0$ is Lipschitz, and $\int_E|\beta_0(e)|^2\Pi(\dd e)<\infty$.

The Hausdorff distance between $[l,u]$ and $[l',u']$ equals $\max(|l-l'|,|u-u'|)$, hence
\[
d_H(\Q_\gamma(\mu,\alpha),\Q_{\gamma'}(\mu',\alpha'))
\le |\mathfrak q(\mu,\alpha)-\mathfrak q(\mu',\alpha')|+\varepsilon_Q|\gamma-\gamma'|.
\]
Together with $|\mathfrak q(\mu,\alpha)-\mathfrak q(\mu',\alpha')|\le|\varepsilon_\mu|\|\psi'\|_\infty W_2(\mu,\mu')+|\varepsilon_\alpha||\alpha-\alpha'|$, this shows that the cuts are compact, nonempty, nested, and continuous, hence measurable in $(\mu,\alpha)$. Moreover,
$|(\partial\phi^q)^0(z)|\le\alpha_A|z|+\overline\kappa$; $h^q,c^q$ satisfy the required bounds, $(q,z)\mapsto\phi^q(z)$ is continuous and convex in $z$, the resolvents are continuous, and the minimal selection is Borel measurable branch by branch. Thus \ref{ass:operator} and \ref{ass:measop} hold, while \ref{ass:finitestat} follows from $\widehat h(q,\alpha,m)=q+\ell_\mu m+\ell_\alpha\alpha$.

A concrete random control is
\[
\alpha_t^P=\alpha_0+\int_0^t\tanh(X_{s-}^P)\dd s,
\qquad P\text{-a.s., for every }t\in[0,T],
\]
with deterministic $\alpha_0$ and, for instance, $b_0(x,m,\alpha)=b_xx+b_mm+b_\alpha\alpha$. The augmented system $(X^P,\alpha^P)$ is Lipschitz since $|m(\mu)-m(\mu')|\le\|\psi'\|_\infty W_2(\mu,\mu')$; Theorem~\ref{thm:forward} applies to it, after which $\alpha$ is regarded as exogenous in the triangular system. Moreover,
\[
|\alpha_t^P|\le|\alpha_0|+T,
\qquad
|\dot\alpha_t^P|\le1\quad\dd t\text{-a.e.},
\]
so $C_{\dot\alpha}\le T$. This control is continuous, adapted, predictable, and belongs to $\mathcal A_2^P$; its bounds are uniform in $P$. The continuous martingale part of $X+h^q(\mu,\alpha)$ remains $\sigma_0B^P$, so $\underline\sigma^2=\sigma_0^2\underline a$ for $n=1$. A statistic with genuine law dependence is obtained, for example, by taking $\psi(x)=\tanh(x)$ and $\ell_\mu\ne0$.

Finally, one may choose
\begin{align*}
f(t,x,y,z,u,\mu,\eta,\alpha)
&=f_xx+f_yy+f_\alpha\alpha+f_\mu\int v\,\mu(\dd v)+f_\eta\int v\,\eta(\dd v),\\
g(x,\mu)&=g_xx+g_\mu\int v\,\mu(\dd v),
\end{align*}
with constant real coefficients. Since
$|\int v\,\mu(\dd v)-\int v\,\mu'(\dd v)|\le W_2(\mu,\mu')$, Assumption~\ref{ass:B} holds uniformly over every family sharing the preceding constants and bounds.
\section{Conclusion}\label{sec:conclusion}

The central difficulty is the jump of the selected monotone source when the forward state crosses an interface that moves with the selection, the law, and the control. Controlling how long the state remains near that interface makes the discontinuity quantitatively manageable. For a fixed selection the system is well posed, while compact selection classes generate compact, nested, left-continuous solution cuts. The second occupation moment is the key estimate behind stability with respect to the selection, the fuzzy level, and the data; for separable finite-branch potentials it also leads to the sharp $O(\lambda)$ Yosida rate uniformly over $\PP$. We further construct intrinsic selection systems for center--radius cuts in arbitrary dimension and, in the scalar case when (S4) is available, for convex feedback cuts.

The two uncertainty layers remain separate throughout the analysis: $q$ represents fuzzy uncertainty, whereas $P\in\PP$ represents model uncertainty. We do not replace either layer by a deterministic representative. Several natural extensions therefore remain outside the scope of the paper. A decision-theoretic treatment, after propagation of chaos, would require robust control or Pontryagin-type arguments; quasi-sure aggregation in the spirit of second-order BSDEs, quantitative Hausdorff stability for general convex cuts when $d_q\ge2$, and fuzzy particle approximations would require additional ideas as well.

\section*{Author contributions (CRediT)}
\textbf{Kayembe Tshiswaka Tcheick}: Conceptualization. \textbf{Mabela Matendo Rostin}: Formal analysis. \textbf{Kabengele Mpunga Yannick}: Writing. \textbf{Bosonga Bofeki Jean-Pierre}: Supervision.

\section*{Disclosure statement}
The authors report there are no competing interests to declare.

\section*{Funding}
This research received no specific grant from funding agencies in the public, commercial, or not-for-profit sectors.

\section*{Data availability statement}
This is a theoretical study and does not rely on any data set.

\end{document}